\documentclass[12pt,reqno]{amsart} 
\usepackage{mathtools}
\usepackage[dvipsnames]{xcolor}
\usepackage{hyperref}
\usepackage{dsfont}
\usepackage{enumerate}
\usepackage{dsfont}
\usepackage[pdftex]{graphicx}
\usepackage{caption}
\usepackage{subcaption}

\usepackage{tikz}
\usetikzlibrary{calc,decorations.pathreplacing,positioning}

\usepackage{verbatim}

\newtheorem{theorem}{Theorem}[section]
\newtheorem{definition}[theorem]{Definition}

\newtheorem{lemma}[theorem]{Lemma}
\newtheorem{corollary}[theorem]{Corollary}

\newtheorem{remark}[theorem]{Remark}
\newtheorem{assumption}[theorem]{Assumption}
\newtheorem{convention}[theorem]{Convention}
\newtheorem*{proposition*}{Proposition}
\newtheorem*{theorem*}{Theorem}

\newcommand{\R}{\mathbb{R}} 
\newcommand{\Z}{\mathbb{Z}}
\newcommand{\T}{\mathcal{T}}
\newcommand{\N}{\mathbb{N}}

\DeclareMathOperator{\ent}{ent}
\DeclareMathOperator{\Ent}{Ent}
\DeclareMathOperator{\E}{E}

\DeclareMathOperator{\argmax}{argmax}
\DeclareMathOperator{\tree}{\mathcal{T}}

\DeclareMathOperator{\geo}{\mathbf{g}}

\mathtoolsset{showonlyrefs}

\begin{document}

\title{A Variational principle for a non-integrable model}

\author{Georg Menz}
\address{Department of Mathematics, University of California, Los Angeles}
\email{gmenz@math.ucla.edu}

\author{Martin Tassy}
\address{Department of Mathematics, University of California, Los Angeles}
\email{mtassy@math.ucla.edu}


\date{February 6 2017}

\begin{abstract}
We show the existence of a variational principle for graph homomorphisms from $\Z^m$ to a $d$-regular tree. The technique is based on a discrete Kirszbraun theorem and a concentration inequality obtained through the dynamics of the model. As another consequence of the concentration inequality we also obtain the existence of a continuum of translation-invariant ergodic gradient Gibbs measures for graph homomorphisms from $\Z^m$ to a regular tree. The method is sufficiently robust such that it could be applied to other discrete models with a quite general target graphs.
\end{abstract}

\maketitle

\tableofcontents



\section{Introduction}

The appearance of limit shapes as a limiting behavior of discrete systems is a well-known and studied phenomenon in statistical physics and combinatorics (e.g.~\cite{Georgii}). Among others, models that exhibits limits shapes are domino tilings and dimer models (e.g.~\cite{Kas63,CEP96,CKP01}), polymer models, random tiling models in particular lozenge tilings (e.g.~\cite{Des98,LRS01,Wil04}), Gibbs models (e.g.~\cite{She05}), the Ising model (e.g.~\cite{DKS92,Cer06}), asymmetric exclusion processes (e.g.~\cite{FS06}), random matrices (see e.g.~\cite{Wig59,KieSpo99}), sandpile models (e.g.~\cite{LP08}), the six vertex model (e.g.~\cite{BCG2016,CoSp16,ReSr16}), Young tableaux (e.g.~\cite{LS77,VK77,PR07}).\\

Limit shapes appear in stiff systems whenever a boundary condition forces a certain response of the system. The main tool to explain limit shapes is a variational principle. The variational principle asymptotically characterizes the number of microscopic states i.e.~the microscopic entropy~$\Ent_n$, via a variational problem. This means that for large system sizes~$n$, the entropy of the system is given by maximizing a macroscopic entropy~$\E(f)$ over all admissible limiting profiles~$f \in \mathcal{A}$. The boundary conditions are incorporated in the admissibility condition. In formulas, the variational principle can be expressed as (see for example Theorem~\ref{p_main_result_variational_principle} below)
\begin{align}\label{e_general_variational_principle}
 \Ent_n \approx \inf_{f \in \mathcal A} \E(f), 
\end{align}
where the macroscopic entropy~$$\E(f)= \int \ent(\nabla f(x)) dx$$ can be calculated via a local quantity~$\ent (\nabla f(x))$. This local quantity is called local surface tension in this article. \\

Often, a consequence of a variational principle is that the uniform measure on the microscopic configurations concentrates around configurations that are close to the minimizer of the variational problem (see comments before Theorem~\ref{p_profile_theorem} below). This is related to the appearance of limit shapes on large scales. \\

In analogy to classical probability theory, one can understand the variational principle as an elaborated version of the law of large numbers. On large scales, the behavior of the system is determined by a deterministic quantity, namely the minimizer~$f$ of the macroscopic entropy. Hence, deriving a variational principle is often the first step in analyzing discrete models, before one attempts to study other questions like the fluctuations of the model. \\

A lot of inspiration for this article comes from the variational principle for domino tilings~\cite{CKP01} (see Figure~\ref{f_aztec_domino_tilings}). It is one of the fundamental results for studying domino tilings and the other integrable discrete models. A detailed analysis of the limit shapes for domino tilings was given in \cite{kenyon2007limit}. Recently a new constructive approach was developed in~\cite{CoSp16} for the determination of the Arctic curve (the frozen boundary of
the limit shape). The approach is discussed mainly in the framework
of the six vertex model, which is integrable. However, the method seems to be very robust. It is an interesting open question if this method can also be used to determine the Arctic curve in a non-integrable model.\\

\begin{figure}
 \centering
 \begin{subfigure}[b]{0.65\textwidth}
 \includegraphics[width=\textwidth]{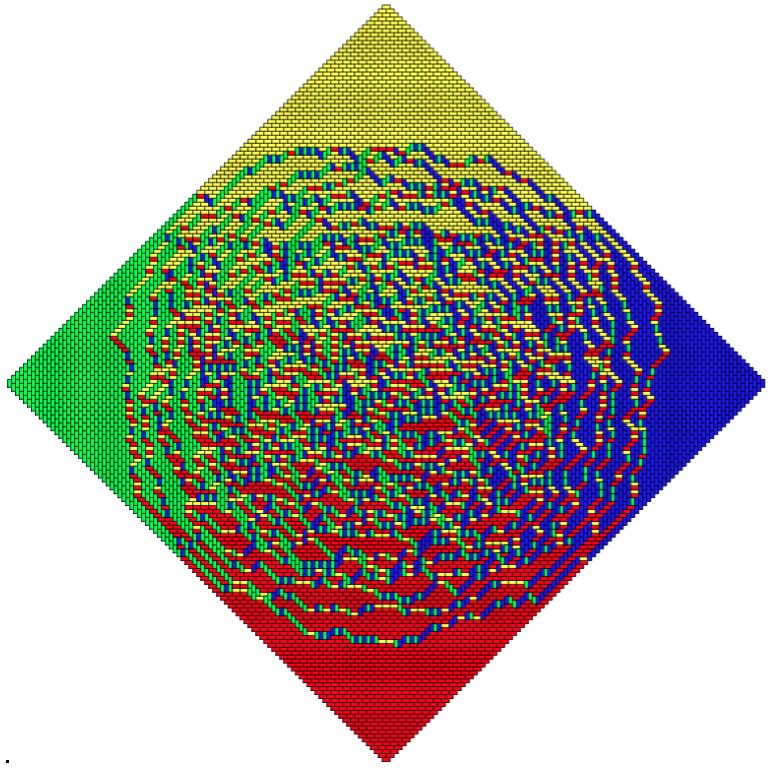}

\newpage
 
\caption{ An Aztec diamond for domino tilings. The combinatorics of the model is similar to Lipschitz functions from $\Z^2$ to $\Z$. (see~\cite{CKP01})}\label{f_aztec_domino_tilings}
 
 \end{subfigure}

 \begin{subfigure}[b]{0.65\textwidth}
 \includegraphics[width=\textwidth]{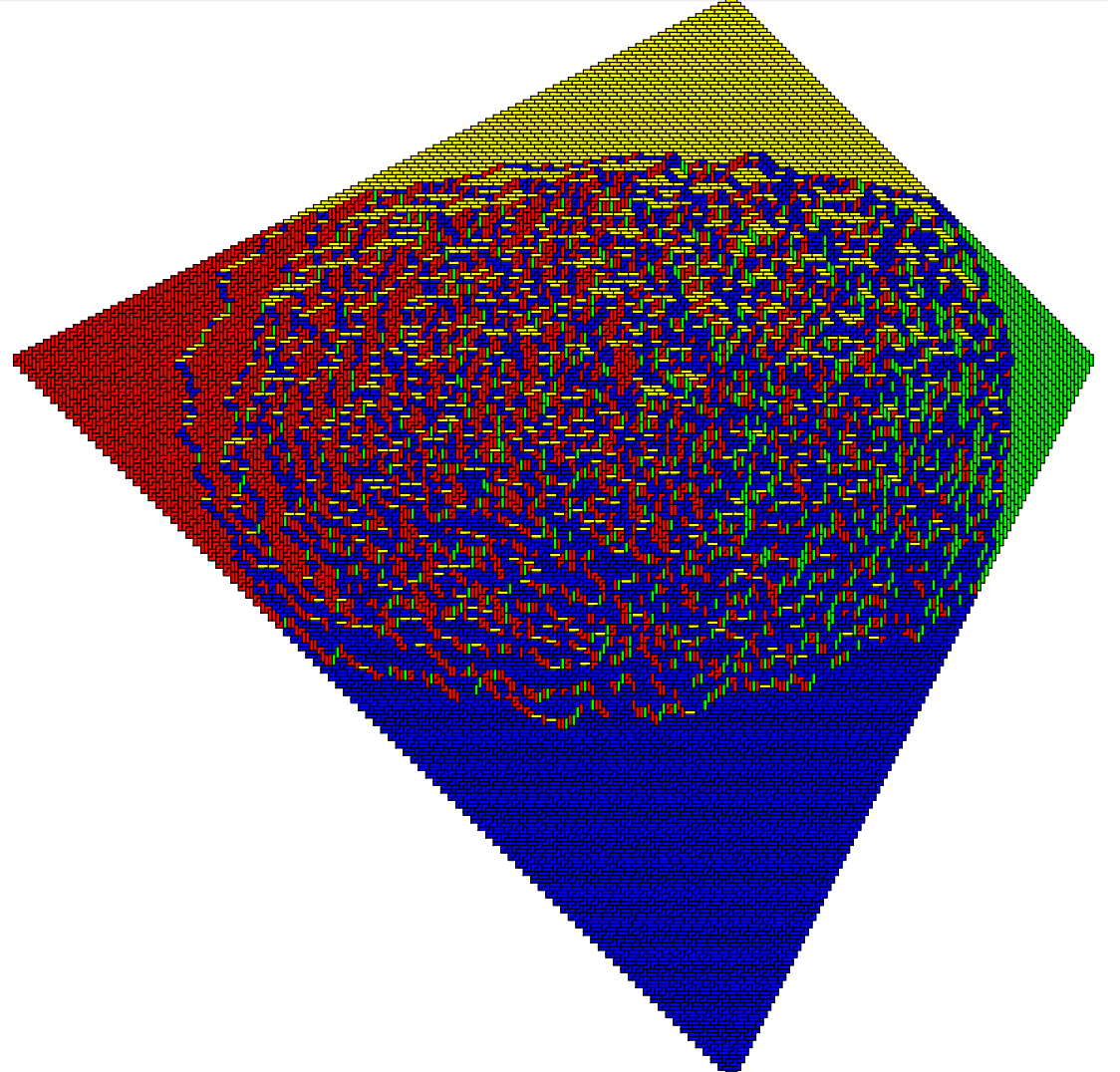}
 \caption{ An Aztec diamond for ribbon tilings. The combinatorics of the model is similar to Lipschitz functions from $\Z^2$ to $\Z^2$ (see \cite{sheffield2002ribbon}).}\label{f_ribbon}
 
 \end{subfigure}
 
 \end{figure}
 
 \begin{figure} 
\ContinuedFloat 
 
 \begin{subfigure}[b]{0.65\textwidth}
 \includegraphics[width=\textwidth]{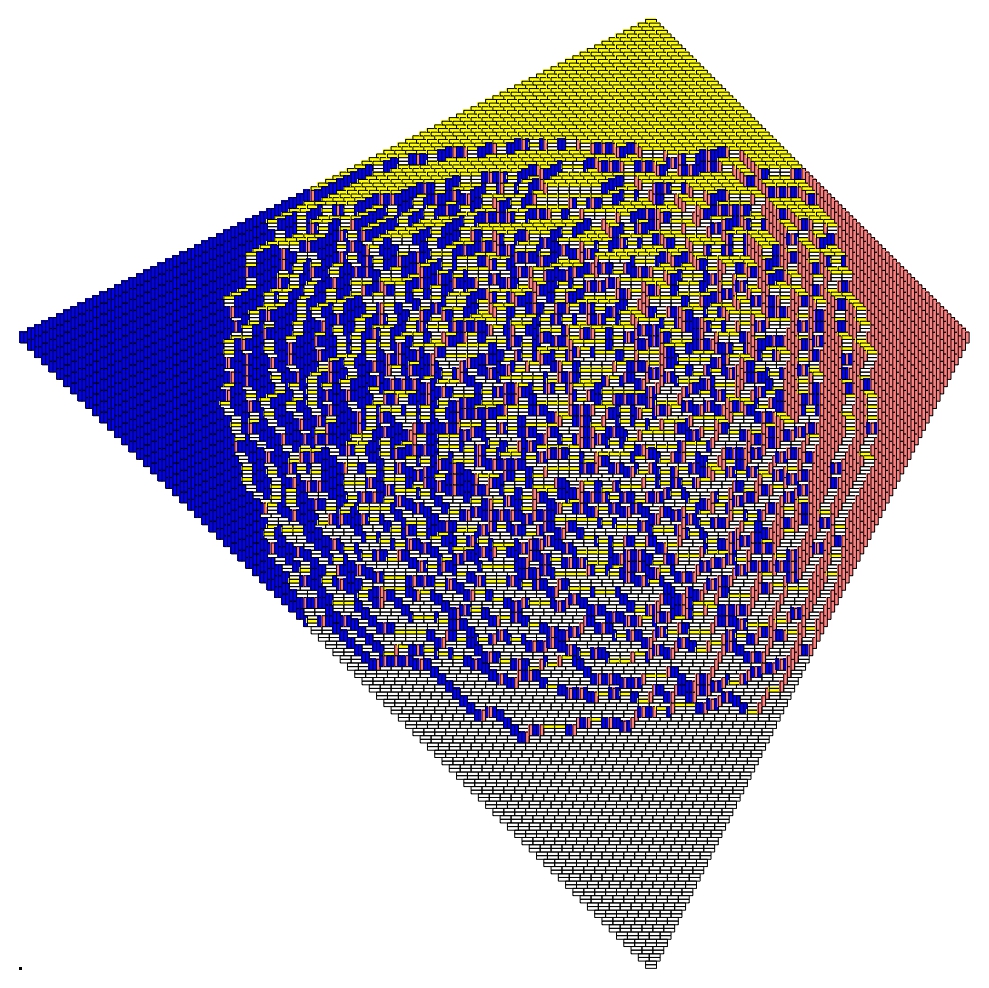}
 \caption{ An Aztec diamond tiling by $3 \times 1$ bars. The combinatorics of the model is similar to Lipschitz functions from $\Z^2$ to $\Z_3 \ast \Z_3$ (see \cite{kenyon1992tiling}).} \label{f_aztec_3_1_bars} 
 \end{subfigure} 
\begin{subfigure}[b]{0.65\textwidth}
 \includegraphics[width=\textwidth]{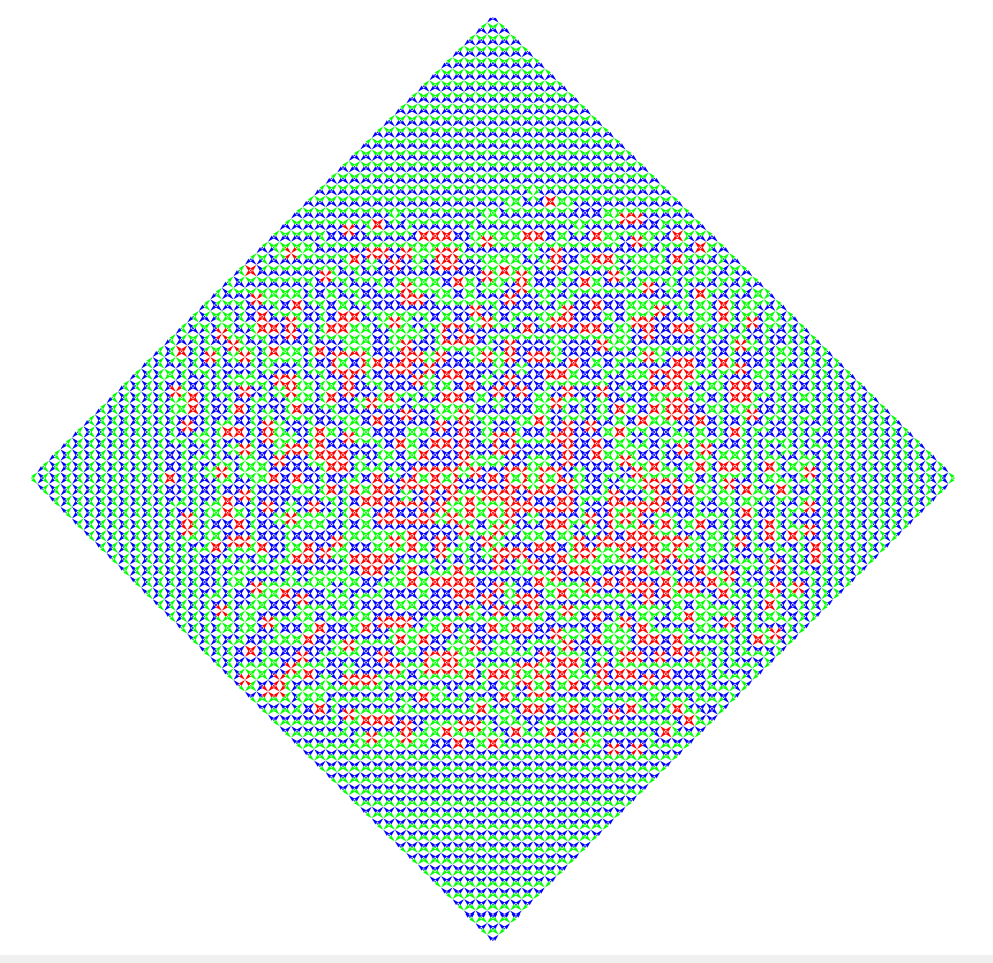}
 \caption{ An Aztec diamond for Graph homomorphisms in a $3$-regular tree. Each color represents one of the $\alpha_i$'s introduced in Section 3.}\label{f_aztec_graph_homo} 
 \end{subfigure}

\end{figure}

So far, all the tools that were developed to study variational principles of discrete models heavily rely on the intermediate value theorem. Specifically, two $\Z$ valued functions moving at different speed in $\Z^d$ must meet at some points of the lattice, whose points or clusters provide sets where the value of the two functions can be exchanged. This allows a general technique called cluster swapping (see \cite{She05}) and its variants. 
Up to the knowledge of the authors, there is no non-trivial example of a variational principle for which the target space of the underlying model does not have this property. However, simulations (see Figure~\ref{f_ribbon}, Figure~\ref{f_aztec_3_1_bars} and~Figure~\ref{f_aztec_graph_homo}) show that limit shapes also appear for discrete maps taking values in a large class of discrete target spaces. The purpose of this article is to go beyond $\Z$-valued models and to find out what properties of a discrete system lead to variational principles and limit shapes. As a general guideline, we strive to develop a robust method for deducing variational principles for graph homomorphisms that is purely based on general principles of statistical mechanics and not on particular methods for certain classes of graphs, similar to what was developed in~\cite{KieSpo99} for the limiting bulk distribution of eigenvalues of random matrices.\\

We are interested in graph homomorphisms because they provide a natural framework to study systems with hard constraints (see~\cite{BrWi00}) and because they are closely related to other classical model in statistical physics: the loop $O(n)$ model (see \cite{duminil2017exponential}) and tilings by bars (see \cite{tassythesis}). In this article, we consider the non-integrable model of graph homomorphisms from $\Z^m$ to a $d$-regular tree~$\mathcal{T}$. We want to point out the fact that in our variational principle the underlying lattice can have arbitrary dimension~$m \geq 2$. In the special case of~$m=2$ and~$\mathcal{T}= \mathbb{Z}$, our model is equivalent to the six-vertex model with uniform weights (cf.~\cite{Henk77,ChPeShTa18}). \\


We identified two properties that a model of discrete maps needs to satisfy in order to have a variational principle. The first one is a stability property that allows to glue together pieces of different configurations at a negligible entropic cost. As a consequence, perturbing the boundary condition on a microscopic scale does not change the macroscopic properties of the model. The second one is a concentration property.\\

In the case of discrete $\Z$-valued models, such as height functions of domino tilings or the antiferromagnetic Potts model, both properties can be deduced naturally. For deducing the first property, one uses that the space of configurations is a lattice. Then it is possible to quickly attach two configurations together provided that the boundary conditions are similar. This is done by using the minimum of two well-chosen extensions of those configurations. The second property, namely the concentration, is tackled by using a cluster swapping argument or an analog version of this argument for other systems (see e.g.~\cite{CEP96,She05}).\\

We want to emphasize again that those arguments are based on properties which rely on the intermediate value theorem in the target space $\Z$ and are not available for general graph homomorphisms. One of the main contributions of this article is that we provide alternative methods. These methods are not based on integrability but on weaker properties of graphs and of the underlying dynamics of the model. The authors believe that the principles behind those new arguments are robust. They should provide a possible line of attack to study variational principles and limiting behavior of a large class of discrete models.\\

Now, let us discuss how the two necessary properties, namely stability and concentration, are deduced without relying on integrability. The first property is obtained by using the discrete version of a well-known theorem for continuous metric spaces: the Kirszbraun theorem (see~Theorem~\ref{Kirszbraun} below). This theorem states that any graph homomorphism from a subset $S$ of $\Z^d$ to a regular tree $\mathcal{T}$, which is contracting for the graph distances in $\Z^d$ and $\mathcal{T}$, can be extended to the whole space $\Z^d$. Up to knowledge of the authors, the first version of a discrete Kirszbraun theorem was developed in the setting of tilings in \cite{tassythesis} and \cite{PST}. Using a new version of the Kirszbraun theorem for graph homomorphisms allows us to show that microscopic variations of the boundary conditions can be neglected on the macroscopic scale and thus do not influence the entropy of a system.\\

The second property, namely the concentration, is provided in~Theorem~\ref{main_concentration_section_3}. The concentration is deduced via a combination of a classical concentration inequality, namely the Azuma-Hoeffding inequality, with a coupling technique relying on dynamical properties of the model. 
Inspiration for this type of argument comes from~\cite{CEP96}, where the Azuma-Hoeffding inequality was used to show concentration for domino tilings. However, it is very difficult to apply directly the Azuma-Hoeffding inequality for more sophisticated models. This would need detailed information about the structure of the underlying space of configurations. Our dynamical approach circumvents this obstacle. One only has to understand the response of the system to changing the value of one point.\\

The natural process for our dynamical approach is the Glauber dynamics (see \cite{chandgotia2015entropyminimal} for details). This choice would be fine for studying graph homomorphisms to~$\mathbb{Z}$. However, using the Glauber dynamics does not work for the more complicated model of graph homomorphisms to a tree~$\mathcal{T}$. Heuristically, this can be understood from the observation that two simple random walks on a tree tend to diverge. We overcome this technical obstacle by two adaptations. On the one hand, we add to the original Glauber dynamics an extra non-local resampling step. On the other hand, we introduce a suitable quantity called depth. It turns out the modified Glauber dynamics conserves this quantity. This allows us to apply the Azuma-Hoeffding inequality and deduce concentration in depth. We then show that concentration in depth is sufficient for deducing our variational principle. \\

The main ingredients for the proof of a variational principle (see~\eqref{e_general_variational_principle} or Theorem~\ref{p_main_result_variational_principle}) are the equivalence of the different notions of the local surface tension, that the local surface tension is bounded from below, and that it is convex. For the equivalence, let us have a closer look at the definition of the local surface tension~$\ent(s)$. As usual, the local surface tension~$\ent(s)$ is defined as the limit of microscopic entropies~$\Ent_n$ on a box~$S_n$ with suitable chosen boundary conditions~(see Section~\ref{s_local_surface_tension}) i.e. 
\begin{align}\label{e_def_local_surface_tension_introduction}
 \ent(s) = \lim_{n \to \infty} \Ent_n.
\end{align}
The freedom of choosing different boundary conditions leads to a-priori different notions of the local surface tension. In particular, choosing the boundary condition to be a fixed state~$h_s$ with slope~$s$ leads to the notion of surface tension~$\ent^{\tt fixed}$ which is associated to fixed boundary conditions. Allowing the boundary condition to be any state~$g$ that is close to~$h_s$ leads to the notion of surface tension~$\ent^{\tt fluct}$ which is associated to fluctuating boundary conditions. Allowing periodic boundary conditions with slope~$s$ leads to the notion of surface tension~$\ent^{\tt per}$ which is associated to fluctuating boundary conditions. \\

A crucial ingredient for deducing a variational principle~\eqref{e_general_variational_principle} is that those a-priori distinct notions of the local surface tension are equivalent i.e.
\begin{align}
 \label{e_equivalence_surface_tensions}
\ent^{\tt fixed}= \ent^{\tt fluct}= \ent^{\tt per}.
\end{align}
This equivalence is deduced in this article in Theorem~\ref{geodesic_square}. The argument relies on the two basic properties outlined above, i.e.~the Kirszbraun theorem (cf.~Theorem~\ref{Kirszbraun}) and the concentration inequality (cf.~Theorem~\ref{main_concentration_section_3}). The existence of the limit~$\ent^{\tt fixed}$ is usually deduced via sub-additivity arguments. For example, in~\cite{FunSpo97} sub-additivity was exploited to show the existence of the local surface tension the continuous~$\nabla \phi$ model. In~\cite{MeKrTa17} it is shown that sub-additivity combined with the Kirszbraun theorem can be used to show the existence of the local surface tension in homogenization of random surfaces. However, those arguments do not carry over to periodic boundary conditions. Those argument cannot be used to show the a-priori existence of the local surface tension~$\ent^{\tt per}$ that is associated to periodic boundary conditions. In this article, we show the existence of the~$\ent^{\tt per}$ via a self-contained argument based on a combination of the Kirszbraun theorem and the concentration inequality (cf.~Theorem~\ref{p_existence_local_surface_tension} below).\\

In the same spirit, we also give a self-contained argument that the local surface tension~$\ent^{\tt per}$ is convex (see Theorem~\ref{p_convexity_local_surface_tension}). From convexity it follows that the variational problem given by our variational principle has a minimizer (cf.~Theorem~\ref{p_main_result_variational_principle} below). However, we do not prove that this minimizing limiting profile is unique. The uniqueness of the minimizer would follow if the local surface tension is strictly convex. We do not know if the local surface tension is strictly convex but we conjecture it. \\

Beside the equivalence~\eqref {e_equivalence_surface_tensions} and the convexity of~$\ent(s)$ (see Theorem~\ref{p_convexity_local_surface_tension}, Theorem~\ref{geodesic_square} and Lemma~\ref{p_entropy_depth_condition_local_surface_tension}), the only additional ingredients for the proof of the variational principle are very general principles, namely the compactness of Lipschitz functions on a bounded region and that Lipschitz functions can be very well approximated by piecewise affine functions on a simplicial complex (see Section~\ref{s_proof_variational_principle}).\\

Compared to~$\mathbb{Z}$, there are infinitely many ways to travel to infinity in a tree~$\mathcal{T}$. Those pathways to infinity are described by geodesics. Limit shapes are sensitive to the choice of geodesics on which the graph homomorphism travels on the boundary (see Figure~\ref{f_triangle} for an illustration). This adds another technical difficulty when deducing the variational principle for graph homomorphisms to a tree. The problem is to define the scaling limit of a graph homomorphism. For domino tilings the height function is an integer valued map, which allows a natural notion of a scaling limit. If the space of geodesics is more complex, as it is the case for trees, the notion of a scaling limit is less obvious. In order to define the scaling limit of a graph homomorphism to a tree one has to additionally keep track of the information on which geodesic the graph homomorphism is traveling on. This leads to a more subtle definition of the limiting profile which involves several compatibility conditions (see Definition~\ref{d_asymp_height_profile}). Another consequence is that the set~$\mathcal{A}$ of admissible limiting profiles~$f$, over which the continuous entropy~$\E(f)$ is minimized, has a more elaborated structure involving additional constraints.\\

In this article, we also discuss another consequence of the Kirszbraun theorem and the concentration inequality. It is the existence of a continuum of shift-invariant ergodic gradient Gibbs measures on tree-valued graph homomorphisms on~$\mathbb{Z}^m$ (see Section~\ref{s_existence_gradient_gibbs} and Theorem~\ref{ergodic}). This is also the reason why we choose periodic boundary conditions for defining the local surface tension~$\ent(s)$. This choice allows us to obtain measures which are translation-invariant. A key point for proving the existence of ergodic gradient Gibbs measures is also that the concentration inequality of Theorem~\ref{main_concentration_section_3} applies to those translation-invariant measures. Although the statement of Theorem~\ref{ergodic} does not imply the strict convexity of the local surface tension, we believe that it is a very strong indication that the surface tension must be strictly convex since it shows that different boundary conditions will lead to different phases.\\

\begin{figure}
\includegraphics[width=0.75\textwidth]{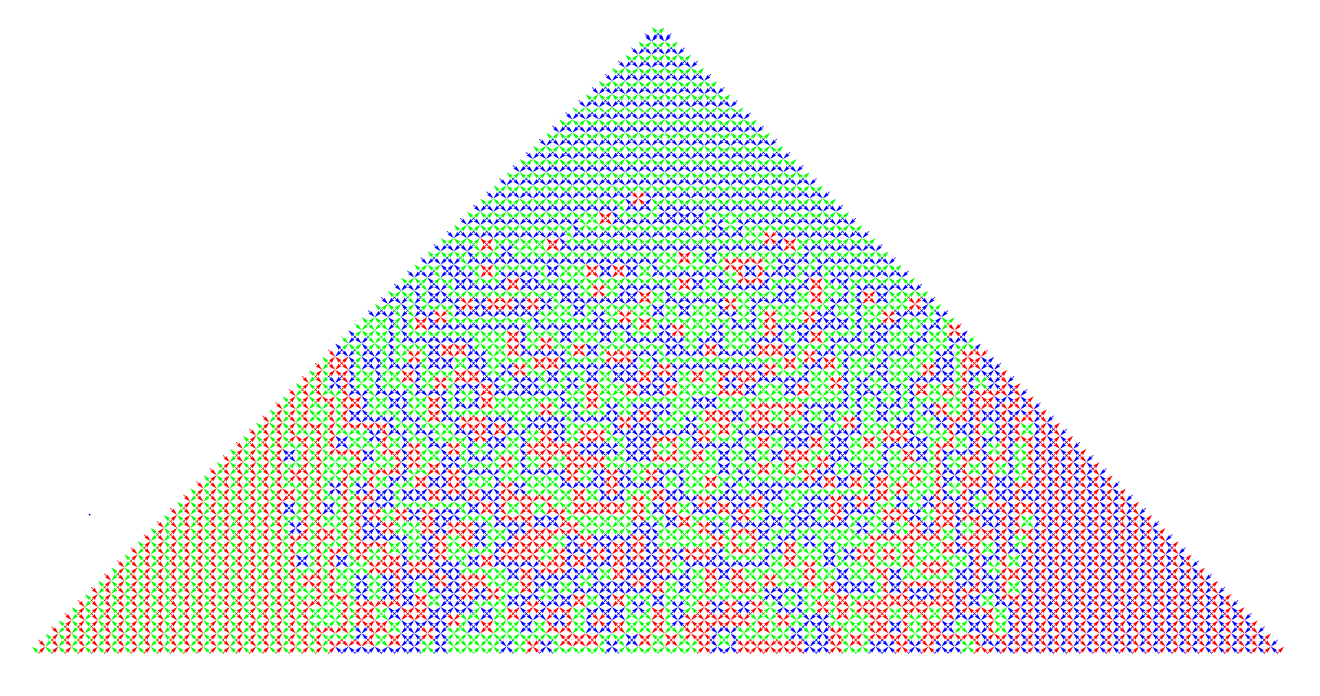}\caption{Illustration of the limit shape of a graph homomorphism to a $3$-regular tree. The boundary travels on three different geodesics.} \label{f_triangle}
\end{figure}

\subsection*{Overview over the article}
In Section \ref{s_main_result}, we outline the precise setting, introduce relevant definitions and state the main result of this article: the variational principle for tree-valued graph homomorphisms (see Theorem \ref{p_main_result_variational_principle}). In Section \ref{s_local_surface_tension} we discuss the local surface tension~$\ent(s)$ associated to the slope~$s$. We show the existence of~$\ent(s)$ in Theorem~\ref{p_existence_local_surface_tension} and the convexity of~$\ent(s)$ in Theorem~\ref{p_convexity_local_surface_tension}. Section~\ref{s_technical_observations} is dedicated to proving technical lemmas which are central to our main results. Namely, in Theorem \ref{Kirszbraun} we prove the Kirszbraun theorem and in Theorem~\ref{main_concentration_section_3} we deduce the concentration inequality. This concentration inequality is used in Section~\ref{s_existence_gradient_gibbs} to prove the existence of an infinite volume gradient Gibbs measure for each possible slope (see Theorem~\ref{ergodic}). Finally, in Section~\ref{s_proof_variational_principle} we give the proof of the variational principle itself.


{\center \textsc{Notation}}

\begin{itemize}

\item $C$ and~$c$ denote generic positive bounded universal constants.
\item If~$A$ is finite then $|A|$ denotes the cardinality of the set~$A$. If~$A \subset \mathbb{R}^m$ then $|A|$ denotes the Lebesgue measure of the set~$A$.
\item $x,y,z$ denote elements~$x,y,z \in \mathbb{Z}^m$.
\item $S_n : = \left\{ 0, \ldots, n-1\right\}^m$.
\item $\vec{i}_k$ denotes the~$k$-th Euclidean basis vector.
\item $x \sim y$ indicates that the points $x$ and $y$ are neighbors.
\item $e_{xy}$ is the oriented edge from $x$ to~$y$.
\item $d_G$ distance in the graph $G$.

\item $\mathcal{T}$ denotes a $d$-regular tree.
\item $w,v$ denote elements ~$w,v \in \mathcal{T}$.
\item $\mathbf{r} \in \mathcal{T}$ denotes the root of the tree~$\mathcal{T}$.
\item $\mathbf{g} \subset \mathcal{T}$ denotes a geodesic of the tree~$\mathcal{T}$.
\item $\Pi_{\mathbf{g}}: \mathcal{T} \to \mathbf{g}$ is the projection on the geodesic~$\mathbf{g}$.
\item $\infty_{\mathbf{g}}$ boundary point of the directed geodesic~$\mathbf{g}\subset \mathcal{T}$.
\item $\partial \mathcal{T}:= \left\{ \infty_{\mathbf{g}} : \mathbf{g} \mbox{ is a geodesic of } \mathcal{T} \right\}$.
\item $\theta (\varepsilon)$ denotes a generic smooth function with~$\lim_{\varepsilon \to 0} \theta (\varepsilon) =0$.
\item $\alpha_i$ denotes colors of edges.
\item $(\alpha_1, \ldots , \alpha_d)$ denotes the generating set of a~$d$-regular graph.
\item $h: \mathbb{Z}^m \to \mathcal{T}$ is a graph homomorphism.
\item $\mathcal{H}_{n}^{g,free } (s)$ is the set of all graph homomorphisms $h: \mathbb{Z}^m \to \mathcal{T}$ that are~$n-$invariant with slope~$s$ and supported on the geodesic~$\mathbf{g}$. 
\item $\mathcal{H}_{n}^{g} (s)$ is the set of all graph homomorphisms $h: \mathbb{Z}^m \to \mathcal{T}$ that are~$n-$invariant with slope~$s$ and supported on the geodesic~$\mathbf{g}$ such that $\Pi_{\mathbf{g}}(h(0))$ is pinned at a fixed point. 
\item $\ent_{n}(s) =- \frac{1}{n^2} \ln |\mathcal{H}_{n}^{g} (s)|.$
\item $\ent^{\tt per}(s)= \lim_{n \to \infty} \ent_n(s).$
\item Since $\Z^m$ and $\tree$ are both bipartite let us fix a $2$-coloring of the two graphs. The color of a vertex is called parity.
 \item $u \cdot v $ denotes the usual inner product of $\R^m$.
\end{itemize}


\section{The variational principle for graph homomorphisms}\label{s_main_result}

Let us start with clarifying the underlying model. We closely follow the setting of~\cite{CKP01} with the main difference that we allow the underlying lattice to be~$m$ dimensional and that we consider graph homomorphism to a regular tree. We equip the~$m$-dimensional lattice~$\mathbb{Z}^m$ and~$\mathbb{R}^m$ with the~$\ell_1$-norm.

\begin{assumption}\label{a_convergence_of_regions}
For each~$n \in \mathbb{N}$ we consider a bounded lattice region~$R_n \subset \mathbb{Z}^m$ (i.e.~a connected region composed of squares from the unit square lattice). We assume that for~$n \to \infty$ the scaled sub-lattice~$\frac{1}{n} R_n$ converges in the Hausdorff distance to a bounded, simply connected, Lipschitz domain~$R \subset \mathbb{R}^m$.
\end{assumption}
The basic objective is to study graph homomorphisms~$ h: R_n \to \mathcal{T}$, where~$\mathcal{T}$ denotes a $d-$regular tree.
\begin{definition}(Graph homomorphism, height function)
 Let~$\mathcal{T}$ denote the d-regular rooted tree and let~$\Lambda \subset \mathbb{Z}^m$ be a finite set. We denote with~$d_{\mathcal{G}}$ the natural graph distance on a graph~$\mathcal{G}$. A function~$h: \Lambda \to \mathcal{T}$ is called graph homomorphism, if 
 \begin{align*}
 d_{\mathcal{T}}(h(x), h(y))=1 
 \end{align*}
for all~$k, l \in \Lambda$ with~$|x-y|_{\ell_1}=1$. In analogy to~\cite{CKP01}, we may also call~$h$ a $\mathcal{T}$-valued height function. Let~$\partial \Lambda$ denote the inner boundary of~$\Lambda \subset \mathbb{Z}^m$ i.e.
\begin{align}
 \partial \Lambda = \left\{ x \in \Lambda \ | \ \exists y \notin \Lambda : |x-y|_{\ell_1}=1 \right\}.
\end{align}
 We call a homomorphism~$h: \partial \Lambda \to \mathcal{T}$ boundary graph homomorphism or boundary height function.
\end{definition}

We want to study the question of how many~$\mathcal{T}-$valued height functions exist that extend a fixed prescribed boundary height function~$h_{\partial R_n}: \partial R_n \to \mathcal{T}$. Hence, let us consider the set~$M(R_n, h_{\partial R_n})$ that is defined as
\begin{align}\label{e_d_micro_partition_function}
 \begin{split}
M(R_n, h_{\partial R_n}) = \left\{ h: R_n \to \mathcal{T} \ | \ h \mbox{ is graph homomorphism,} \right. \\
 \qquad \qquad \left. \mbox{ and } h(\sigma)= h_{\partial R_n}(\sigma) \quad \forall \sigma \in \partial R_n \right\}. 
 \end{split}
\end{align}
The goal of the article is to derive an asymptotic formula as~$n \to \infty$ of the \emph{microscopic entropy}
\begin{equation}
\label{e_d_microscopic_entropy}
\Ent\left( R_n, h_{\partial R_n} \right) := - \frac{1}{|R_n|} \ln |M(R_n, h_{\partial R_n})|.
\end{equation}
For this purpose, let us introduce the notion of an asymptotic height profile and the notion of an asymptotic boundary height profile. Those two objects will serve as the possible limits of sequences of graph homomorphisms~$h_{R_n}: R_n \to \mathcal{T}$ and boundary graph homomorphisms~$h_{\partial R_n} : \partial R_n \to \mathcal{T}$.

\begin{definition}[Asymptotic height profile]\label{d_asymp_height_profile}
Let $k\in \N$, let $f_R: R \to \mathbb{R}^+ \times \{1,..,k\}$ be a function and let $(a_{i,j})_{k \times k}$ be a set of non-negative real numbers satisfying the following compatibility conditions 
\begin{align}\label{e_a_ij_symmetry}
 a_{i,j} = a_{j,i} \quad \mbox{and} \quad a_{i,i}=\infty
\end{align}
and (cf.~Figure~\ref{f_geodesic_meeting_points}) 
\begin{align}\label{e_geodesics_compatibility}
 a_{i,j} < a_{i,k} \Rightarrow a_{j,k}= a_{i,j}
\end{align}
for all~$i,j \in \left\{ 1, \ldots, k \right\}$. We say that~$(f_R, (a_{i,j})_{k \times k} )$ is an asymptotic height profile if:
\begin{itemize}
\item The first coordinate of the map~$f_{R}$ is 1-Lipschitz with respect to the~$\ell_1$-norm in $R$, i.e.~for all~$x,y \in R$: 
\begin{align}
\left| f_{R}^1 (x) - f_{R}^1 (y) \right| \leq |x-y|_{\ell_1}. \label{e_asy_hf_lipschitz}
\end{align}
\item The map~$f_{R}$ is $(a_{i,j})_{k \times k}$-admissible in the sense that for all $i\neq j$:
\begin{align}\label{e_asyhf_admissible}
 \overline{f_R^{-1}(\mathbb{R}^+,i)} \cap \overline{f_R^{-1}(\mathbb{R}^+,j)} \subset \overline{ f_R^{-1}( [0,a_{i,j}] ,\{1,..,k \})}.
\end{align} 
\end{itemize}
\end{definition}

\begin{figure}
\includegraphics[width=0.5\textwidth]{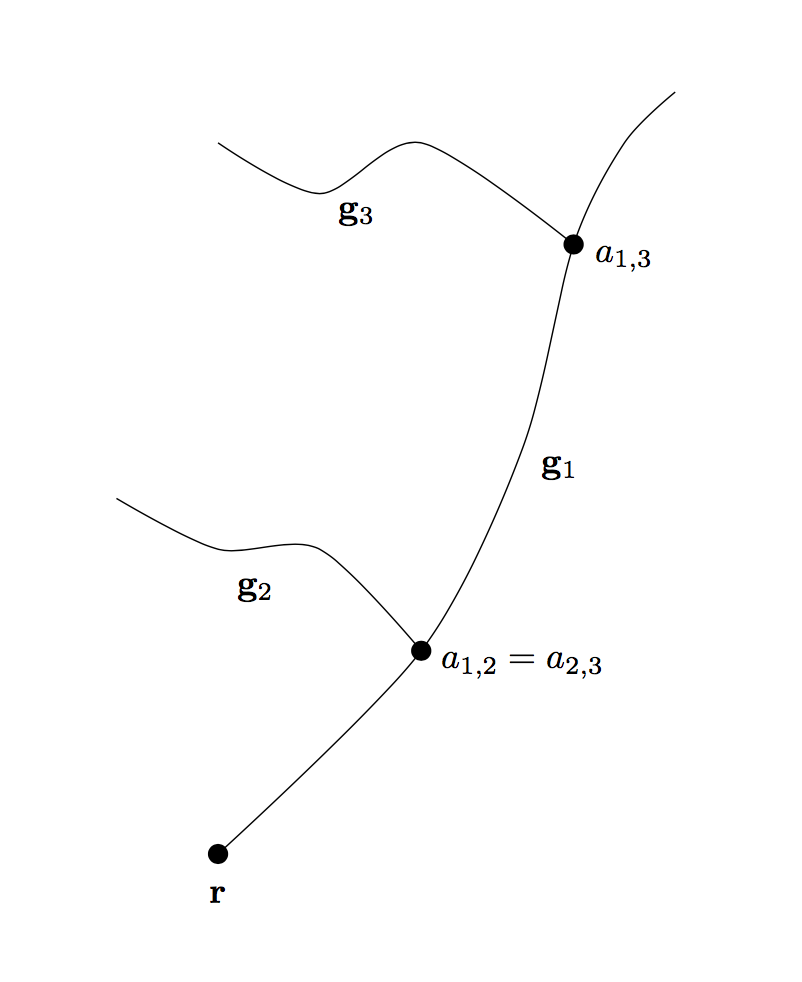}\caption{Illustration of the compatibility condition~\eqref{e_geodesics_compatibility}}\label{f_geodesic_meeting_points}
\end{figure}

Compared to a classical asymptotic height function~$f_R: R \to \mathbb{R}$ (see for example~\cite{CKP01}) our notion of an asymptotic height profile~$(f_R,(a_{i,j})_{k \times k})$ has two coordinates~$f_R^1$ and~$f_R^2$. Opposed to~$\mathbb{R}$, there are infinitely many ways to travel from zero to infinity on a tree~$\mathcal{T}$. Those pathways to infinity are described by directed geodesics~$\mathbf{g}$ starting at the root~$\mathbf{r}$.

\begin{definition}(One- and two-sided geodesics on the tree~$\mathcal{T}$)\label{d_geodesic}
 Let~$\mathcal{T}$ denote the~$d$-regular tree with root~$\mathbf{r}$. A graph homomorphism~$\mathbf{g}: \mathbb{N}\to \mathcal{T}$ is called one-sided (or directed) geodesic if the map~$\mathbf{g}$ is one-to-one. Two one-sided geodesics~$\mathbf{g}_1$ and~$\mathbf{g}_2$ are asymptotic if there are~$n_1,n_2 \in \mathbb{N}$ such that
 \begin{align*}
 \mathbf{g}_1(n_1+k)=\mathbf{g}_2(n_2+k) \qquad \mbox{for all~$k \in \mathbb{N}$.}
 \end{align*}
We say that a one-sided geodesic~$\mathbf{g}$ starts in~$\mathbf{g}(0)$. A graph homomorphism~$\mathbf{g}: \mathbb{Z}\to \mathcal{T}$ is called two-sided geodesic if the map~$\mathbf{g}$ is one-to-one.
 
\end{definition}
\begin{convention}\label{e_d_point_on_geodesic}
If~$\mathbf{g}$ is a two-sided geodesic on the graph~$\mathcal{T}$ containing the root~$\mathbf{r} \in \mathcal{T}$, we identify~$\mathbf g$ with a one-to-one graph homomorphism~$\mathbf{g}: \mathbb{Z} \to \tree$ such that~$\mathbf{g}(0) = \mathbf{r}$.
\end{convention}

We assume that the asymptotic boundary height profile will travel only on finitely many one-sided geodesics~$\mathbf{g}_i$ that are indexed by~$1, \ldots, k$ and start in the root~$\mathbf{r}$. This requirement is natural since graph homomorphisms are $1$-Lipschitz. Hence the limiting profile on the boundary of $R$ must have finite total variation.\\ 

The second coordinate~$f_R^2(x)=i$ of an asymptotic height profile~$f$ indicates on which one-sided geodesic~$\mathbf{g}_i$ the point~$x\in R$ is mapped to. More precisely, the point~$x \in R$ will be mapped onto a point on the one-sided geodesic~$\mathbf{g}_{f_R^2(x)}$. The first coordinate~$f_R^1(x)$ of the asymptotic boundary profile~$f_R$ specifies the exact location on the one-sided geodesic~$\mathbf{g}_{f_R^2(x)}$. This means that~$x\in R$ will be mapped onto a point on the one-sided geodesic~$\mathbf{g}_{f_R^2(x)}$ that has distance~$f_R^1(x)$ from the root~$\mathbf{r}$. Working with one-sided geodesics allows to assume that~$f_R^1(x) \in \mathbb{R}^+$ is non negative. The Lipschitz condition~$\eqref{e_asy_hf_lipschitz}$ on~$f_R^1$ is very natural and follows from the fact that graph homomorphisms are~$1-$Lipschitz. This is very similar to the setting of classical height functions (see~\cite{CKP01}).\\

Let us now describe the meaning of the numbers~$a_{i,j}$ appearing in the compatibility condition~\eqref{e_geodesics_compatibility} and the condition~\eqref{e_asyhf_admissible}. The numbers~$a_{i,j}$ have their origin in the following observation. Any two one-sided geodesics $\mathbf{g}_1\subset \mathcal{T}$ and~$\mathbf{g}_2 \subset \mathcal{T}$ starting at~$\mathbf{r} \in \mathcal{T}$ have a nonzero intersection~$\mathbf{g}_1 \cap \mathbf{g}_2 \neq \emptyset$. However, they must split up at some vertex~$v_{1,2} \in \mathcal{T}$ (see also discussion below and Figure~\ref{f_scaling_graph_homo}). If seen from the root~$\mathbf{r}$, the vertex~$v_{1,2} \in \mathcal{T}$ can be interpreted as the splitting point of the one-sided geodesics~$\mathbf{g}_1$ and~$\mathbf{g}_2$. If seen from infinity, the vertex~$v_{1,2} \in \mathcal{T}$ can be interpreted as the meeting point of the one-sided geodesics~$\mathbf{g}_1$ and~$\mathbf{g}_2$. The number~$a_{1,2}$ denotes the asymptotic height of this meeting point (see also~\eqref{e_asymptotic_characterization_meeting_point} in Definition~\ref{d_convergence_boundary_height_function}). When traveling on a geodesic from infinity, it is only possible to change to the other geodesic by passing through the meeting point~$v_{1,2}$. The admissibility condition~\eqref{e_asyhf_admissible} enforces that the asymptotic height profile has a similar property. Using this interpretation of the number~$a_{i,j}$ it also becomes clear why the compatibility condition~\eqref{e_geodesics_compatibility} is needed (see also Figure~\ref{f_geodesic_meeting_points}). This interpretation is made precise in the following lemma.
\begin{lemma}\label{p_equivalence_path_property}
We consider a map~$f_R: R \to \mathbb{R}^+ \times \left\{ 1, \ldots k \right\}$. We assume that the first coordinate~$f^1$ is continuous and that the numbers~$a_{i,j}$ satisfy the conditions~\eqref{e_a_ij_symmetry} and~\eqref{e_geodesics_compatibility}. Then it is equivalent:
\begin{itemize}
\item The map~$f$ satisfies the condition~\eqref{e_asyhf_admissible}.
\item For any two points~$x,y\in R$ and any path~$\mathbf{p} \subset R$ that connects~$x$ and~$y$ there is a point~$z \in \mathbf{p}$ such that
\begin{align}\label{e_path_property}
f_R^1(z) \leq a_{f^2(x),f^2(y)}.
\end{align}
\end{itemize}
\end{lemma}
The statement of Lemma~\ref{p_equivalence_path_property} is not used later in the article. The proof is straight-forward and is therefore omitted in the article. Let us consider an example and assume that there are points~$x,y\in R$ such that the second coordinate~$f_R^2(x)=1$ and~$f_R^2(y)=2$. This indicates that the asymptotic height function~$f_R^1$ travels at~$x$ on the one-sided geodesic~$\mathbf{g}_1$ and at~$y$ on the one-sided geodesic~$\mathbf{g}_2$. Now, let us consider a path~$\mathbf{p}\subset R$ from~$x$ to~$y$. Then the asymptotic height function has to change geodesics on that path. The admissibility condition~\eqref{e_asyhf_admissible} enforces that changing from the one-sided geodesic~$\mathbf{g}_1$ to the one-sided geodesic~$\mathbf{g}_2$ can only take place below the meeting point, which is characterized by the height~$a_{1,2}$. We also want to note that if~$f_R^1(z)\leq a_{1,2}$ then the second coordinate~$f_R^2(z)$ is allowed to change freely between~$1$ and~$2$ with no regularity condition. Below the height~$a_{1,2}$ the geodesics~$\mathbf{g}_1$ and~$\mathbf{g}_2$ are indistinguishable.\\

\begin{figure}



 \includegraphics[width=0.4\textwidth]{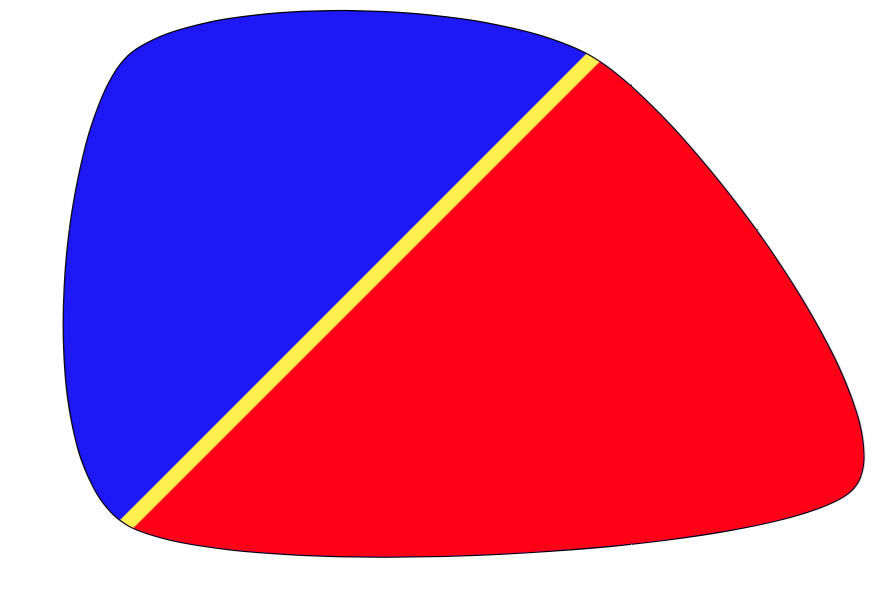}
\caption{Illustration of an asymptotic height profile~$(f_R, a_{1,2})$.}\label{f_AHP}
\end{figure}

For an illustration of an asymptotic height profile~$(f_R, a_{1,2})$ we refer to Figure~\ref{f_AHP}. On the blue region $R_{\tt{blue}} \subset R$, the height profile travels on the one-sided geodesic~$\mathbf{g}_1$. On the red region $R_{\tt{red}} \subset R$, the height profile travels on the one-sided geodesic~$\mathbf{g}_2$. Mathematically, this means that the second coordinate of~$f_R$ satisfies
\begin{align}
f_R^2(x) = 
\begin{cases}
1 , & \mbox{ if } x \in R_{\tt{blue}}, \\
2 , & \mbox{ if } x \in R_{\tt{red}}.
\end{cases} 
\end{align}
 
The yellow line~$L_{\tt{yellow}}$ separates the blue region~$R_{\tt{blue}}$ and the red region~$R_{\tt{red}}$. The admissibility condition~\eqref{e_asyhf_admissible} means that one can only cross from $R_{\tt{blue}}$ to $R_{\tt{red}}$ below the meeting point of~$\mathbf{g}_1$ and~$\mathbf{g}_2$. Hence, the first coordinate of~$f_R$ satisfies for all~$x \in L_{\tt{yellow}}$
 \begin{align}
f_R^1(x) \leq a_{1,2} .
\end{align}\mbox{}\\[-1ex]

In a variational principle only the boundary condition is prescribed. For that reason, we now adapt Definition~\ref{d_asymp_height_profile} and define the notion of an asymptotic boundary height profile.

\begin{definition}[Asymptotic boundary height profile]\label{d_asymptotic_boundary_height_profile}
Let $k \in \mathbb{N}$, let $f_{\partial R}: \partial R \to \mathbb{R}^+ \times \{1,..,k\}$ be a function and let $(a_{i,j})_{k \times k}$ be a set of non-negative real numbers satisfying the condition~\eqref{e_a_ij_symmetry} and~\eqref{e_geodesics_compatibility}. We say that~$(f_{\partial R}, (a_{i,j})_{k \times k} )$ is an asymptotic boundary height profile if it satisfies the conditions~\eqref{e_asy_hf_lipschitz} and~\eqref{e_asyhf_admissible} from above and for all $x,y \in \partial R$ it holds
\begin{align}\label{e_extenability_asym_bhf}
|f_{\partial R}^1 (x)-a_{f_{\partial R}^2(x),f_{\partial R}^2(y)}|_{\ell_1}+|a_{f_{\partial R}^2(x),f_{\partial R}^2(y)}-f_{\partial R}^1(y)|_{\ell_1} \leq |x-y|_{\ell_1}.
\end{align}
\end{definition}
Compared to the Definition~\ref{d_asymp_height_profile} of an asymptotic height profile, the condition~\eqref{e_extenability_asym_bhf} is new. It is needed to guarantee that every asymptotic boundary height profile~$f_{\partial R}$ can be extended to an asymptotic boundary height function.
\begin{lemma}\label{p_extension_of_bhf}
Let~$(f_{\partial R},(a_{i,j})_{k \times k})$ be an asymptotic boundary height function in the sense of Definition~\ref{d_asymptotic_boundary_height_profile}. Then it can be extended to an asymptotic height profile~$(f_{\partial R},(a_{i,j})_{k \times k})$ on the full region~$R$.
\end{lemma}

The proof of Lemma~\ref{p_extension_of_bhf} is stated in Section~\ref{s_proof_variational_principle}. Lemma~\ref{p_extension_of_bhf} is important because otherwise the statement of the variational principle, formulated in Theorem~\ref{p_main_result_variational_principle} below, could be empty.\\

The next step towards the variational principle is to define in which sense a sequence of (boundary) graph homomorphisms~$f_{\partial R_n}: \partial R_n \to \mathcal{T}$ converges to an asymptotic height profile~$\left( f_{\partial R}, (a_{i,j})_{k \times k}\right)$. For this purpose let us introduce some necessary definitions.

\begin{definition}(Boundary points~$\infty_{\mathbf{g}}$ (cf.~\cite{Kloe08}))\label{d_boundary_point}
A boundary point~$\infty$ is an equivalence class of asymptotic one-sided geodesics. We denote with $\infty_{\mathbf{g}}$ the boundary point associated to an one-sided geodesic~$\mathbf{g}$ i.e.~$\mathbf{g} \in \infty_{\mathbf{g}}$. We denote with~$\partial \mathcal{T}$ the set of all boundary points. We observe that each boundary point~$\infty$ has a canonical representative~$\mathbf{g} \in \infty$ such that~$\mathbf{g}$ is a one-sided geodesic starting at the root~$\mathbf{r}$ i.e.~$\mathbf{g}(0)=\mathbf{r}$. Because a two-sided geodesic~$\mathbf{g}$ is the union of two one-sided geodesics~$\mathbf{g}_1$ and~$\mathbf{g}_2$, the two-sided geodesic~$\mathbf{g}$ has two boundary points which we denote with~$\infty_{\mathbf{g}} = \infty_{\mathbf{g}_1}$ and~$-\infty_{\mathbf{g}} = \infty_{\mathbf{g}_2}$. 
\end{definition}

It follows from the definition that for two boundary points~$\infty_{\mathbf{g}_1}$,$\infty_{\mathbf{g}_2} \in \partial \mathcal{T}$ there is a unique element~$v_{1,2} \in \mathcal{T}$ such that (cf.~Figure~\ref{f_scaling_graph_homo})
\begin{align}\label{e_def_joining_geodesics}
\max_{v \in \mathbf{g}_1 \cap \mathbf{g}_2}d_{\mathcal{T}}(\mathbf{r},v) = d_{\mathcal{T}} (\mathbf{r}, v_{1,2}),
\end{align}
where~$\mathbf{g}_1$ and~$\mathbf{g}_2$ are the canonical representatives of~$\infty_{\mathbf{g}_1}$ and~$\infty_{\mathbf{g}_2}$ (cf.~Definition~\ref{d_boundary_point}). We will write 
\begin{align}\label{e_d_depth_boundary_points}
|\infty_{g_1} \cap \infty_{g_2} | : = \max_{v \in \mathbf{g}_1 \cap \mathbf{g}_2}d_{\mathcal{T}}(\mathbf{r},v) = d_{\mathcal{T}} (\mathbf{r}, v_{1,2}),
\end{align}
and call~$|\infty_{g_1} \cap \infty_{g_2} |$ the height of the meeting point of the two one-sided geodesics~${g_1}$ and~$g_2$.\\

We are now ready to define the convergence of a sequence of boundary graph homomorphisms to an asymptotic boundary height profile.

\begin{definition}\label{d_convergence_boundary_height_function}
Let~$h_{\partial R_n}: \partial R_n \to \mathcal{T}$ be a sequence of boundary height functions and let~$\left( f_{\partial R}, (a_{i,j})_{k \times k}\right)$ be an asymptotic boundary height profile in the sense of Definition~\ref{d_asymptotic_boundary_height_profile}. We say that the sequence $h_{\partial R_n}$ converges to~$\left(f_{\partial R}, ( a_{i,j} )_{k \times k} \right)$ (i.e.~$\lim_{n \to \infty} h_{\partial R_n} = \left( f_{\partial R},( a_{i,j})_{k \times k} \right)$), if the following conditions are satisfied: 
\begin{itemize}
\item For all $(x,y) \in \partial R_n : d_{\mathcal{T}}(h_{\partial R_n}(x),h_{\partial R_n}(y)) \leq |x-y|_{\ell_1}$.
\item There exist $k$ sequences of boundary points $\{ \infty_{\mathbf{g}_{1,n}} ,..,\infty_{\mathbf{g}_{k,n}} \}_{n \in \mathbb{N}}$ such that for all $1\leq i,j \leq k$:
\begin{align}\label{e_asymptotic_characterization_meeting_point}
\lim_{n \rightarrow \infty} \frac{1}{n}|\infty_{\mathbf{g}_{i,n}} \cap \infty_{\mathbf{g}_{j,n}}|= a_{i,j}
\end{align}
where $|\infty_{\mathbf{g}_{i,n}} \cap \infty_{\mathbf{g}_{j,n}}|$ is given by~\eqref{e_d_depth_boundary_points}.
\item For $z \in \partial R_n$ we define the set 
\begin{align}
 S(z):= \partial R \cap \left\{ x \in \mathbb{R}^m : \left| x - \frac{z}{n} \right|_{\infty} \leq \frac{1}{2n} \right\}.
\end{align}
Then it holds that
\begin{align}\label{e_convergence_height_function}
 \lim_{n \to \infty}& \sup_{ \left\{ z \in \partial R_n: S(z) \neq \emptyset \right\}} \sup_{x \in S(z)} \frac{1}{n}d_{\mathcal{T}} \left( h_{\partial R_n} (z), \mathbf{g}_{f_{\partial R }^2 (x),n} \left( \lfloor n f_{\partial R }^1 (x) \rfloor \right) \right) \\
 &=0,
\end{align}
where $f^1_{\partial R }$ and $f^2_{\partial R }$ are the two components of the map~$f_{\partial R }$.
\end{itemize}
\end{definition}
Definition~\ref{d_convergence_boundary_height_function} is illustrated in Figure~\ref{f_scaling_graph_homo}. The condition~\eqref{e_asymptotic_characterization_meeting_point} ensures that the quantity~$a_{i,j}$ characterizes the asymptotic meeting point of the one-sided geodesics~$\mathbf{g}_i$ and~$\mathbf{g}_j$. One can observe that the compatibility condition~\eqref{e_geodesics_compatibility} on~$a_{i,j}$ is actually a consequence of the condition~\eqref{e_asymptotic_characterization_meeting_point}. The condition~\eqref{e_convergence_height_function} asymptotically characterizes the values of graph homomorphism~$h_{\partial R_n}$ via the asymptotic height profile.\\

Before we turn to the main result of this article let us introduce the local surface tension~$\ent(s)$ for~$s \in \left[-1,1 \right]^m$.
\begin{theorem}(The local surface tension.)\label{p_existence_lst_section_2}
 Let~$\mathbf{g}\subset \mathcal {T}$ be a two-sided geodesic satisfying~$\mathbf{g}(0) = \mathbf{r}$ and let~$S_n := \left\{ 0,1 \ldots, n-1 \right \}^m$. For~$s \in \mathbb{R}^m$ satisfying~$|s|_{\ell_1}<1$, let~$h_{\partial S_n}^s: \partial S_n \to \mathbf{g}$ be the boundary graph homomorphism given by~
 \begin{align*}
 h_{\partial S_n}^s(x) =
 \begin{cases}
 \mathbf{g} (\lfloor s \cdot x \rfloor) , & \mbox{if~$x$ and $\mathbf{g} (\lfloor s \cdot x \rfloor)$ have the same parity},\\
 \mathbf{g} (\lceil s \cdot x \rceil) , & \mbox{if~$x$ and $\mathbf{g} (\lceil s \cdot x \rceil)$ have the same parity}.
 \end{cases}
 \end{align*}
 Then the limit in the following equation exists and defines the local surface tension~$\ent(s)$:
 \begin{align}\label{d_section_2_local_surface_tension}
\ent(s) := \ent^{\tt fixed}(s):=\lim_{n \to \infty} \Ent \left( S_n, h_{\partial S_n}^s \right). 
 \end{align}
 If~$|s|_{\ell_1}=1$ we define~$\ent(s):=0$.
 \end{theorem}
This notion of surface tension is associated to fixed boundary condition. The proof of Theorem~\ref{p_existence_lst_section_2} is given in Section~\ref{s_micro_surface_tenstion}. There, we will also show the equivalence of the notion of local surface tension wrt.~different boundary conditions (cf.~\eqref{e_equivalence_surface_tensions}). Contrary to the case of random domino tilings, we do not have an explicit formula for the local surface tension~$\ent(s)$. In Section~\ref{s_micro_surface_tenstion} we also deduce the convexity of the local surface tension~$\ent (s)$. In analogy to random domino tilings, the authors believe that the local surface tension is strictly convex, but they are missing a proof.\\

Let us now formulate the main result of this article, namely the variational principle for graph homomorphisms to a regular tree. 
As we outlined in the introduction, a variational principle contains two statements. The first statement, namely Theorem~\ref{p_main_result_variational_principle}, gives a variational characterization of the entropy (cf.~\eqref{e_d_microscopic_entropy}) 
$$\Ent \left( R_n, h_{\partial R_n} \right)= - \frac{1}{|R_n|} \ln \left|M (R_n, h_{\partial R_n}) \right| .$$ Hence, it asymptotically characterizes the number of possible graph homomorphisms~$h_n \in M (R_n, h_{\partial R_n})$ with boundary data~$h_{\partial R_n}$.

\begin{theorem}[Variational principle]\label{p_main_result_variational_principle}
Under the Assumption~\ref{a_convergence_of_regions}, we assume that the boundary height functions~$h_{\partial R_n}$ converge to an asymptotic boundary height profile $(f_{\partial R},( a_{i,j})_{k \times k} )$ in the sense of Definition~\ref{d_convergence_boundary_height_function}. Let~$AHP(f_{\partial R}, (a_{i,j})_{k \times k})$ denote the set of asymptotic height profiles that extend~$(f_{\partial R}, (a_{i,j})_{k \times k})$ from~$\partial R$ to~$R$. Given a 1-Lipschitz function $g$ defined on $R$, we define the \emph{macroscopic entropy} via
\begin{align}\label{e_d_macroscopic_entropy}
 \E \left(g\right)= \int_{R} \ent \left( \nabla g (x) \right) dx,
\end{align}
where the local surface tension~$\ent(s)$ is given by Theorem~\ref{p_existence_lst_section_2}. Then it holds that
\begin{align}\label{e_main_result}
 \lim_{n \to \infty}\Ent \left( \Lambda_n, h_{\partial R_n}\right) = \min_{f_R\in AHP ( f_{\partial R, (a_{i,j})_{k \times k}} )} & \E \left(f^1_R \right) .
\end{align}

\end{theorem}


\begin{figure}
 \centering
 \begin{subfigure}[b]{0.4\textwidth}
 \includegraphics[width=\textwidth]{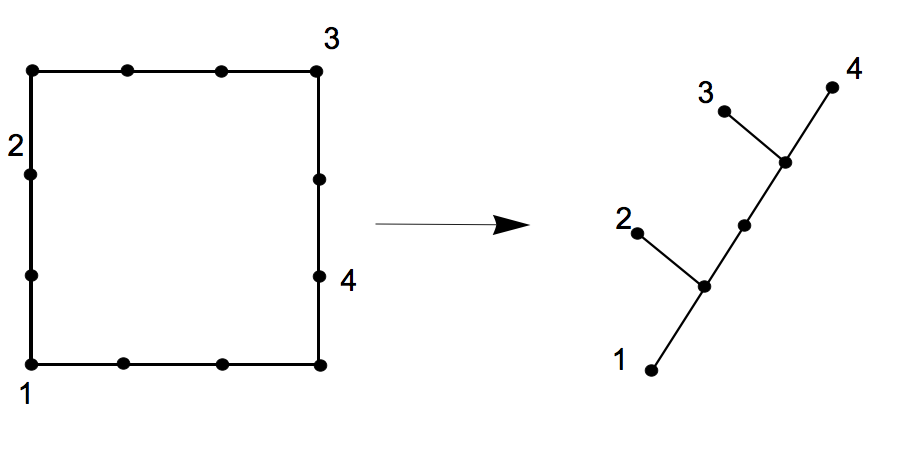}
 \caption{n=1}
 
 \end{subfigure}
 \begin{subfigure}[b]{0.4\textwidth}
 \includegraphics[width=\textwidth]{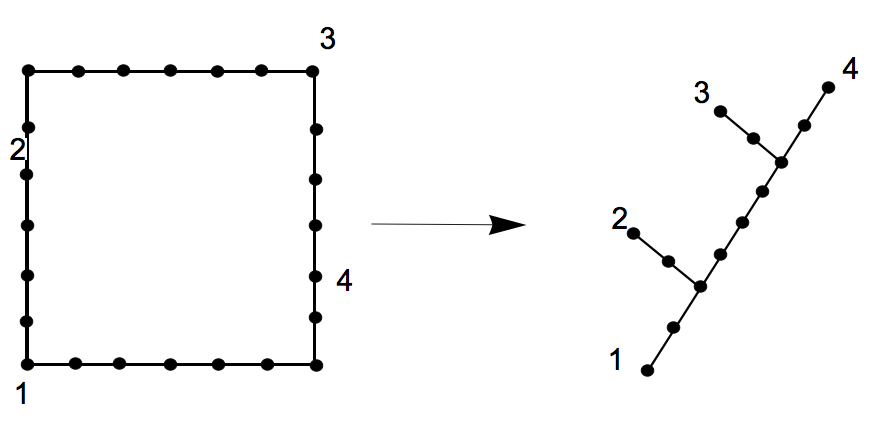}
 \caption{n=2}
 
 \end{subfigure}
 \begin{subfigure}[b]{0.4\textwidth}
 \includegraphics[width=\textwidth]{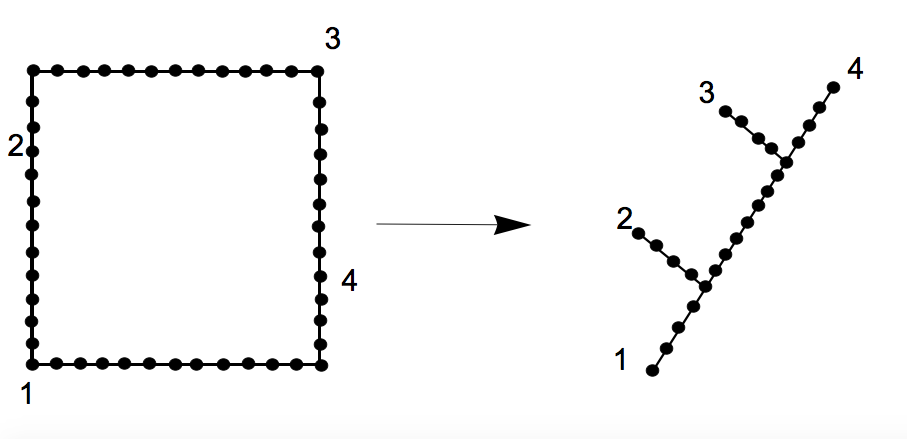}
 \caption{n=4} 
 
 \end{subfigure}
 \caption[scaling]{Scaling of graph homomorphism into $\tree$}\label{f_scaling_graph_homo} 
\end{figure}
 
Because the function~$\nabla f_R^1$ is Lipschitz by Definition~\ref{d_asymp_height_profile}, the gradient~$\nabla f_R^1$ exists almost everywhere by Rademacher's theorem. Because the local surface tension is bounded from below and convex we know that the variational problem~\eqref{e_main_result} has a minimizer. However, we do not know if the minimizer of the surface tension is unique.\\

Let us now turn to the second part of the variational principle, namely the profile theorem (see Theorem~\ref{p_profile_theorem} from below). The profile theorem contains information about the profile of a graph homomorphisms~$h_n$ that is chosen uniformly at random from~$ M (R_n, h_{\partial R_n})$. Heuristically, the statement of the profile theorem is the following. Let us consider an asymptotic boundary height profile~$f_R \in AHP (f_{\partial R}, (a_{i,j})_{k \times k})$. Then the macroscopic entropy~$\E(f_R)$ is given by the number of graph homomorphisms~$h_n \in M(R_n, h_{\partial R_n})$ that are close to~$f_R$. Applying this statement to the minimizer~$f_{\min}$ of the continuous entropy~$\E(f)$ has the following consequence. The uniform measure on the set of graph homomorphisms~$M_{(R_n, h_{\partial R_n})}$ concentrates on graph homomorphisms~$h_n$ that have a profile that is close to~$f_{\min}$. As a consequence, a uniform sample of~$M(R_n, h_{\partial R_n})$ will have a profile that is close to the minimizing profile~$f_{\min}$ for large~$n$.\\

Let us now make this discussion precise. For that purpose, we have to specify when the profile of a graph homomorphism~$h_n$ is close to an asymptotic height profile~$f$.

\begin{definition}\label{d_ball_around_h}
For fixed~$\varepsilon>0$ and integer $n$, we define the simplicial lattice $\mathcal{K}_n^{ \varepsilon}$ with $\varepsilon$-spacing at scale $n$ to be the union of the boundary of the simplices $\Delta^\sigma(k_1,..k_m)$ defined by
\begin{align}\label{e_lattice}
& \Delta^\sigma(k_1,..k_m) \\
&=\{x \in \Z^m: 0 \leq x_{\sigma(1)}- \lfloor \varepsilon n \rfloor k_1\leq ... \leq x_{\sigma(m)}- \lfloor \varepsilon n \rfloor k_m \leq \lfloor \varepsilon n \rfloor \}
\end{align}
for $\sigma$ permutation of size $m$ and $(k_1,..,k_m) \in \Z^m$ (see Figure~\ref{f_set_R_with_grid_section_main_result}). In other words, $\mathcal{K}_n^{ \varepsilon}$ is the set of points in $\Z^m$ for which all of the inequalities in \eqref{e_lattice} are satisfied for a fixed $k$-uplet $(k_1,..,k_m)$ and at least one of them is an equality.

For a given asymptotic height profile~$f$, we define the ball~$HP_{n} (f, \delta, \varepsilon)$ of size~$\delta >0$ on the scale~$\varepsilon>0$ by the formula
\begin{align}\label{e_balls_around_AHP_intro}
& HP_n (f, \varepsilon, \delta) \\
&= \left\{ h_{n} \in M (R_n, h_{\partial R_n}) \ | \ \sup_{x \in \frac{1}{n}\mathcal{K}_n^{ \varepsilon} \cap \frac{1}{n}R_n } \left| \frac{1}{n} d_{\mathcal{T}} (h_n (x), \mathbf{r}) - f^1 \left( \frac{x}{n} \right) \right| \leq \varepsilon \delta \right\},
\end{align} 
where the set~$M (R_n, h_{\partial R_n})$ of graph homomorphisms is given by~\eqref{e_d_micro_partition_function}. 
\end{definition}

\begin{figure}

\includegraphics[width=0.6\textwidth]{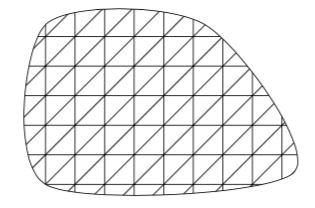}\caption{Illustration of the set~$R$. The grid is the set~$R_{\tt{grid, \varepsilon}}.$}\label{f_set_R_with_grid_section_main_result}
\end{figure}

Informally, $HP_n (f, \varepsilon, \delta) $ is the set of all graph homomorphisms which stays within $\varepsilon \delta $ of the rescaled function $f(\frac{\cdot}{n})$ on the boundary on the rescaled lattice $\frac{1}{n}\mathcal{K}_n^{ \varepsilon}$. Now, let us formulate the profile theorem.

\begin{theorem}(Profile theorem)\label{p_profile_theorem}
Let~$(f_R, (a_{i,j})_{k \times k})$ be an extension of the asymptotic boundary height profile~$(f_{\partial R},(a_{i,j})_{k \times k})$. Then
\begin{align}\label{e_profile_theorem}
 \lim_{n \to \infty} \frac{1}{|R_n|} \ln |HP_n (f_R, \delta, \varepsilon)| = - \E(f^1_R) + \theta_{\varepsilon}(\delta) + \theta(\varepsilon) ,
\end{align}
where~$\{\theta_{\varepsilon}\}_{\varepsilon \in \R}$ is a family of functions indexed by $\varepsilon$ such that for $\varepsilon$ fixed~$\lim_{x\downarrow 0 } \theta_{\varepsilon} (x)=0$ and $\theta$ is a function such that $\lim_{x\downarrow 0 } \theta(x)=0$.
\end{theorem}

\begin{remark}\label{r_geodesic_effect_negligible}
 We want to point out that the second coordinate~$f^2$ does not play a role in the definition~\eqref{e_balls_around_AHP_intro} of $HP_n (f_R, \varepsilon,\delta)$. This means that we neglect the information which one-sided geodesic a graph homomorphism~$ h_{n} \in M (R_n, h_{\partial R_n})$ follows, because the entropic effect of choosing the geodesics is of lower order. Rigorously, this fact is deduced in Lemma~\ref{p_entropy_depth_condition_local_surface_tension} below. Here, let us give a heuristic argument. The variational principle lives on the scale~$|R|$. Approximating the set~$R$ by blocks of side length~$\varepsilon n$ it follows that ~$|R| \approx l \varepsilon^m n^m $, where~$l$ is the number of blocks. Having a close look at the definition~\eqref{e_balls_around_AHP_intro} of $HP_n (f_R, \varepsilon,\delta)$ shows that only the grid $\mathcal{K}_n^{ \varepsilon}$ is important. Hence, the entropic effect of choosing different geodesics lives at most on the scale of the length of the lattice $\mathcal{K}_n^{ \varepsilon}$. The length of the lattice $\mathcal{K}_n^{ \varepsilon}$ is of the order~$l \varepsilon^{m-1} n^{m-1} $ and therefore negligible on the scale of the variational principle. 
\end{remark}

Remark~\ref{r_geodesic_effect_negligible} shows that choosing different one-sided geodesics $\mathbf{g}_1, \ldots, \mathbf{g}_k$ and meeting points~$(a_{i,j})_{k \times k}$ has no effect on Theorem~\ref{p_profile_theorem}. However, it still has an effect on the variational principle formulated in Theorem~\ref{p_main_result_variational_principle}. Choosing different geodesics and meeting points changes the set~$AHP(f_{\partial R}, (a_{i,j})_{k \times k})$ of asymptotic height functions~$f_{R}$ over which the continuous entropy~$\E (f^1_R)$ is minimized. This is another main aspect of how the variational principle of Theorem~\ref{p_main_result_variational_principle} is distinct from the variational principle for domino tilings~\cite{CKP01}.\\

The proofs of the variational principle (see~Theorem~\ref{p_main_result_variational_principle}) and of the profile theorem (see~Theorem~\ref{p_profile_theorem}) are given in Section~\ref{s_proof_variational_principle}. The profile theorem is deduced first and then used for verifying the variational principle via a compactness argument. The second main ingredient in the compactness argument is that the variational problem on the right hand side of~\eqref{e_main_result} has a minimizer, which follows from the fact that the local surface tension is convex and bounded from below (see Section~\ref{s_local_surface_tension} and Theorem~\ref{p_convexity_local_surface_tension}). \\

Let us now explain the main idea for deducing the profile theorem. One observes that Lipschitz functions can approximated very well by piecewise affine function on a simplicial complex. This observation heuristically yields that the profile theorem only needs to be verified for a simplicial complex~$R$ and a piecewise affine profile~$f$. Here, we want to recall that simplicial complex means that~$R$ is a union of finitely many simplices~$R_1, \ldots, R_k$. The desired estimate~\eqref{e_profile_theorem} is deduced in two steps. In the first step one underestimates the number of possible configurations, establishing one direction of the desired (in-)equality. In the second step one overestimates the number of possible configurations, establishing the other direction of the desired (in-)equality.\newline
Let us now describe how the underestimation works. Instead of looking at all possible graph homomorphisms that are close to the profile~$f$, one only considers those that match the profile~$f$ exactly on the boundaries of the individual simplices $R_i$. By this procedure the simplices~$R_i$ become independent of each other. Hence, the entropy of the simplicial complex~$R$ is bigger then the sum of the entropy of the individual simplices~$R_i$ with fixed linear boundary data. \newline
For overestimating the number of configuration on~$R$, we allow the values of the graph homomorphisms on~$\partial R_i$ to be chosen independent from each simplex, as long as they do not fluctuate too much from the affine profile~$f$. Hence, configurations again become independent on each simplex. Therefore, the entropy of the simplicial complex~$R$ is smaller then the sum of the entropy of the individual simplices~$R_i$ with fluctuating linear boundary data. \\

The only missing ingredient is that on a simplex the entropy with fluctuating linear boundary data is equivalent to the entropy with fixed linear boundary data. This fact is provided in Lemma~\ref{p_entropy_depth_condition_local_surface_tension} as a consequence of the equivalence of the local surface tension, i.e.~that the microscopic entropy on a box with fixed or fluctuating boundary conditions agree (see discussion in the introduction and Theorem~\ref{geodesic_square}). For the proof of Theorem~\ref{geodesic_square} we need two crucial technical results, namely the Kirszbraun theorem (see Theorem~\ref{Kirszbraun}) and the concentration inequality (see Theorem~\ref{main_concentration_section_3}) which are provided in Section~\ref{s_technical_observations}.


\section{The local surface tension} \label{s_local_surface_tension}

It is well known that there are many equivalent definitions of the local surface tension (see e.g \cite{She05}). The most simple framework to show the existence is to define the surface tension as the limit of a microscopic entropy on a box with canonical, fixed, linear boundary conditions, as it is done in Theorem~\ref{p_existence_lst_section_2}. In this framework the existence follows easily from subadditivity and Feteke's lemma. One of the most important ingredients for the proof of a variational principle is that the local surface tension is robust to changes in the boundary conditions (see Theorem~\ref{geodesic_square} below). In this article, we show that this robustness is a consequence of a combination of the Kirszbraun theorem and of the concentration inequality. In principle, we could proceed in this framework and deduce the concentration inequality on a box for a fixed canonical boundary data.\\

However, in this article we also show another application of the concentration inequality: It is the existence of a continuum of ergodic gradient Gibbs measures. In order to deduce this result, we need concentration on the level of translation invariant boundary data. To avoid deducing the concentration inequality for two different spaces, we chose to define the local surface tension via the limit of $n$-translation invariant boundary data. This will lead to the notion of local surface tension associated to periodic boundary conditions, denoted in the introduction with~$\ent^{\tt per}$. In Theorem~\ref{p_existence_local_surface_tension} we show the existence of~$\ent^{\tt per}$. A consequence of Theorem~\ref{geodesic_square} is that the local surface tension~$\ent^{\tt fixed}$ and~$\ent^{\tt per}$ agree, justifying the existence of~$\ent^{\tt fixed}$ a-posterior. This also shows that the local surface tension is a universal object. To be self-contained, our argument of the existence does not use subadditivity but relies on our two main technical ingredients, namely the Kirszbraun theorem and the concentration inequality.\\

In Section~\ref{s_micro_surface_tenstion}, we give the precise definition of the microscopic surface tension associated to $n$-translation invariant boundary data and deduce some important auxiliary results. In Section~\ref{s_existence_local_surface_tension}, we show the existence of the local surface tension~$\ent^{\tt per}$ as the limit of the microscopic surface tension (see Theorem~\ref{p_existence_local_surface_tension} from below). We also show that~$\ent^{\tt per}= \ent^{\tt fixed}$, which deduces Theorem~\ref{p_existence_lst_section_2}. In Section~\ref{s_convexity_local_surface_tension}, we deduce the convexity of the local surface tension (see Theorem~\ref{p_convexity_local_surface_tension} from below).

\subsection{Definition of the microscopic surface tension associated to periodic boundary data}\label{s_micro_surface_tenstion}

We use a similar approach as in~\cite{CKP01} and define the local surface tension~$\ent(s)$ as the limit of a microscopic surface tension~$\ent_{n}(s)$, i.e.
\begin{align}\label{d_local_surface_tension_via_periodic_bcs}
 \ent(s) := \ent^{\tt per}(s):= \lim_{n \to \infty} \ent_{n} (s).
\end{align}
In order to define the microscopic surface tension~$\ent_n(s)$ associated to periodic boundary data let us study the translation invariant measures of our model. For this reason, we start with generalizing the notion of periodicity of height functions to graph homomorphisms. For this purpose, we will identify the $d$-regular rooted tree $\tree$ with the group $$G= <\alpha_1,..,\alpha_d|\alpha_1^2=..=\alpha_d^2=e >.$$
This is done through the natural bijection induced by the Cayley graph of $G$ generated by the $\alpha_i$'s. We use the convention that the root of $\tree$ is represented by the identity of $G$. The reason for this identification is that the group structure provides an easy way to define gradient measures. Using the previous bijection we can choose a canonical way to associate a unique $\alpha_i$ to each edge of $\tree$. As a consequence, there is a natural way to associate to a graph homomorphism~$h$ a dual function $\tilde h $ acting on edges of~$ \mathbb{Z}^m$:
\begin{definition}[Dual of a graph homomorphism]\label{d_dual_graph_homomorphism}
Let~$h: \mathbb{Z}^m \to \mathcal{T}$ be a graph homomorphism. We define its dual map $$\tilde h : \left\{e_{x,y} \ | \ x,y \in \mathbb{Z}^m : |x-y|_{\ell_1}=1 \right\} \to \left\{\alpha_1, \ldots, \alpha_d \right\}$$ in the following way. Note that for any~$x \sim y \in \mathbb{Z}^m$ there is a unique~$\alpha_i$ such that~$h(y)=\alpha_i h(x)$. Then, the value of the dual function~$\tilde h$ on the edge~$e_{xy}$ is given by~$\tilde h (e_{xy})= \alpha_i$.\newline
The dual map~$\tilde h$ is the gradient map associated to $h$ and determines the graph homomorphism up to translations in the graph~$\mathcal{T}$. If there is no source of confusion, we will denote the dual map and the graph homomorphism with the same symbol~$h$.\newline
The dual map~$\tilde h$ maps each path $\mathbf{p}=\{ x_0,..,x_l \}$ in~$\mathbb{Z}^m$ onto a word in the alphabet $\{ \alpha_1,..,\alpha_d \}$ and it is not hard to see that this word only depends on the first and last vertex of the path $\mathbf{p}$. We will use the notation $\tilde h(\mathbf{p}_{xy})$ for the word corresponding to a path between $x$ and $y$.
\end{definition}

Let us now define the analog of periodicity for graph homomorphisms.
\begin{definition}[$n$-translation invariant graph homomorphism] \label{d_translational_invariant_graph_homomorphism}
We denote by $\vec{i}_k$ the $k$-th vector of the standard basis of $\Z^m$. Let $h: \mathbb{Z}^m \rightarrow \tree$ be a graph homomorphism. We say that $h$ is $n$-translation invariant if for all $x \sim y \in \Z^m$ and $k \in \{1,..,m\}$
$$\tilde h(e_{xy})=\tilde h(e_{(x+n\vec{i}_k)(y+n\vec{i}_k)}).$$
\end{definition}


\begin{remark}
If we denote by $\tau_{k}$ the shift by $\vec{i}_k$ in $\Z^m$ then $h$ is $n$-translation invariant iff $h \circ \tau_k$ is $n$-translation invariant.
\end{remark}

In order to define the microscopic surface tension we need to associate to every $n$-translation invariant homomorphism $h: \mathbb{Z}^m \rightarrow \tree$ a \emph{slope} which indicates the speed at which the homomorphism travels on the graph in every direction of the plane.

\begin{definition}[Slope of a translation invariant graph homomorphism]
Let $h: \mathbb{Z}^m \rightarrow \tree$ be a $n$-translation invariant homomorphism. The slope $s=(s_1,..,s_m)$ of $h$ is defined by
$$s_k=\frac{1}{n}\min_{x \in \Z^m}d_{\tree}(h(x),h(x+n\vec{i}_k)), \qquad \text{ for }1 \leq k \leq m.$$
\end{definition}

An essential property of $n$-translation invariant homomorphisms is that, if the slope is nonzero, they must stay within finite distance of a unique two-sided geodesic of $\tree$ or, if the slope is zero, stay within finite distance of a single point. This statement is made precise in the next lemma.

\begin{lemma}\label{p_closeness_translation_invariant_geodesic}
Let $h: \mathbb{Z}^m \rightarrow \tree$ be a $n$-translation invariant homomorphism with slope~$s \neq 0$. Then there exist a unique two-sided geodesic $\mathbf{g} \subset \tree$ such that for all ~$x \in \Z^m$
\begin{align}\label{e_finite_range_non_null_slope_estimate}
 d_{\mathcal{T}} \left(h(x) , \mathbf g \right) \leq \frac{mn}{2}.
\end{align}
In this case we say that~$h: \mathbb{Z}^m \rightarrow \tree$ is supported on the two-sided geodesic~$\mathbf{g}$.\newline
If $h$ has slope~$s=0$ then $h$ has finite range, and for all~$x \in \mathbb{Z}^m$
\begin{align}\label{e_finite_range_0slope_estimate}
d_{\mathcal{T}}(h(x), h(0))\leq \frac{mn}{2}.
\end{align}
\end{lemma}

\begin{proof}[Proof of Lemma~\ref{p_closeness_translation_invariant_geodesic}]
We start with considering the case where the slope of $h$ is $(0,..,0)$. In this case the $n$-invariance yields that for $l \in \Z$ and $1 \leq k\leq m$ there exist $x_k$ such that:
\begin{align*}
h(x_k+ln\vec{i}_k)=h(x_k)
\end{align*}
Using the $n$ translation invariance we obtain that for all $1 \leq k\leq m$:
\begin{eqnarray}
\tilde{h}(\mathbf{p}_{0(0+n\vec{i}_k)}) &= & \tilde{h}(\mathbf{p}_{0x_k})\tilde{h}(\mathbf{p}_{x_k(x_k+n\vec{i}_k)})\tilde{h}(\mathbf{p}_{(x_k+n\vec{i}_k)(0+n\vec{i}_k)}) \\
& = & \tilde{h}(\mathbf{p}_{0x_k})\tilde{h}(\mathbf{p}_{x_k0})=id_G \\
\end{eqnarray}
This yields that for $(l_1,..,l_m) \in \Z^m$
\begin{align*}
h(x_1, \ldots, x_m) = h(x_1 + l_1 n, \ldots, x_m + l_m n). 
\end{align*}
Now, the estimate~\eqref{e_finite_range_0slope_estimate} follows directly from the observation that any point~$x \in \mathbb{Z}$ is within graph distance~$\frac{mn}{2}$ of the set~$$\left\{ k_1 \vec{i}_1+ \ldots, +k_m \vec{i}_m \in \mathbb{Z}^m \ | \ k_i \in \mathbb{Z} \right\}.$$

Consider now the case where the slope of $h$ is not zero. We start by noticing that any two-sided geodesic~$\mathbf{g}$ that satisfies~\eqref{e_finite_range_non_null_slope_estimate} must be unique. Indeed, since $G$ is hyperbolic, two geodesics cannot stay within bounded distance of each other. As a consequence $h(\Z^m)$ can only stay within bounded distance of at most one two-sided geodesic~$\mathbf{g}$. \newline

Let us now deduce the estimate~\eqref{e_finite_range_non_null_slope_estimate}. Let $s_i$ be a non-zero coefficient of the slope~$s=(s_1, \ldots, s_m)$. W.l.o.g.~generality we assume that~$s_1 >0$. Let us choose $y=(y_1,..,y_m)$ such that 
$d_{\mathcal{T}}(h(y),h(y+n\vec{i}_1))=n s_1$ and let us define $\omega=\tilde{h}(\mathbf{p}_{y(y+n\vec{i}_1)})$.

Along the line in $\Z^m$ with equation $x_2=y_2,..,x_m=y_m$ the map~$h$ travels with speed~$s_1$ on the only two-sided geodesic~$\mathbf{g}$ which goes through all the vertices $v_k \in \mathcal{T}$ such that the path between $h(y)$ and $v_k$ is $\omega^k$ for $k \in \Z$. Moreover, since $h$ is $n$-periodic we can make the same observation if we replace $y$ by one of the points $y_{k_2,..,k_m}=(y_1,y_2+nk_2,..,y_m+nk_m)$ where $k_2,..,k_m\in\Z$. For each one of those points we obtain that along the line with equation $x_2=y_2+k_2,..,x_m=y_m+k_m$ the map~$h$ travels with speed~$s_1$ on a two-sided geodesic~$\mathbf{g}_{k_2,..,k_m}$. 
 
 Since any two of those lines stays within bounded distance of each other, any two-sided geodesics in $ \{\mathbf{g}_{k_2,..,k_m} \}_{k_2,..,k_m \in \Z}$ must also stay within bounded distance of each other. Hence we deduce that $\mathbf{g}_{k_2,..,k_m}= \mathbf{g}$ for all $k_2,..,k_m\in\Z$. As a consequence we also deduce that for all $k_1,k_2,..,k_m\in\Z$, the vertex $h(y_{k_1,..,k_m})$ is on $\mathbf{g}$ where $y_{k_1,..,k_m}=(y_1+nk_1,..,y_m+nk_m)$.
 
Now, let~$z \in \mathbb{Z}^m$ be arbitrary, we can write~$z=y_{k_1,..,k_m}+ (r_1,..,r_m)$ for some numbers~$k_1,..,k_m\in \mathbb{Z}$ and~$ - \frac{n}{2} \leq r_i < \frac{n}{2}$ for $1 \leq i \leq m$. Since graph homomorphisms are $1$-Lipschitz, we obtain that 
$$d_{\mathcal{T}}(h(z),h(y_{k_1,..,k_m})) \leq \sum_{1 \leq i \leq m} r_i \leq \frac{mn}{2}$$
 which is the desired estimate~\eqref{e_finite_range_non_null_slope_estimate}.
\end{proof}

\begin{definition}[Microscopic surface tension~$\ent_{n}(s)$]\label{e_def_micro_surface_tension}
Let~$\mathbf{g}\in \mathcal{T}$ be a periodic two-sided geodesic. For~$s=(s_1, \ldots, s_m) \in [-1,1]^m$ with~$|s|_{\ell_1}<~1$ let us denote by $\mathcal{H}^\mathbf{g}_n(s)$ the set 
\begin{align}
& \mathcal{H}^{\mathbf{g}}_n (s) : = \left\{ h: \mathbb{Z}^m \to \tree \ : \ \mbox{$h$ is $n$-translation invariant} \right.\\[-0.5ex] 
& \quad \mbox{with slope~$\left( \frac{\lfloor s_1 n \rfloor}{n}, \ldots, \frac{\lfloor s_m n \rfloor}{n} \right)$ supported on~$\mathbf{g}$ and} \\ 
 & \quad \mbox{$\Pi_{\mathbf{g}} \left( h(0) \right) = \mathbf{g} (0)$}\text{ if $d_{\mathcal{T}} \left(h(0),\Pi_{\mathbf{g}} \left( h(0) \right)\right)=0\mod 2$ and}\\
 & \quad \left. \mbox{$\Pi_{\mathbf{g}} \left( h(0) \right) = \mathbf{g} (1)$ otherwise } \right\},
\end{align}
where~$\Pi_{\mathbf{g}} : \mathcal{T} \to \mathbf{g}$ denotes the projection onto the two-sided geodesic~$\mathbf{g}$. 

The microscopic surface tension~$\ent_{n}(s)$ is defined as
\begin{align}
 \ent_n (s) : = - \frac{1}{n^m} \ln \left| \mathcal{H}^\mathbf{g}_n(s) \right|. 
\end{align}
We denote by $\mathbb{P}_n^{s}$ the uniform probability measure on $\mathcal{H}^\mathbf{g}_n(s)$.
We also define
\begin{align}
\mathcal{H}^{\mathbf{g},free}_n (s) &: = \left\{ h: \mathbb{Z}^m \to \tree \ : \ \mbox{$h$ is $n$-translation invariant} \right.\\[-0.5ex] 
& \qquad \qquad \left. \mbox{with slope $\left( \frac{\lfloor s_1 n \rfloor}{n}, \ldots, \frac{\lfloor s_m n \rfloor}{n} \right)$ supported on~$\mathbf{g}$} \right\}. 
\end{align}
\end{definition}

\begin{remark}
 Because all two-sided geodesics are isomorphic, the definition of~$\ent_n (s)$ is independent from the particular choice of~$\mathbf{g}$. In particular by re-orientating the geodesic~$\mathbf{g}$ we assume wlog.~that~$s_1 \geq 0$. The second and third condition in the definition of~$\mathcal{H}^{\mathbf{g}}_n (s)$ anchors the height/depth of the graph homomorphisms~$h$ on the geodesic~$\bf{g}$. As a consequence, the space~$\mathcal{H}^{\mathbf{g}}_n (s)$ is finite. The parity condition on the anchoring comes from the fact that graph homomorphisms conserve the parity. For the precise definition of depth we refer to Definition~\ref{d_depth_on_a_tree}. 
\end{remark}

We note that the Definition~\ref{e_def_micro_surface_tension} of~$\ent_n(s)$ is well posed, i.e.~the set~$\mathcal{H}^\mathbf{g}_n(s)$ is not empty. Indeed, one can easily construct elements of~$\mathcal{H}^\mathbf{g}_n(s)$ by using the Kirszbraun theorem for graphs (see Theorem~\ref{Kirszbraun} from below). The space~$\mathcal{H}^{\mathbf{g},free}_n(s)$ does not play a further role in Section~\ref{s_local_surface_tension} but will become important for the proof of the concentration inequality in Section~\ref{s_concentration}.

\subsection{Existence of the local surface tension }\label{s_existence_local_surface_tension}

In this section, we show that the limit~\eqref{d_local_surface_tension_via_periodic_bcs}, defining the local surface tension~$\ent^{\tt per}(s)$ associated to periodic boundary data. 
\begin{theorem}\label{p_existence_local_surface_tension}
Let $s\in \mathbb{R}^m$ such that $|s|_{\ell_1} < 1$ and let~$\ent_n(s)$ be given by Definition~\ref{e_def_micro_surface_tension}. Then the limit
\begin{align}
 \ent^{\tt per}(s) : = \lim_{n \to \infty} \ent_{n} (s)
\end{align}
 exists and defines the local surface tension~$\ent^{\tt per} (s)$. For~$|s|_{\ell_1} = 1$ we define~$\ent^{\tt per}(s):=0.$
\end{theorem}

In this section, we also show another result which is fundamental not only for the proof of Theorem~\ref{p_existence_local_surface_tension} but also for deducing the variational principle. It is the robustness of the local surface tension with respect to changes in the boundary condition (see Theorem~\ref{geodesic_square} below). A combination of Theorem~\ref{p_existence_local_surface_tension} and Theorem~\ref{geodesic_square} immediately yields the statement of Theorem~\ref{p_existence_lst_section_2} and shows the identity
\begin{align}
 \ent^{\tt per}(s) = \ent^{\tt fixed}(s).
\end{align}
In the remaining article, let us follow the convention that $\ent(s)$ denotes the notion of local surface tension~$\ent^{\tt per}(s)$ associated to periodic boundary data.\\

The proof of Theorem~\ref{p_existence_local_surface_tension} and the proof of Theorem~\ref{geodesic_square} uses the following two ingredients. The first ingredient is a Kirszbraun theorem for graphs. It states under which conditions one can attach together two different graph homomorphisms.
\begin{theorem}[Kirszbraun theorem for graphs] \label{Kirszbraun}
	Let $\Lambda$ be a connected region of $\Z^m$, $S$ be a subset of $\Lambda$ and $\bar{h}: S \rightarrow \mathcal{T}$ be a graph homomorphism which conserves the parity. There exists a graph homomorphism $h: \Lambda \rightarrow \mathcal{T}$ such that $h = \bar{h}$ on $S$ if and only if for all $x,y$ in $S$ 
	\begin{equation}
	d_{\mathcal{T}}(\bar{h}(x),\bar{h}(y)) \leq d_{\Lambda}(x,y) .
	\end{equation}
\end{theorem}
We give the proof of Theorem~\ref{Kirszbraun} in Section~\ref{s_kirszbraun}.

The second ingredient is a concentration inequality. It states that, for canonical boundary data, a graph homomorphism cannot deviate too much from a linear height profile.
\begin{theorem}[Concentration inequality] \label{main_concentration_section_3}
	There exists universal constants $C$ and $c$ such that, under the uniform measure~$\mathbb{P}_n^{s}$ on $\mathcal{H}^\mathbf{g}_n(s)$, for all $x=(x_1,..,x_m)$ in $\Z^m$, and for all~$\varepsilon \geq n^{-0.45}$ we have
	\begin{align}\label{e_main_concentration}
	\mathbb{P}_n^{s} \left( \max_{x \in \mathbb{Z}^m} d_{\mathcal{T}}(h(x),\mathbf{g}(\lfloor s \cdot x \rfloor)) \geq \varepsilon n \right) \leq Ce^{-c\varepsilon^2 n},
	\end{align}
	where~$\mathbf{g}$ is a two-sided geodesic and the map~$\mathbf{g}(n)$ is given by Convention~\ref{e_d_point_on_geodesic}.
\end{theorem}
We give the proof of Theorem~\ref{main_concentration_section_3} in Section~\ref{s_concentration}. 

The first step towards the proof of Theorem~\ref{geodesic_square} is the following statement, which shows that the microscopic surface tension is not oscillating wildly between two values with ratio close to $1$. 

\begin{lemma}\label{p_comparison_ent_n_ent_m}
	Let $s$ be such that $|s|_{\ell_1} < 1$, let $\varepsilon >0$ and let $n,n_1,n_2$ be three integers such that $(1-\varepsilon)n \leq n_1,n_2 \leq n$. Then the following inequality holds
	\begin{align}
	\left| \ent_{n_1} (s) - \ent_{n_2} (s) \right| \leq \theta(\varepsilon)+\theta \left( \frac{1}{n} \right).
	\end{align}
\end{lemma}

\begin{figure}
	
	
	

	
	\includegraphics[width=0.4\textwidth]{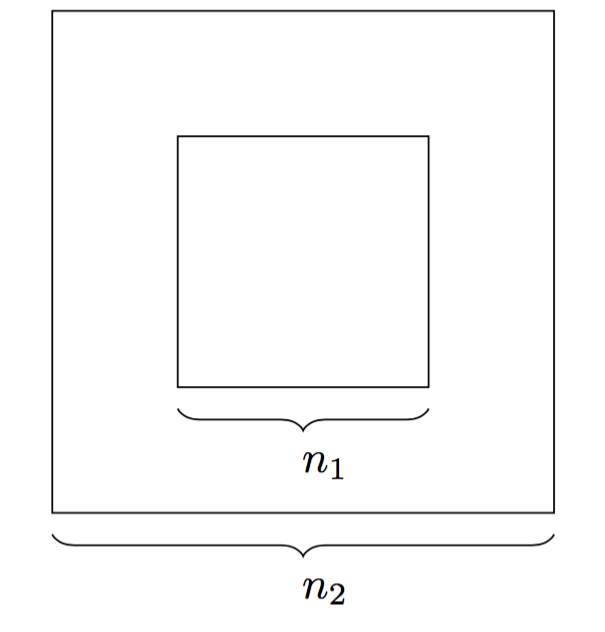}
	
	\caption{The block~$S_{n_2}$ with the centered block $S_{n_1}$ inside.}\label{f_existence_local_entropy_nested_blocks}
\end{figure}

\begin{proof}[Proof of Lemma~\ref{p_comparison_ent_n_ent_m}]
	Before starting the argument let us recall the definition of~$\ent_{n}(s)$. It is defined via
	\begin{align}
	\ent_{n}(s) := - \frac{1}{n^m} \ln \left| \mathcal{H}_{n}^{\mathbf{g}} (s) \right|.
	\end{align}
Since $n_1$ and $n_2$ have symmetric roles in the statement of Lemma \ref{p_comparison_ent_n_ent_m}, it is sufficient to show that:
	\begin{align}\label{e_simrole}
 \ent_{n_2} (s) \geq \ent_{n_1} (s)+\theta(\varepsilon)+\theta(\frac{1}{n}).
\end{align}
	Set $\delta = \frac{1-|s|_{\ell_1}}{2}\varepsilon$, our first step is to show that for large $n$ the size of the set~$\mathcal{H}_{n}^{\mathbf{g}} (s)$ is comparable to the set
	\begin{align}
	M_n : = \left\{ h \in \mathcal{H}_{n}^{\mathbf{g}} (s) \ | \ \max_{x \in S_n} d_{\mathcal{T}}(\mathbf{g}(\lfloor s \cdot x \rfloor),h(x)) \leq \delta n \right\}.
	\end{align} 
	This is a direct consequence of the concentration inequality~\eqref{e_main_concentration} of Theorem~\ref{main_concentration_section_3} which yields that for some universal~$c>0$
	\begin{equation} \label{concentration_bis}
	\mathbb{P}_n^{s} \left( \max_{x \in S_n} d_{\mathcal{T}}(\mathbf{g}(\lfloor s \cdot x \rfloor),h(x)) \geq  \delta n \right) \leq Ce^{-c \delta^2 n}.
	\end{equation}
	This implies that for $n$ large enough 
	\begin{align}\label{e_concentration_max_smaller_box}
	\mathbb{P}_n^{s}\left( h \in M_n \right) \geq 1/2.
	\end{align}
	Because~$\mathbb{P}_n^{s}$ is the uniform measure on~$\mathcal{H}^\mathbf{g}_n(s)$, this yields that for all sufficiently large~$n \in \mathbb{N}$
	\begin{align}\label{e_comparison_H_n_and_M_n}
	|M_n| \leq |\mathcal{H}^\mathbf{g}_n(s)| \leq 2 |M_n|.
	\end{align}
	
	For the second step of the argument, let us assume that the box~$S_{n_1}$ and the box~$S_{n_2}$ have the same center (see Figure~\ref{f_existence_local_entropy_nested_blocks}).
	We will show that
	\begin{align}\label{e_comparison_M_n_and_M_n_plus_2}
	 |M_{n_2}| \leq |\mathcal{H}_{n_1}^{\mathbf{g}} (s)| d^{2m\varepsilon n^m}. 
	\end{align}
	Let~$\tilde h \in M_{n_2}$.
	For all $x \in \partial S_{(1-2\varepsilon)n}$, $y \in \partial S_{n_1}$ and large enough~$n$ it holds
	\begin{align}\label{e_Kirszbraun_estimate}
	d_{\mathcal{T}}(\tilde{h}(x), \mathbf{g}(\lfloor s\cdot y \rfloor) ) & \leq d_{\mathcal{T}}(\mathbf{g}(\lfloor s \cdot x \rfloor),\mathbf{g}(\lfloor s \cdot y \rfloor))+ \delta n\\
	& \leq |s|_{\ell_1} d_{\mathbb{Z}^m}(x,y)+\frac{1-|s|_{\ell_1}}{2}d_{\mathbb{Z}^m}(x,y) \leq d_{\mathbb{Z}^m}(x,y),
	\end{align}
	where in the last inequality we used that: 
	$$  \delta n= \frac{1-|s|_{\ell_1}}{2} \varepsilon n\leq \frac{1-|s|_{\ell_1}}{2} d_{\mathbb{Z}^m}(x,y).$$
	Using the Kirszbraun theorem, this implies that the restriction ~$\tilde h_{|S_{(1-2\varepsilon)n}}$ of $\tilde{h}$ to the box~$ S_{(1-2\varepsilon)n}$ can be extended to the homomorphism $g$ defined by $g(x)=\mathbf{g}(\lfloor s \cdot x \rfloor)$ on $\partial S_{n_1}$. Additionally, we observe that there are less than~$d^{2m\varepsilon n^{m}}$ many ways to extend a configuration on~$S_{(1-2\varepsilon)n}$ to the boundary $\partial S_{n_2}$ where $2m\varepsilon n^{m}$ represents an upper bound on the maximal volume of the box difference $S_{n_2} \setminus S_{(1-2\varepsilon)n}$. This yields the desired estimate
	\begin{align}\label{e_Mn_entropy}
	|M_{n_2}| \leq | \mathcal{H}_{n_1}^{\mathbf{g}} (s)| d^{2m\varepsilon n^{m}}.
	\end{align}
	By combining the estimates \eqref{e_concentration_max_smaller_box} and \eqref{e_Mn_entropy} we obtain that
	\begin{align}\label{e_Mn_entropy2}
	\frac{1}{2}| \mathcal{H}_{n_2}^{\mathbf{g}} (s)| \leq | \mathcal{H}_{n_1}^{\mathbf{g}} (s)| d^{2m\varepsilon n^{m}}.
	\end{align}
	Now if we take the logarithm and divide both sides by $-\frac{1}{n^m_2}$, we can rewrite \eqref{e_Mn_entropy2} as
	\begin{align}\label{e_Mn_entropy3}
	 \ent_{n_2}(s) & \geq - \frac{n_1^m}{n_2^m} \ent_{n_1}(s)- \frac{1}{n_2^m} \ln 2 - \frac{2m\varepsilon n^m}{n_{2}^m} \ln d \\
	 & \geq \frac{1}{(1-\varepsilon)^m} \ent_{n_1}(s)- \frac{1}{n_2^m} \ln 2 - \frac{2\varepsilon}{(1-2\varepsilon)} \ln d \\
	 & \geq \ent_{n_1}(s)+\theta(\varepsilon)+\theta \left( \frac{1}{n} \right), 
	\end{align}
where we used that $\frac{1}{(1-\varepsilon)^m}=1+\theta(\varepsilon)$ and $\ent_{n_1}(s)\geq -d$ to go from the third to the last line in the previous inequality. This gives us the desired estimate and finishes our proof.
\end{proof}

Another ingredient for the proof of Theorem~\ref{p_existence_local_surface_tension} is that the microscopic entropy with $n$-translation invariance and fixed boundary conditions are equivalent:
\begin{theorem}\label{geodesic_square}
Let $\delta > 0$, $S_n$ be a $m$-dimensional hypercube of size $n$, $\mathbf{g}$ be a fixed two-sided geodesic in $\tree$ and~$s\in \mathbb{R}^m$ such that~$|s|_{\ell_1} < 1$. Let $h_{\partial S_n}: \partial S_n \rightarrow \tree$ such that for all $x \in \partial S_n$ it holds 
\begin{align}\label{e_bc_condition_close_geodesic}
d_{\tree}(h_{\partial S_n}(x),\mathbf{g} (\lfloor s \cdot x \rfloor) \leq \delta n,
\end{align}
where~$\mathbf{g} (\lfloor s \cdot x \rfloor)$ is given by Convention~\ref{e_d_point_on_geodesic}. Then it holds that 
\begin{align}\label{e_ent_free_vs_bounded_bc}
\Ent(S_n, h_{\partial S_n}) = \ent_n(s) + \theta \left( \frac{1}{n} \right) + \theta (\delta),
\end{align}
 where~$\Ent(S_n, h_{\partial S_n})$ is given by~\eqref{e_d_microscopic_entropy} and $\ent_n(s)$ is the microscopic surface tension given by Definition~\ref{e_def_micro_surface_tension}.
\end{theorem}

\begin{proof}[Proof of Theorem~\ref{geodesic_square}]
In order to deduce the desired statement it suffices to show that
\begin{align}~\label{e_fixed_vs_free_bc_lower_estiamte2}
\ent_{ n }(s) \leq \Ent (S_n, h_{\partial S_n})+ \theta (\delta)+ c \frac{\delta^2}{n^{m-1}}
\end{align}
and 
\begin{align}~\label{e_fixed_vs_free_bc_upper_estiamte2}
\ent_{ n }(s) \geq \Ent (S_n, h_{\partial S_n})+ \theta (\delta) + c \frac{\delta^2}{n^{m-1}}.
\end{align}

We start with deducing the estimate~\eqref{e_fixed_vs_free_bc_lower_estiamte2}. We set $\varepsilon= \frac{\delta}{1-|s|_{\ell_1}}$. It follows from the concentration estimate of Theorem~\ref{main_concentration_section_3} that
\begin{align}\label{e_concentration_free_bc_vs_fixed_bc}
\mathbb{P}_{(1-2\varepsilon)n}^s \left( \max_{x \in \Z^m}d_{\tree}(h(x),\mathbf{g}(\lfloor s \cdot x \rfloor ) \geq \delta n \right) \leq C e^{-c \delta^2 n}, 
\end{align}
where~$\mathbb{P}_{(1-2\varepsilon)n}^s$ denotes the uniform measure on $\mathcal{H}^{\mathbf{g}}_{(1-2\varepsilon)n} (s)$. Let us define the set~$M_{(1-2\varepsilon)n, \delta} \subset \mathcal{H}^{\mathbf{g}}_{(1-2\varepsilon)n} (s)$ according to
\begin{align}
 M_{(1-2\varepsilon)n, \delta} := \left\{ h \in \mathcal{H}^{\mathbf{g}}_{(1-2 \varepsilon)n} (s) : \max_{x \in \Z^m}d_{\tree}(h(x),\mathbf{g}(\lfloor s \cdot x \rfloor) \leq \delta n \right\}.
\end{align}
Then it follows from~\eqref{e_concentration_free_bc_vs_fixed_bc} that
\begin{align}\label{e_comparison_free_bc_close_to_geodesic_bc}
\frac{ \left| M_{(1-2\varepsilon)n, \delta} \right|}{|\mathcal{H}^{\mathbf{g}}_{(1-2 \varepsilon)n} (s)|} \leq C e^{-c \delta^2 n}.
\end{align}
By taking the logarithm and dividing by $(1- 2 \varepsilon)^m n^m$ in the previous equation we obtain
\begin{align}\label{e_comparison_free_bc_close_to_geodesic_bc}
 \ent_{(1- 2 \varepsilon) n} (s) \geq - \frac{1}{(1- 2 \varepsilon)^m n^m} \ln \left| M_{(1-2\varepsilon)n, \delta} \right| - c\frac{\delta^2}{n^{m-1}}.
\end{align}
We observe that due to \eqref{e_bc_condition_close_geodesic} and the Kirszbraun theorem (cf.~Theorem~\ref{Kirszbraun}) any element~$h_{(1- 2 \varepsilon) n} \in M_{(1-2\varepsilon)n, \delta}$ can be extended to a graph homomorphism~$h : S_n \to \mathcal{T}$ such that $h=h_{\partial S_n }$ on $\partial S_n$ and $h=h_{(1-2\varepsilon)n}$ on $S_{(1-2\varepsilon)n}$. This implies that for large enough~$n$
\begin{align}
 \Ent(S_n, h_{\partial S_n}) & \leq - \frac{1}{n^m} \ln \left| M_{(1-2\varepsilon)n, \delta} \right| \\
& \leq (1-2\varepsilon)^m \ent_{(1- 2\varepsilon)n} (s) + c \frac{\delta^2}{n^{m-1}} \\
& = \ent_{n} (s) + \theta (\delta) + c\frac{\delta^2}{n^{m-1}} +\theta \left( \frac{1}{n}\right), 
\end{align}
where we have used the estimate~\eqref{e_comparison_free_bc_close_to_geodesic_bc} from above and the identity 
\begin{align}
 \ent_{(1- 2 \varepsilon) n }(s) = \ent_{ n }(s) + \theta \left( \delta\right)+\theta \left( \frac{1}{n}\right),
\end{align}
which follows from Lemma~\ref{p_comparison_ent_n_ent_m} applied with the parameters $n_1= (1- 2 \varepsilon) n $ and $n_2=n$ for large enough~$n$. This verifies the estimate~\eqref{e_fixed_vs_free_bc_lower_estiamte2}.\\

The estimate~\eqref{e_fixed_vs_free_bc_upper_estiamte2} can be verified by a similar argument as was used for~\eqref{e_fixed_vs_free_bc_lower_estiamte2}. Instead of restricting~$S_n$ to a smaller box and compare~$S_n$ with~$\mathcal{H}_{(1-2\delta)n }^{\mathbf{g}}$ and~$M_{(1-2\delta)n, \varepsilon}$, one has to extend the box $S_n$ and comparing it with~$\mathcal{H}_{(1+2\delta)n }^{\mathbf{g}}$ and $M_{(1-2\delta)n, \varepsilon}$. We omit the details. 
\end{proof}
The merit of Theorem~\ref{geodesic_square} is the following: In order to show that the limit~$\lim_{n\to \infty }\ent_n (s)$ exists it suffices to show that the limit
$$\lim_{n \to \infty} \Ent(S_n, h_{\partial S_n})$$ exists. The advantage of considering~$\Ent(S_n, h_{\partial S_n})$ is that the boundary data~$h_{\partial S_n}$ is fixed and can be chosen such that~\eqref{e_bc_condition_close_geodesic} is satisfied.
For a well chosen configuration~$h_{\partial S_n}$ one could now show the existence of the limit~$\lim_{n \to \infty} \Ent(S_n, h_{\partial S_n})$ using subadditivity. However, for being self-contained let us give an alternative argument. From now on, let us fix one particular sequence of boundary data $h_{\partial S_n}$ that additionally to~\eqref{e_bc_condition_close_geodesic} also satisfies the following condition.

\begin{definition}(Periodic boundary data)\label{d_periodic_boundary_data}
Let~$h_{\partial S_n}$ denote a boundary graph homomorphism on the boundary~$\partial S_n$ of the box~$S_n$. Let~$\tilde h_{\partial S_n}$ denote the dual boundary graph homomorphism on the edge set~$\mathcal{E}_{\partial S_n}$ of $\partial S_n$ (see Definition~\ref{d_dual_graph_homomorphism}). We say that the boundary graph homomorphism $h_{\partial S_n}$ has well-periodic boundary data if the following two conditions are satisfied:
\begin{itemize}
\item If $e_{x,y} \in \mathcal{E}_{\partial S_n}$ and $e_{x+ n \vec{i}_k,y+ n \vec{i}_k} \in \mathcal{E}_{\partial S_n}$ for some~$k \in \left\{1, \ldots n \right\}$, then~$\tilde h_{\partial S_n} (e_{x,y}) = \tilde h_{\partial S_n} (e_{x+ n \vec{i}_k,y+ n \vec{i}_k}) $.\\[-0.5ex]

\item If $e_{x,y} \in \mathcal{E}_{\partial S_n}$ and $e_{x- n \vec{i}_k,y+ n \vec{i}_k} \in \mathcal{E}_{\partial S_n}$ for some~$k \in \left\{1, \ldots n \right\}$, then~$\tilde h_{\partial S_n} (e_{x,y}) = \tilde h_{\partial S_n} (e_{x+ n \vec{i}_k,y+ n \vec{i}_k}) $.
\end{itemize}
\end{definition}
The Definition~\ref{d_periodic_boundary_data} has the following simple interpretation. The dual boundary graph homomorphism~$\tilde h_{\partial S_n}$ can be understood as a coloring of the edges of the set~$\partial S_n$. Then the boundary data~$h_{\partial S_n}$ is periodic, if the coloring of one face of~$\partial S_n$ matches the coloring of the opposite face.\\

The advantage of using periodic boundary data is that one gets monotonicity of a subsequence of~$\Ent(S_n, h_{\partial S_n})$ for free.

\begin{lemma}\label{p_monotonicity_subsequence}
Under the same assumptions as in Theorem~\ref{geodesic_square}, let us consider the entropy~$\Ent(S_n,h_{\partial S_n})$. We additionally assume that the boundary graph homomorphism~$h_{\partial S_n}$ is periodic in the sense of Definition~\ref{d_periodic_boundary_data}. On the box~$S_{kn}$ we consider the boundary condition~$h_{\partial S_{kn}}$ that arises from attaching~$k$ copies of~$h_{\partial S_n}$ to each other. Then it holds that for all integers~$k \in \mathbb{N}$
\begin{align}\label{e_monotonicity_subsequence}
\Ent(S_n,h_{\partial S_n}) \geq \Ent(S_{nk},h_{\partial S_{nk}}).
\end{align} 
\end{lemma}
The proof of Lemma~\ref{p_monotonicity_subsequence} follows from a simple underestimation of the configurations in~$\Ent(S_{nk},h_{\partial S_{nk}})$. Because the boundary data~$ h_{\partial S_{n}}$ is periodic one can just take a configuration~$h_{S_n}$ on~$S_{n}$ and extend it to the box~$S_{n}$ by attaching~$k$ copies of~$h_{S_n}$ to each other. By construction, the resulting configuration on~$S_{nk}$ will have the correct boundary data~$h_{\partial S_{nk}}$ and therefore the estimate~\eqref{e_monotonicity_subsequence} follows automatically. We omit the details of this proof. \\

Now, we have everything that is needed for the proof of Theorem~\ref{p_existence_local_surface_tension}.

\begin{proof}[Proof of Theorem~\ref{p_existence_local_surface_tension}]

The main idea is to consider a sequence of periodic boundary data~$h_{\partial S_n}$ (see Definition~\ref{d_periodic_boundary_data})
that satisfies~\eqref{e_bc_condition_close_geodesic} and show that
\begin{align}\label{e_limit_entropy_fixed_bc_exists}
\mbox{the limit } \lim_{n \to \infty} \Ent(S_n, h_{\partial S_n}) \mbox{ exists.}
\end{align}
Then it easily follows from statement of Theorem~\ref{geodesic_square} that also 
\begin{align}
\mbox{the limit } \lim_{n \to \infty} \ent_n (s) \mbox{ exists.} 
\end{align}
This would verify the statement of Theorem~\ref{p_existence_local_surface_tension}.\\

We begin with observing that 
\begin{align}
0 \geq c_n:= \Ent(S_n, h_{\partial S_n}) \geq - \ln d.
\end{align}
The reason is that edges take at most $d$-values. Therefore, it suffices to show that the sequence~$c_n$ cannot have two distinct accumulations points~$x_1$ and~$x_2$. We argue by contradiction and assume that~$x_1$ and~$x_2$ are two accumulation points of the sequence~$e_n$ satisfying the relation $$- \ln d \leq x_1 < x_2 \leq 0.$$
Then there exists a number~$l$ such that
\begin{align}
c_{l} \leq \frac{x_1 + x_2}{2}.
\end{align}
By Lemma~\ref{p_monotonicity_subsequence}, the subsequence~$k \to c_{kl}$ is decreasing, hence for all~$k \in \mathbb{N}$
\begin{align}
 c_{kl} \leq \frac{x_1 + x_2}{2}.
\end{align}
We will now show that this implies for large enough~$l$ and~$k$ that also for all~$n \geq kl$
\begin{equation}\label{e_accumulation_points}
 c_{n} \leq \frac{x_1 + x_2}{2} + \varepsilon,
\end{equation}
for some small constant~$\varepsilon >0$. This would be a contradiction to the assumption that~$x_1$ and~$x_2$ are accumulation points of the sequence~$c_n$ and therefore would verify~\eqref{e_limit_entropy_fixed_bc_exists}.\\

Hence, it is left to deduce the estimate~\eqref{e_accumulation_points}. Let~$n \geq lk$. Then we know that we can write
\begin{align}
 n = \tilde k l + v,
\end{align} 
where~$\tilde k \geq k$ and~$0 \leq v \leq l$. By using a combination of Theorem~\ref{geodesic_square} and Lemma~\ref{p_comparison_ent_n_ent_m} with parameters $n_1=\tilde k l$ and $n_2=\tilde k l+v$, it follows that
\begin{align}
 |c_n - c_{\tilde k l}| & \leq \left| \ent_{\tilde k l+ v}(s) - \ent_{\tilde k l} \right| + \theta \left( \frac{1}{\tilde k l} \right) + \theta (\delta) \\
& \leq \theta \left( \frac{1}{k } \right) + \theta (\delta).
\end{align}
Hence, we see that if choosing~$k$ large enough and~$\delta$ small enough that
\begin{align}
 c_n \leq c_{\tilde k l} + \varepsilon \leq \frac{x_1+ x_2}{2} + \varepsilon,
\end{align}
which verifies~\eqref{e_limit_entropy_fixed_bc_exists} and closes the argument. 
\end{proof}

\subsection{Convexity of the local surface tension }\label{s_convexity_local_surface_tension}

Once the existence of the local surface tension~$\ent (s)$ is established it is natural to ask if the local surface tension is convex. This is the case in our model.

\begin{theorem}\label{p_convexity_local_surface_tension}
The local surface tension~$\ent (s)$ given by Definition~\ref{e_def_micro_surface_tension} is convex in every coordinate. In particular, this implies that $\ent (s)$ is convex.
\end{theorem}
As in the proof of the existence of the local surface tension~$\ent(s)$, the main tools for the proof of Theorem~\ref{p_convexity_local_surface_tension} are the Kirszbraun theorem (see Theorem~\ref{Kirszbraun}) and the concentration inequality (see Theorem~\ref{main_concentration_section_3}). 

\begin{proof}[Proof of Theorem~\ref{p_convexity_local_surface_tension}]
We will show that the local surface tension~$\ent(s)$ is convex in every coordinate which yields that~$\ent(s)$ is convex. By symmetry it suffices to show that~$\ent (s)$ is convex in the first coordinate. For convenience, we only give the argument for the case~$m=2$. The argument for the general case is similar. For that reason let~$s_2$ be fixed. We argue by contradiction. Hence, let us suppose that~$\ent(s)$ is not convex in the first coordinate. Then there are numbers~$s_{1,0} <s_{1,1} < s_{1,2}$ such that
\begin{align}
 \frac{1}{2} s_{1,0} + \frac{1}{2} s_{1,2} = s_{1,1} 
\end{align}
and
\begin{align} \label{e_convexity_nonconvex_assumption}
 \ent (s_{1,1}, s_2) > \frac{1}{2}\ent(s_{1,0},s_2) + \frac{1}{2} \ent(s_{1,2},s_2).
\end{align}
 For an integer~$n$ we consider the microscopic entropy~$\ent_{n}(s_{1,1}, s_2)$ given by Definition~\ref{e_def_micro_surface_tension} i.e.
\begin{align}
 \ent_{n} (s_{1,1},s_2) : = - \frac{1}{n^2} \ln \left|\mathcal{H}_{n}^{\mathbf{g}} (s_{1,1},s_2) \right|,
\end{align}
 where~$\mathbf{g}$ denotes a two-sided geodesic in~$\mathcal{T}$. We want to recall that elements~$h \in \mathcal{H}_{n}^{\mathbf{g}} (s_{1,1})$ are graph homomorphisms~$h: S_n \to \mathcal{T}$, where~$S_n \subset \mathbb{Z}^2$ denotes the~$n \times n $ box
\begin{align}
 S_n : = \left\{ 0, \ldots, n-1\right\}^2.
\end{align}
We assume that without loss of generality that~$n$ is odd. The idea is to split up the box~$S_n$ into four boxes of side length~$\frac{n}{2}$ i.e.
\begin{align}
S_{n} = B_1 \cup B_2 \cup B_3 \cup B_4.
\end{align}
Down below, we will compare the number of graph homomorphisms in~$\mathcal{H}_{n}^{\mathbf{g}} (s_{1,1},s_2)$ to the number the number of graph homomorphisms on each sub-box~$B_i$ with a fixed \emph{buckled} boundary (see Figure~\ref{f_convexity_block_decomposition} and Figure~\ref{f_h_b_illustration}). For that purpose let~$h_{b} \in \mathcal{H}_{n}^{\mathbf{g}} (s_{1,1}, s_2)$ be a graph homomorphism such that
\begin{itemize}
\item for all~$x \in \partial S_n$ with~$0 \leq x_1 \leq \frac{n-1}{2}$ and~$x_2 \in \left\{ 0, \frac{n-1}{2} , n-1 \right\}$ it holds
\begin{align}
 d_{\mathcal{T}} \left( h_b(x), \mathbf{g} \left( \lfloor s_{1,0} x_1 + s_2 x_2 \rfloor \right) \leq \delta \frac{n}{2} \right) ;
\end{align}
 \item for all~$x \in \partial S_n$ with~$\frac{n}{2} \leq x_1 \leq n-1$ and~$x_2 \in \left\{ 0, \frac{n-1}{2} , n-1 \right\}$ it holds
\begin{align}
 d_{\mathcal{T}} \left( h_b(x), \mathbf{g} \left( \lfloor s_{1,0} \frac{n-1}{2} + s_{1,2} \left( x_1 - \frac{n-1}{2} \right)+ s_2 x_2 \rfloor \right) \leq \delta \frac{n}{2} \right);
\end{align}
\item for all~$x \in \partial S_n$ with~$x_1 \in \left\{ 0, \frac{n-1}{2} \right\} $ and~$0 \leq x_2 \leq n-1$ it holds
\begin{align}
 d_{\mathcal{T}} \left( h_b(x), \mathbf{g} \left( \lfloor s_{1,0} x_1 + s_2 x_2 \right) \rfloor \leq \delta \frac{n}{2} \right) ; \mbox{and}
\end{align}
 \item for all~$x \in \partial S_n$ with~$x_1 = n-1$ and~$0 \leq x_2 \leq n-1$ it holds
\begin{align}
 d_{\mathcal{T}} \left( h_b(x), \mathbf{g} \left( \lfloor s_{1,1} n + s_2 x_2 \rfloor \right) \leq \delta \frac{n}{2} \right).
\end{align}
\end{itemize}
The role of the graph homomorphism~$h_b$ is to fix the boundary condition on each box~$B_i$,~$i=1, \ldots, 4$. Let us introduce the ad-hoc notation
\begin{align}
U := \left\{ h \in \mathcal{H}_{n}^{\mathbf{g}} (s_{1,1},s_2) \ | \ h(x) = h_b(x) \ \forall x \in \bigcup_{i=1}^4 \partial B_i \right\}.
\end{align}
We underestimate the number of graph homomorphisms by conditioning that a graph homomorphism~$h$ has to coincide with~$h_b$ on the boundary of~$B_i$. This yields
\begin{align}
 \ent_{n} (s_{1,1},s_2) & = - \frac{1}{n^2} \ln |\mathcal{H}_{n}^{\mathbf{g}} (s_{1,1},s_2)| \leq - \frac{1}{n^2} \ln |U|. 
\end{align}
We observe that for any element~$h \in U$ the values on the boundary~$\partial B_i$,~$i =1 , \ldots, 4$, are fixed. Therefore the values of~$h \in U$ on the distinct boxes~$B_i$,~$i=1, \ldots, 4$ are independent. This implies that
\begin{align}
 - \frac{1}{n^2} \ln |U| & = \left( \frac{n}{2}\right)^2 \ \frac{1}{n^2} \sum_{i=1}^4 \Ent(B_i, h_{b|\partial B_i}) . 
\end{align}
Now, we observe that by Theorem~\ref{geodesic_square} it holds that
\begin{align}
 \Ent(B_1, h_{b|\partial B_1})= \Ent(B_3, h_{b|\partial B_3}) = \ent_{\frac{n}{2}} (s_{1,0},s_2) + o(\delta) + o \left(\frac{1}{n} \right)
\end{align}
and
\begin{align}
 \Ent(B_2, h_{b|\partial B_2})= \Ent(B_4, h_{b|\partial B_4}) = \ent_{\frac{n}{2}} (s_{1,2},s_2) + o(\delta) + o \left(\frac{1}{n} \right).
\end{align}
The last estimates in combination with the fact that (cf.~Theorem~\ref{p_existence_local_surface_tension})
\begin{align}
 \ent_n (s_1, s_2) = \ent (s_1, s_2) + o \left( \frac{1}{n} \right)
\end{align}
implies that
\begin{align}
 \ent (s_{1,1},s_2) & \leq \frac{1}{2} \ent (s_{1,0},s_2) + \frac{1}{2} \ent (s_{1,2},s_2) + o(\delta) + o \left(\frac{1}{n} \right),
\end{align}
which contradicts~\eqref{e_convexity_nonconvex_assumption} by choosing~$\delta>0$ small enough and~$n$ large enough and therefore closes the argument.
\end{proof}

\begin{figure}




 \includegraphics[width=0.4\textwidth]{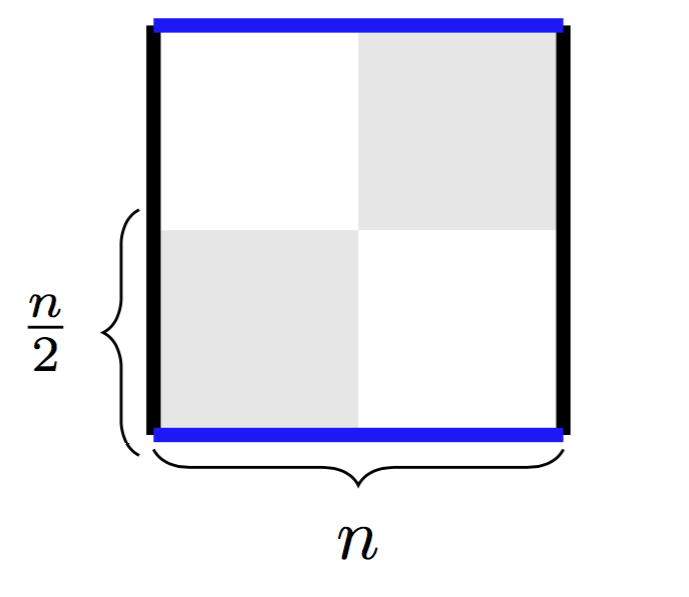}

\caption{Schematic drawing of a typical graph homomorphism~$h\in \mathcal{H}_{n}^{\mathbf{g}}(s_1, s_2)$ on the block~$S_n$. A blue line means that the graph homomorphism~$h$ travels with speed~$s_1$ on~$\mathbf{g}$ and a black line means that~$h$ travels with speed~$s_2$.
}\label{f_convexity_block_decomposition}
 
\end{figure}

\begin{figure}







 \includegraphics[width=0.42\textwidth]{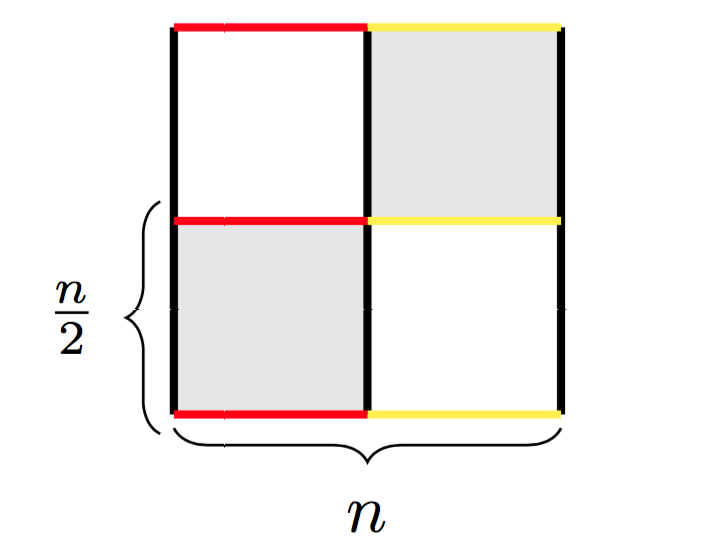}

\caption{Schematic drawing of a graph homomorphism~$h_b$. A red line means that the graph homomorphism~$h$ travels with speed~$s_{1,0}$ on~$\mathbf{g}$. A yellow line means that the graph homomorphism~$h$ travels with speed~$s_{1,2}$. A black line means that~$h$ travels with speed~$s_2$.
}\label{f_h_b_illustration}
\end{figure}

\section{A Kirszbraun theorem and a concentration inequality }\label{s_technical_observations}
In this section, we provide the technical tools that are needed in our proof of the variational principle to overcome the difficulty that our model is not integrable. In Section~\ref{s_kirszbraun} we deduce the Kirszbraun theorem in regular trees. In Section~\ref{s_concentration} we deduce the concentration inequality for graph homomorphisms. Those tools were used in Section~\ref{s_local_surface_tension} to show the existence and convexity of the local surface tension~$\ent(s)$ (see Theorem~\ref{p_existence_local_surface_tension} and Theorem~\ref{p_convexity_local_surface_tension}) and the equivalence of fixed and free boundary conditions (see Theorem~\ref{geodesic_square}). Those tools are also the technical foundation to derive the existence of a continuum of shift-invariant ergodic gradient Gibbs measures in Section~\ref{ergodic}.

\subsection{A Kirszbraun theorem for graph homomorphisms}\label{s_kirszbraun}

For continuous metrics, Kirszbraun theorems state that under the right conditions a $k$-Lipschitz function defined on a subset of a metric space can be extended to the whole space (cf.~\cite{Kir34,Val43,Sch69}). The goal of this section is to show that such theorems also exist for various spaces of discrete functions. We also derive a corollary statement of the Kirszbraun theorem which will be important in our proof on the variational principle. We only consider the special case of graph homomorphisms from $\Z^m$ to a $d$-regular tree for the convenience of the reader. The concepts of this section are quite universal and certainly could be applied to more general situations. \\


We recall the Kirszbraun theorem for graphs stated in Theorem~\ref{Kirszbraun}.

\begin{theorem*}[Kirszbraun theorem for graphs]
	Let $\Lambda$ be a connected region of $\Z^m$, $S$ be a subset of $\Lambda$ and $\bar{h}: S \rightarrow \mathcal{T}$ be a graph homomorphism which conserves the parity. There exists a graph homomorphism $h: \Lambda \rightarrow \mathcal{T}$ such that $h = \bar{h}$ on $S$ if and only if for all $x,y$ in $S$ 
	\begin{equation}
	\label{eq:Kirszbraun}
	d_{\mathcal{T}}(\bar{h}(x),\bar{h}(y)) \leq d_{\Lambda}(x,y).
	\end{equation}
\end{theorem*}


\begin{remark} The parity condition in Theorem~\ref{Kirszbraun} is necessary. Indeed, let us consider a situation where \begin{align*}
	d_{\Lambda}(x,y) =2
	\end{align*}
	and
	\begin{align*}
	d_{\mathcal{T}}(\bar{h}(x),\bar{h}(y))=1 .
	\end{align*}
	Then it follows that the condition~\eqref{eq:Kirszbraun} is satisfied but there cannot be an extension~$h$ of~$\bar h$. Indeed, the graph homomorphism~$\bar h$ in this example is violating the parity condition.
\end{remark}

For the proof of Theorem~\ref{Kirszbraun} we need the following observation which states that once the image of a single point is fixed it is always possible to build a graph homomorphism that goes as fast as possible in one direction of a geodesic in the tree.

\begin{lemma}
	\label{maximal_homomorphisms}
	Let $\Lambda$ be a connected region of $\mathbb{Z}^m$, $x$ be a point in $\Lambda$, $w$ be a vertex of $\mathcal{T}$ and $\mathbf g= \{v_0=\omega,v_1,...\}$ be a one-sided geodesic in $\tree$ starting at~$w$. The map $h_x^{w}:\Lambda \rightarrow \mathcal{T}$ given by $h_x^{w}(y)=v_{d_{\Lambda}(x,y)}$ is a graph homomorphism.
\end{lemma}

\begin{proof}
	The function $h_x^{w}$ is $1$-Lipschitz since the graph distance is $1$-Lipschitz. Moreover two neighbors cannot have the same image because for bipartite graphs, the parity of the graph distance to a single point depends on the parity of the vertex. Therefore $h_x^{w}$ is a graph homomorphism.
\end{proof}

For the proof of Theorem~\ref{Kirszbraun} let us introduce the natural analogue of the norm of~$|\cdot|$ on a tree, which we call \emph{depth}.
\begin{definition}[Depth on a tree]\label{d_depth_on_a_tree}
	Let $\mathbf{g}$ be a two-sided geodesic of $\tree$ and let~$\pm \infty_{\mathbf{g}}$ denote the boundary points of~$\mathbf{g}$. The depth associated to the two-sided geodesic~$\mathbf{g}$ is given by the unique function $|\cdot| : \tree \to \mathcal{Z}$ such that $|\mathbf{r}|=0$ and for nearest neighbors~$v \sim w \in \mathcal{T}$
	\begin{align}
	|v|=|w| + \begin{cases} 
	-1 , & \mbox{if $v$ is closer to~$-\infty_{\mathbf{g}}$ than~$w$,} \\
	1 , & \mbox{if $w$ is closer to~$-\infty_{\mathbf{g}}$ than~$v$.} \\
	\end{cases}
	\end{align}
	We want to note that the depth can be negative. On the set~$\mathcal{H}_n^{\mathbf{g}}(s)$ we always consider the depth associated with the two-sided geodesic~$\mathbf{g}$. For all other spaces, either we specify which two-sided geodesic~$\mathbf{g}$ is used for defining~$|\cdot|$, or an a-priori, arbitrary but fixed, two-sided geodesic~$\mathbf{g}$ is used.
\end{definition}

Let us now turn to the proof of Theorem~\ref{Kirszbraun}.

\begin{proof}[Proof of Theorem~\ref{Kirszbraun}]
	The condition \eqref{eq:Kirszbraun} is clearly necessary since a graph homomorphism is $1$-Lipschitz. Suppose now that \eqref{eq:Kirszbraun} holds. In order to prove Theorem \ref{Kirszbraun}, we only need to construct a graph homomorphism ~$h:\Lambda \to \mathcal{T}$ such that $h= \bar h$ on $S$. For that purpose, let us fix an arbitrary two-sided geodesic $\mathbf{g}$ in $\T$ to which we associate the corresponding depth function.
	For $y \in \Lambda$ we define $h(y)$ in the following way: For $x \in S$ let $h_x^{\bar{h}(x)}(y)$ as in Lemma \ref{maximal_homomorphisms} where the one-sided geodesic goes from $\bar{h}(x)$ towards $-\infty_{\mathbf{g}}$ and  find the vertex~$\tilde x \in S$ such that
	\begin{align}\label{e_hdef}
	|h_{\tilde x}^{\bar{h}(\tilde x)}(y)| = \max_{x \in S}|h_x^{\bar{h}(x)}(y)|.
	\end{align}
	Then we set
	\begin{align*}
	h(y) = h_{\tilde x}^{\bar{h}(\tilde x)}(y).
	\end{align*}

	We will show that the function $h: \Lambda \to \mathcal T$ is well defined, that $h=\bar{h}$ on $S$ and that $h$ is a graph homomorphism. We start with deducing that $h$ is well defined. The fact that~$h$ is well defined will follow from the following observation: If there are $y \in \Lambda$ and $x_1,x_2 \in S$ such that $|h_{x_1}^{\bar{h}(x_1)}(y)|=|h_{x_2}^{\bar{h}(x_2)}(y)|$ then 
	\begin{align} \label{e_aux_Kirszbraun_well_defined}
	h_{x_1}^{\bar{h}(x_1)}(y)=h_{x_2}^{\bar{h}(x_2)}(y).
	\end{align}
	Therefore, let us deduce now the statement~\eqref{e_aux_Kirszbraun_well_defined}. We can assume without loss of generality that~$x_1 \neq x_2$. By definition of~$h_{x_i}^{\bar{h}(x_i)}$ it holds
	\[
	d_{\Lambda}(x_i,y)=d_{\T}(\bar{h}(x_i), h_{x_i}^{\bar{h}(x_i)}(y)).
	\]
	Using this fact, the subadditivity of the graph distance and \eqref{eq:Kirszbraun} yields that
	\begin{align} 
	d_{\tree}(\bar{h}(x_1),\bar{h}(x_2)) & \leq d_{\Lambda}(x_1,x_2) \\
	& \leq d_{\Lambda}(x_1,y)+d_{\Lambda}(y,x_2) \\
	& = d_{\T}(\bar{h}(x_1), h_{x_1}^{\bar{h}(x_1)}(y))+d_{\T}(\bar{h}(x_2), h_{x_2}^{\bar{h}(x_2)}(y)). \label{eq:distance}
	\end{align}
	Now, let $v \in \mathcal{T}$ be the unique vertex on the geodesic path from $\bar{h}(x_1)$ to $-\infty_{\mathbf{g}}$ such that
	\begin{equation} 
	\label{eq:midpoint}
	d_{\tree}(\bar{h}(x_1),\bar{h}(x_2)) = d_{\tree}(\bar{h}(x_1),v)+d_{\tree}(v,\bar{h}(x_2)).
	\end{equation}
	Combining \eqref{eq:distance} and \eqref{eq:midpoint} gives 
	\[
	d_{\tree}(\bar{h}(x_1),v)+d_{\tree}(v, \bar{h}(x_2)) \leq d_{\T}(\bar{h}(x_1), h_{x_1}^{\bar{h}(x_1)}(y))+d_{\T}(\bar{h}(x_2), h_{x_2}^{\bar{h}(x_2)}(y)).
	\]
	Thus either 
	\begin{align}\label{e_kirszbraun_well_defineedness_alternative_1}
	d_{\tree}(\bar{h}(x_1),v) \leq d_{\T}(\bar{h}(x_1), h_{x_1}^{\bar{h}(x_1)}(y))
	\end{align}
	or
	\begin{align}\label{e_kirszbraun_well_defineedness_alternative_2}
	d_{\tree}(\bar{h}(x_2),v) \leq d_{\T}(\bar{h}(x_2), h_{x_2}^{\bar{h}(x_2)}(y)).
	\end{align}
	Due to the tree structure, all the vertices $w$ on one geodesic path between $\{\bar{h}(x_i)\}_{i =1,2}$ and $-\infty_{\mathbf{g}}$ and such that $ d_{\tree}(\bar{h}(x_i),v) \leq d_{\tree}(\bar{h}(x_i),w)$ must also be on the other geodesic path. Hence due to~\eqref{e_kirszbraun_well_defineedness_alternative_1} and~\eqref{e_kirszbraun_well_defineedness_alternative_2} at least one of the $h_{x_i}^{\bar{h}(x_i)}(y)$ is on both geodesic paths to~$-\infty_{\mathbf{g}}$. Since the depth is one-to-one on those paths, the only way that $|h_{x_1}^{\bar{h}(x_1)}(y)|=|h_{x_2}^{\bar{h}(x_2)}(y)|$ is if $h_{x_1}^{\bar{h}(x_1)}(y)=h_{x_2}^{\bar{h}(x_2)}(y)$ which deduces~\eqref{e_aux_Kirszbraun_well_defined}.
	\\

	We prove now that $h= \bar{h}$ on $S$.
	For any pair of points $x_1,x_2 \in S$, we know from the definition of the map $h_{x_2}^{\bar h (x_2)} $ that
	\begin{align} \label{e_aux_identity_Kirszbraun}
	d_{\mathcal{T}} (\bar h (x_2), h_{x_2}^{\bar h (x_2)} (x_1) ) = d_{\Lambda}(x_2,x_1).
	\end{align}
	Thus it follows from the definition of the depth that
	\begin{eqnarray}
	|h_{x_2}^{\bar{h}(x_2)}(x_1)| & = & |\bar{h}(x_2)|- d_{\T}(\bar{h}(x_2), h_{x_2}^{\bar{h}(x_2)}(x_1)) \\
	& = & |\bar{h}(x_2)|- d_{\Lambda}(x_2,x_1) \\
	& \leq & |\bar{h}(x_2)|-(|\bar{h}(x_2)|-|\bar{h}(x_1)|) \\
	& \leq & |\bar{h}(x_1)|.
	\end{eqnarray}
	This means that, at the point $x_1$, the maximal argument in~\eqref{e_hdef} must be reached for $x_1$ and thus $h(x_1)=\bar{h}(x_1)$.\\
	
	Now, we will show that~$h$ is a graph homomorphism. For this purpose let~$y \sim z\in \Lambda$ be nearest neighbors. We have to show that this implies~$d_{\mathcal{T}}\left( h(y), h(x) \right) =1$. We distinguish two cases. In the first case we assume that there is~$x \in S$ such that
	\begin{align*}
	h (y)= h_{x}^{\bar h (x)} (y) \quad \mbox{and} \quad h (z)= h_{x}^{\bar h (x)} (z). 
	\end{align*}
	In this case, the fact that~$d_{\mathcal{T}}\left( h(y), h(z) \right) =1$ directly follows from Lemma~\ref{maximal_homomorphisms} which states that all the $h_x^{h(x)}$ are graph homomorphisms themselves. \newline
	Let us now consider the second case where we assume that there exist~$x_1, x_2 \in S$, $x_1 \neq x_2$ such that
	\begin{align*}
	h (y)= h_{x_1}^{\bar h (x_1)} (y) \quad \mbox{and} \quad h (z)= h_{x_2}^{\bar h (x_2)} (z). 
	\end{align*}
	Then, we have to show that
	\begin{align}
	d_{\mathcal{T}} \left(h_{x_1}^{\bar h (x_1)} (y) , h_{x_2}^{\bar h (x_2)} (z)\right) =1.
	\end{align}
	In this situation we will show below that
	either
	\begin{align}\label{e_Kirszbraun_Homo_aux_1}
	h_{x_1}^{\bar h (x_1)} (y)=h_{x_2}^{\bar h (x_2)} (y) \quad \mbox{or} \quad h_{x_1}^{\bar h (x_1)} (z)=h_{x_2}^{\bar h (x_2)} (z).
	\end{align}
	Indeed, if this statement is true we have reduced the second case to the first case and the argument is complete.\newline 
	Let us now deduce the statement~\eqref{e_Kirszbraun_Homo_aux_1}. By the definition of the map~$h$ it must hold that
	\begin{align*}
	\left| |h_{x_1}^{\bar{h} (x_1)} (z)| - |h_{x_2}^{\bar{h} (x_2)} (y)| \right| \leq 2,
	\end{align*}
	else one could easily construct a contradiction. The case
	\begin{align*}
	|h_{x_1}^{\bar{h} (x_1)} (z)| = |h_{x_2}^{\bar{h} (x_2)} (y)| 
	\end{align*}
	cannot happen. Else one would get a contradiction to the fact that~$\bar h$ conserves the parity. Hence it follows that
	\begin{align*}
	|h_{x_1}^{\bar h (x_1)} (z)| = |h_{x_2}^{\bar h (x_2)} (y)| \pm 1. 
	\end{align*}
	Hence we have that either 
	\begin{align*}
	|h_{x_1}^{\bar h (x_1)} (z)| = |h_{x_2}^{\bar h (x_2)} (z)|
	\end{align*}
	or
	\begin{align*}
	|h_{x_1}^{\bar h (x_1)} (y)| = |h_{x_2}^{\bar h (x_2)} (y)|.
	\end{align*}
	Now, we are in the same situation as when arguing that~$h$ is well defined, and can deduce the in the same way the desired statement (see~\eqref{e_aux_Kirszbraun_well_defined}).
\end{proof}
We will now prove an important corollary of the Kirszbraun theorem for graphs.

\begin{corollary}\label{p_corollary_kirszbraun}
	Let $\Lambda$ be a simply connected region of $\Z^m$ and let $h: \Lambda \to \mathcal{T}$ and $g: \Lambda \to \mathcal{T}$ be two graph homomorphisms. 
	Then there exists a graph homomorphism $\tilde{g}: \Lambda \to \mathcal{T}$ such that for all $x \in \partial \Lambda: \tilde{g}(x)=g(x)$ and
	\begin{equation}\label{e_boundary_bound}
	\sup_{x \in \Lambda}d_{\mathcal{T}}(\tilde{g}(x),h(x)) \leq \sup_{x \in \partial \Lambda}d_{\mathcal{T}}(g(x),h(x)) 
	\end{equation}
\end{corollary}

\begin{proof}[Proof of Corollary~\ref{p_corollary_kirszbraun}]
	We argue by contradiction. Let us set $k=\sup_{x \in \partial \Lambda}d_{\mathcal{T}}(g(x),h(x))$ and suppose that~\eqref{e_boundary_bound} does not hold. We consider a graph homomorphism $\tilde{g}: \Lambda \to \mathcal{T}$ which maximizes the number of vertices in $\Lambda$ for which
	$
	d_{\mathcal{T}}(\tilde{g}(x),h(x)) \leq k
	$
	among all graph homomorphisms which are equal to $g$ on $\partial \Lambda$. Since~\eqref{e_boundary_bound} does not hold we know that there exist at least one maximal connected component $\mathcal{C} \subset \Lambda$ on which
	$
	d_{\mathcal{T}}(\tilde{g}(x),h(x)) > k.
	$
	Moreover since $\mathcal{C}$ is maximal we also know that for all $x \in \partial^{out} \mathcal{C}$ the outer boundary of $\mathcal{C}$:
	$
	d_{\mathcal{T}}(\tilde{g}(x),h(x)) \leq k.
	$
	Now choose a point $x_0 \in \mathcal{C}$. Since a graph homomorphism is $1$-Lipschitz we know that
	\begin{equation}
	h(x_0) \in \cap_{x \in \partial^{out} C} \, \mathcal{B}_{\mathcal{T}}(h(x),d_{\Lambda}(x,x_0)),
	\end{equation}
	where $\mathcal{B}_{\mathcal{T}}(y,r)$ denote the ball of radius $r$ around $y$ in $\mathcal{T}$. In particular we also have that
	\[
	h(x_0) \in \cap_{x \in \partial^{out} C} \, \mathcal{B}_{\mathcal{T}}(\tilde{g}(x),d_{\Lambda}(x,x_0)+k).
	\]
	We also notice that $\cap_{x \in \partial^{out} C}\mathcal{B}_{\mathcal{T}}(\tilde{g}(x),d_{\Lambda}(x,x_0))$ necessarily is non-empty since it must contain $\tilde{g}(x_0)$. Combining the last two observations we obtain that 
	$$d_{\mathcal{T}}(h(x_0),\cap_{x \in \partial^{out} C}\mathcal{B}_{\mathcal{T}}(\tilde{g}(x),d_{\Lambda}(x,x_0)) ) \leq k.$$
	 Now, we choose any vertex $v \in \mathcal{T}$ within distance $k$ of $h(x_0)$ in the set $\cap_{x \in \partial^{out} C} \, \mathcal{B}_{\mathcal{T}}\left(\tilde{g}(x),d_{\Lambda}(x,x_0)\right)$. The Kirszbraun theorem yields that there exist a graph homomorphism $\bar{g}: \partial^{out} \mathcal{C} \cup \mathcal{C} \to \mathcal{T} $ such that $\bar{g}=\tilde{g}$ on $\partial^{out} \mathcal{C}$ and $\bar{g}(x_0)=v$. This contradicts the assumption that $\tilde{g}$ maximizes the number of vertices in $\Lambda$ for which
	$
	d_{\mathcal{T}}(\tilde{g}(x),h(x)) \leq k
	$
	and finishes our proof.
\end{proof}

\subsection{A concentration inequality}\label{s_concentration}
The main purpose of this section is to deduce a concentration inequality stated in Theorem \ref{main_concentration_section_3} for the uniform measure on the set $\mathcal{H}^\mathbf{g}_n(s)$, which is defined in Definition~\ref{e_def_micro_surface_tension}. We recall the statement of the concentration inequality:
\begin{theorem*}[Concentration inequality]
	There exists universal constants $C$ and $c$ such that, under the uniform measure~$\mathbb{P}_n^{s}$ on $\mathcal{H}^\mathbf{g}_n(s)$, for all $x=(x_1,..,x_m)$ in $\Z^m$, and for all~$\varepsilon \geq n^{-0.45}$ we have
	\begin{align}
	\mathbb{P}_n^{s} \left( \max_{x \in \mathbb{Z}^m} d_{\mathcal{T}}(h(x),\mathbf{g}(\lfloor s \cdot x \rfloor)) \geq \varepsilon n \right) \leq Ce^{-c\varepsilon^2 n},
	\end{align}
	where~$\mathbf{g}$ is a two-sided geodesic and the function~$\mathbf{g}(n)$ is given by Convention~\ref{e_d_point_on_geodesic}.
\end{theorem*}

We will deduce Theorem~\ref{main_concentration_section_3} from an auxiliary concentration inequality.


\begin{definition}
To each $x=(x_1,..,x_m) \in \Z^m$ we associate a point $0_x$ defined by
 \begin{align}
 \label{e_0x}
 0_x := ( \lfloor \frac{x_1}{n} \rfloor n,.. , \lfloor \frac{x_m}{n} \rfloor n).
 \end{align}
 The point $0_x$ is the lower left corner of the box containing $x$ when partitioning $\Z^m$ into boxes of size $n$.
	
	For any~$h \in \mathcal{H}^\mathbf{g}_n(s)$ and~$x \in \mathbb{Z}^m$ we define
	\begin{align}
	\label{e_d_norm_0}
	|h(x)|_0 := |h(x)| - |h(0_x)|,
	\end{align}
	where~$|\cdot|$ denotes the depth defined in Definition~\ref{d_depth_on_a_tree}. 
\end{definition}

\begin{remark}
	 The $n$-translation invariance of $h$ imposes that for all $k \leq m$:
	\begin{align}\label{e_ntranslation}
	|h(x+n \vec{i}_k)|_0 &= |h(x+n \vec{i}_k)|-|h(0_{x+n \vec{i}_k})| \\
	&= (|h(x)|+s_in)-(|h(0_x)|+s_in) \\
	& = |h(x)|-|h(0_x)|= |h(x)|_0 .
	\end{align}
\end{remark}

Let us now state our auxiliary concentration inequality. 

\begin{lemma}[Auxiliary concentration inequality]\label{p_azuma_distance}
	Let $S_n=\left\{ 0,1 \ldots, n-1 \right \}^m$. Then there exist a universal constants $C$ and $c$ such that for all~$\varepsilon>0$, $n \in \N$ and~$x \in S_n$ it holds
	\begin{align}
	\mathbb{P}_n^{s} \left( \left| |h(x)|_0-\mathbb{E}_n^{s}[|h(x)|_0] \right| \geq \varepsilon |x|_{\ell_1} \right) \leq C e^{-c\varepsilon^2 |x|_{\ell_1}}.
	\end{align}
	where $\mathbb{E}_n^{s}$ is the expectation with respect to $\mathbb{P}_n^{s}$.
\end{lemma}

Since the distance between $x$ and the origin is bounded from above by $mn$ in $S_n$, we deduce directly from Lemma \ref{p_azuma_distance} and \eqref{e_ntranslation} the following corollary:

\begin{corollary}\label{p_aux_concentration}
There exist a universal constants $C$ and $c$ such that for all~$\varepsilon>0$ and~$x \in \mathbb{Z}^m$ it holds
\begin{align}
\mathbb{P}_n^{s} \left( \left| |h(x)|_0-\mathbb{E}_n^{s}[|h(x)|_0] \right| \geq \varepsilon n \right) \leq C e^{-c\varepsilon^2 n}.
\end{align}
\end{corollary}

The proof of Lemma~\ref{p_azuma_distance} is in several steps. In~\cite{CEP96} a concentration inequality was deduced for domino tilings of an Aztec diamond. Using this argument as an inspiration, our argument is also based on the Azuma-Hoeffding inequality (see Lemma~\ref{p_Azuma_martingales} from below). In the setting of tree valued graph homomorphisms, verifying the assumptions of the Azuma-Hoeffding inequality becomes very challenging. For this purpose we developed a completely new argument based on coupling. We state the proof of Lemma~\ref{p_azuma_distance} with all the details in Section~\ref{s_aux_concentration_proof} and now continue to state the proof of Theorem~\ref{main_concentration_section_3}. The only additional ingredient that is needed in the proof of Theorem~\ref{main_concentration_section_3} is the fact that every element~$h \in \mathcal{H}_n^{\mathbf{g}}(s)$ has to stay close to the two-sided geodesic~$\mathbf{g}$ (see Lemma~\ref{p_closeness_translation_invariant_geodesic}).

	\begin{proof}[Proof of Theorem~\ref{main_concentration_section_3}]
		For notational simplicity, we write~$\mathbb{P}$ and~$\mathbb{E}$ instead of~$\mathbb{P}_n^{s}$ and~$\mathbb{E}_n^{s}$. We recall that by definition
		$$|h(x)|_0=|h(x)|-|h(0_x)|.$$
		
	The goal is to derive the estimate~\eqref{e_main_concentration}, which is verified in several steps.

		 Let $x=(x_1,..,x_m) \in \Z^m$. The auxiliary concentration inequality of Corollary~\ref{p_aux_concentration} states that
		\begin{align}\label{e_proof_concentration_azuma_bound}
		\mathbb{P} \left( \left| |h(x)|_0-\mathbb{E}[|h(x)|_0] \right| \geq \varepsilon n \right) \leq C e^{-c\varepsilon^2 n},
		\end{align}
		where we used the simplified notation~$\mathbb{P}=\mathbb{P}_n^{s}$ and $\mathbb{E}=\mathbb{E}_n^{s}$. 
		In order to extend this estimate to the whole space $\Z^m$ we observe that \eqref{e_ntranslation} implies:
%
		\begin{align}\label{e_estimate_max2}
		\mathbb{P}& (\max_{x\in \Z^m}||h(x)|_0-\mathbb{E}[|h(x)|_0]| \geq \varepsilon n ) \\
		& = \mathbb{P} (\max_{x\in S_n}||h(x)|_0-\mathbb{E}[|h(x)|_0]|) \geq \varepsilon n) \\
		& \leq n^m \max_{x \in S_n} \mathbb{P} (||h(x)|_0-\mathbb{E}[|h(x)|_0]|) \geq \varepsilon n) \\
		& \leq C \ n^m e^{-c\varepsilon^2 n} \\
		& \leq C \ e^{-c\varepsilon^2 n}.
		\end{align}
		In the next step, we show that the estimate~\eqref{e_estimate_max2} yields 
		\begin{align}
		\mathbb{P}(d_\mathcal{T}(h(x), \Pi_{\mathbf{g}} (h(x) )) \geq \varepsilon n) \leq Ce^{-c\varepsilon^2 n}. \label{e_concentration_d_h_g2}
		\end{align}
		Notice that $|h(0+n \vec{i}_l)|-|h(0)|=s_l \cdot n$ almost surely. Since the uniform measure on $\mathcal{H}^\mathbf{g}_n(s)$ is translation invariant for the gradient, this implies that for all $l\leq m$:
		\begin{eqnarray}\label{e_expectation}
		\mathbb{E}[|h( 0+\vec{i}_l)|_0] &= & \mathbb{E}[|h(0+ \vec{i}_l)|-|h(0)|] \\
		& = & 	\frac{1}{n} \sum^{n-1}_{l= 0} \mathbb{E}[|h(0+(l+1) \vec{i}_l)|-|h(0+l \vec{i}_l)|] \\
		& = & \frac{1}{n} \mathbb{E}[|h(0+n \vec{i}_l)|-|h(0)|]= \frac{1}{n}s_l x.\\
		\end{eqnarray}
		Using the same reasoning we also obtain that $\mathbb{E}[|h(x)|_0]=s \cdot x$.
		
		Assume wlog.~that $s_1 \geq 0$ and suppose that 	$d_{\tree}(h(x),\mathbf{g}( \lfloor s \cdot x \rfloor)) \geq d$.\newline
Since all configurations are supported on the geodesic $\mathbf{g}$, there exist $k > 0$ such that $h(x+k\vec{i}_1)$ is on $\mathbf{g}$. If $x_1+k \leq n$, this implies
		$$|h(x+k \vec{i}_1)|_0 \leq |h(x)|_0-d.$$
		Moreover, since $$ \mathbb{E}[|h(x+k \vec{i}_1)|_0]- \mathbb{E}[|h(x)|_0]=ks_1 \geq 0,$$ we obtain that
		\begin{eqnarray}
		d & \leq & |h(x)|_0 - |h(x+ \vec{i}_1)| \\
		& = & |h(x)|_0 - \mathbb{E}[|h(x)|_0] + \mathbb{E}[|h(x)|_0] - \mathbb{E}[|h(x+k \vec{i}_1)|_0] \\
		& & + \mathbb{E}[|h(x+k \vec{i}_1)|_0] - |h(x+k \vec{i}_1)|_0 \\
		& \leq & \left| |h(x)|_0 - \mathbb{E}[|h(x)|_0] \right|+ \left| |h(x+k \vec{i}_1)|_0 - \mathbb{E}[|h(x+k \vec{i}_1)|_0] \right|.\\ 
		\end{eqnarray}
		It follows that either
		$$||h(x+k \vec{i}_1)|_0-\mathbb{E}[|h(x+k \vec{i}_1)|_0]| \geq d/2$$
		or 
		$$||h(x)|_0-\mathbb{E}[|h(x)|_0]| \geq d/2.$$
		Combining this with the estimate~\eqref{e_estimate_max2} yields the desired estimate~\eqref{e_concentration_d_h_g2}.
		If $x_1+k > n$, we can simply look at the auxiliary function defined by
		$$f(x+k \vec{i}_1):=|h(x)|-|h(0)|.$$
		The following identity holds almost surely
		\begin{eqnarray}
		f(x+n \vec{i}_k)-\mathbb{E}[|f(x+n \vec{i}_k)|_0] & = & |h(x)|_0+s_in-\mathbb{E}[|h(x)|_0+s_in] \\
		& = & |h(x)|_0-\mathbb{E}[|h(x)|_0]
		\end{eqnarray}
	
		and $ \mathbb{E}[f(x+k \vec{i}_1)]- \mathbb{E}[|h(x)|_0]=ks_1 \geq 0$. Hence the previous argument still applies if we replace $|h(x+k \vec{i}_1)|_0$ by $f(x+k \vec{i}_1)$.

		We can now prove that 
		\begin{align}
		\mathbb{P} \left( d_{\tree}(h(x),\mathbf{g}( \lfloor s \cdot x \rfloor)) \geq \varepsilon n \right) \leq C e^{-c\varepsilon^2 n}. \label{e_main_concentration_preliminary2}
		\end{align}
		Notice that by definition
		\begin{eqnarray}
			|h(x)|_0-s \cdot x & = & d_{\mathcal{T}}(h(x), \Pi_{\mathbf{g}} (h(x) ))-d_{\mathcal{T}}(h(0), \Pi_{\mathbf{g}} (h(0) )) \\
			& & + \left(|\Pi_{\mathbf{g}} (h(x))|-|\mathbf{g} ( \lfloor s \cdot x \rfloor)|\right)
		\end{eqnarray}
		
		and since $||\Pi_{\mathbf{g}} (h(x))|-|\mathbf{g} (\lfloor s \cdot x \rfloor )||= d_{\mathcal{T}}(\mathbf{g}( \lfloor s \cdot x \rfloor), \Pi_{\mathbf{g}} (h(x) ))$, notice also from the triangle inequality that

		\[
		d_{\mathcal{T}}(\mathbf{g}( \lfloor s \cdot x \rfloor), \Pi_{\mathbf{g}} (h(x) ))+d_{\mathcal{T}}(h(x), \Pi_{\mathbf{g}} (h(x) )) \geq d_{\mathcal{T}}(\mathbf{g}( \lfloor s \cdot x \rfloor), h(x)).
		\]
		Combining the last two inequalities and using the triangle inequality again we obtain
			\begin{eqnarray}
			d_{\mathcal{T}}(h(x), \Pi_{\mathbf{g}} (h(x) )) & \leq & ||h(x)|_0-s \cdot x |+2d_{\mathcal{T}}(h(x), \Pi_{\mathbf{g}} (h(x) )) \\
			& & +d_{\mathcal{T}}(h(0), \Pi_{\mathbf{g}} (h(0) )).
		\end{eqnarray}
		As a consequence
		\begin{eqnarray}
		\left\{ d_{\tree}(h(x),\mathbf{g}( \lfloor s \cdot x \rfloor)) \geq \varepsilon n \right\} & \subset & \left\{d_\mathcal{T}(h(0), \Pi_{\mathbf{g}} (h(0) )) \geq \frac{\varepsilon n}{4} \right\} \\
		& & \cup \left\{||h(x)|_0-s \cdot x| \geq \frac{ \varepsilon n }{4} \right\} \\
		& & \cup \left\{d_\mathcal{T}(h(x), \Pi_{\mathbf{g}} (h(x) )) \geq \frac{\varepsilon n}{4} \right\}.\\
		\end{eqnarray}
		Hence we can apply estimates \eqref{e_concentration_d_h_g2} and \eqref{e_estimate_max2} to obtain that:
		\begin{eqnarray}
		\mathbb{P} \left( d_{\tree}(h(x),\mathbf{g}( \lfloor s \cdot x \rfloor)) \geq \varepsilon n \right) & \leq & \mathbb{P} \left( ||h(x)|_0-s \cdot x| \geq \frac{ \varepsilon n }{4} \right) \\
		& & + \mathbb{P} \left( d_\mathcal{T}(h(x), \Pi_{\mathbf{g}} (h(x) )) \geq \frac{\varepsilon n}{4} \right) \\
		& & + \mathbb{P} \left( d_\mathcal{T}(h(0), \Pi_{\mathbf{g}} (h(0) )) \geq \frac{\varepsilon n}{4} \right) \\
		& \leq & Ce^{-c\varepsilon^2 n}+Ce^{-c\varepsilon^2 n}+Ce^{-c\varepsilon^2 n} \\
		& \leq & Ce^{-c\varepsilon^2 n}. \\
		\end{eqnarray}
		Now, the estimate~\eqref{e_main_concentration} follows easily from~\eqref{e_main_concentration_preliminary2} in the following way
		\begin{align}
		\mathbb{P} \left( \max_{x \in \mathbb{Z}^m} d_{\tree}(h(x),\mathbf{g}( \lfloor s \cdot x \rfloor) \geq \varepsilon n \right) & = \mathbb{P} \left( \max_{x \in S_n} d_{\tree}(h(x),\mathbf{g}( \lfloor s \cdot x \rfloor) \geq \varepsilon n \right) \\
		& \leq n^m\max_{x \in S_n} \mathbb{P} \left( d_{\tree}(h(x),\mathbf{g}( \lfloor s \cdot x \rfloor)) \geq \varepsilon n \right) \\
		& \leq n^m 2C e^{-c\varepsilon^2 n} \\
		& \leq C e^{-c\varepsilon^2 n}.
		\end{align}
		which finishes our proof.
	\end{proof}

\subsubsection{Proof of Lemma~\ref{p_azuma_distance}}\label{s_aux_concentration_proof}

We will derive the statement of Lemma~\ref{p_azuma_distance} from the well-known Azuma-Hoeffding inequality.

\begin{lemma}[Azuma-Hoeffding~\cite{Azu67,Hoe63}]\label{p_Azuma_martingales}
	Suppose that~$(M_k)_{\mathbb{N}}$ is a martingale and
	\begin{align}
	|M_{k} - M_{k-1}|< c_k \qquad \qquad \mbox{almost surely.}
	\end{align}
	Then for all~$N \in \mathbb{N}$ and all~$\varepsilon > 0$
	\begin{align}
	\mathbb{P} \left( |M_N - M_0| \geq t \right) \leq 2 \exp \left( - \frac{t^2}{2 \sum_{k=1}^N c_k^2} \right) .
	\end{align}
\end{lemma}
In order to apply Azuma-Hoeffding we have to specify which martingale~$M_k$ we are considering. For this let us first introduce the filtration of sigma algebras~$\mathcal{F}_k$ we are using. For a given $x=(x_1,..x_m) \in S_n$, we define a canonical geodesic path $\mathbf{p}_x$ between $0$ and $x$ by:
	\begin{align}\label{e_canonical_path}
	\mathbf{p}_x=\{p_0=0,..,p_{x_1}=(x_1,..,0),p_{x_1+1}=(x_1,1,..,0) ,..,p_{|x|_{\ell_1}}=x\}.
	\end{align}
Informally $\mathbf{p}_x$ is the unique geodesic path from $0$ to $x$ which is increasing in the total order of $\Z^m$ given by the coordinates.
We define the functions~$g_k : \mathcal{H}_n^{\mathbf{g}} \to \mathbb{R}$ for~$k=0$ via
\begin{align}
g_0 (h) = |h(0)|
\end{align}
and for~$k \geq 1$ via 
\begin{align}
g_k (h) = |h(p_k)| - |h(p_{k-1})|.
\end{align}
Then, the sigma algebras~$\mathcal{F}_k$ are defined via
\begin{align}
\mathcal{F}_k = \sigma (g_l , \ 0 \leq l \leq k).
\end{align} 
Now, we define the martingale~$M_k$,~$0 \leq k \leq |x|_{\ell_1}$ in the usual way using conditional expectations i.e.
\begin{align}\label{e_azuma_def_X_k}
M_k = \mathbb{E}_n^{s} \left[ |h(x)|_0| \mathcal{F}_k\right],
\end{align}
where~$\mathbb{E}_n^{s}$ denotes the expectation under the uniform probability measure on~$\mathcal{H}_{n}^{\mathbf{g}}(s)$. We note that for~$k=|x|_{\ell_1}$ it holds
\begin{align}
M_{|x|_{\ell_1}} = |h(x)|_0.
\end{align}
As a consequence, the statement of Lemma~\ref{p_azuma_distance} follows directly from Azuma-Hoeffding by choosing~$N= |x|_{\ell_1}$ (see Lemma~\ref{p_Azuma_martingales}), if we can show that almost surely
\begin{align}
|M_k - M_{k-1}|\leq 4. 
\end{align}
This is exactly the statement of the following lemma.

\begin{lemma}\label{p_azuma_hypo_verification}
	Using the definitions from above, it holds for~$k=1 \ldots, |x|_{\ell_1}$ that almost surely
	\begin{align}\label{e_azuma_hoeffding_crucial_ingredient}
	|M_k - M_{k-1}| & =
	\left| \mathbb{E}_n^{s} \left[|h(x)|_0 | \mathcal{F}_{k} \right] - \mathbb{E}_n^{s} \left[|h(x)|_0 | \mathcal{F}_{k-1} \right]\right| \leq 4.
	\end{align}
\end{lemma}
Hence, we see that in order to deduce Lemma~\ref{p_aux_concentration} it is only left to verify Lemma~\ref{p_azuma_hypo_verification}. For deducing the estimate~\eqref{e_azuma_hoeffding_crucial_ingredient}, we need to show that changing the depth of a single point $y$ does not influence too much the expected depth of the point $x$. The proof of Lemma~\ref{p_azuma_hypo_verification} is quite subtle. The reason is that the structure of the uniform probability measure~$\mathbb{P}$ on the set~$\mathcal{H}_n^{\mathbf{g}}(s)$ is extremely hard to break down. Unfortunately, without classical tools derived from the model being one-dimensional, like for example stochastic monotonicity or the FKG inequality (see also~\cite{CEP96}), it seems that there is no direct way to compare $ \mathbb{E}_n^{s}[\cdot| \mathcal{F}_{k+1}]]$ and $ \mathbb{E}_n^{s}[\cdot| \mathcal{F}_{k}]]$.\\ 

Instead, we use a dynamical approach. We construct a coupled discrete time Markov chain~$(X_t, Y_t)$ on the product space~$\mathcal{H}^{\mathbf{g},free}_n(s) \times \mathcal{H}^{\mathbf{g},free}_n(s)$ such that for all $x \in \Z^m$ the law of~$|X_t(x)|_0$ converges to~$\mathbb{E}_n^{s}[|h(x)|_0| \mathcal{F}_{k+1}]]$ and the law of~$|Y_t(x)|_0$ converges to $ \mathbb{E}_n^{s}[|h(x)|_0| \mathcal{F}_{k}]]$. The crucial property will be that the Markov chain keeps the depth deviation of~$X_t$ and~$Y_t$ invariant. This means that if~
\begin{align}
\left||X_t(x)| - |Y_t(x)|\right| \leq 2
\end{align}
then also 
\begin{align}\label{e_depth_conservation}
\left| |X_{t+1}(x)| - |Y_{n+1}(x)| \right| \leq 2.
\end{align}
Because this property is verified at each step by going to the limit also holds that
\begin{align}
\left| \mathbb{E}_n^{s}[|h(x)|_0| \mathcal{F}_{k+1}]] - \mathbb{E}_n^{s}[|h(x)|_0| \mathcal{F}_{k}]] \right| \leq 4,
\end{align}
which is the desired statement of Lemma~\ref{p_azuma_hypo_verification}.\\

We will now explain how to build the Markov chain necessary to derive Lemma \ref{p_azuma_hypo_verification}. We decide to build the Markov chain on the set $\mathcal{H}^{\mathbf{g},free}_n(s)$ rather than $\mathcal{H}^\mathbf{g}_n(s)$ because it is easier to build a Markov chain which leaves the set $\mathcal{H}^{\mathbf{g},free}_n(s)$ invariant at each step. Once the Markov chain on $\mathcal{H}^{\mathbf{g},free}_n(s)$ is defined, it is easy to obtain a Markov chain on~$\mathcal{H}^\mathbf{g}_n(s)$ by shifting the homomorphisms to the right level on the geodesic. We describe precisely this operation later in this section.\\


There is a natural choice for the Markov chain on~$\mathcal{H}^{\mathbf{g},free}_n(s)$ such that the projected Markov chain on~$\mathcal{H}^\mathbf{g}_n(s)$ would have the uniform measure as its invariant measure. We call it the Glauber dynamics (see Definition~\ref{d_original_Glauber_dynamic} below). Unfortunately, even when coupled, the Glauber dynamics does not conserve the depth deviation. More precisely, the Glauber dynamics does not have the desired property~\eqref{e_depth_conservation}. This can be seen by constructing counter examples involving fake local minima (cf.~Definition~\ref{d_extrema_graph_homo} and the proof of Lemma~\ref{p_azuma_hypo_verification} from below). If one would consider graph homomorphisms to~$\mathbb{Z}$ then this technical problem would not appear and one could use the classical Glauber dynamics.\\

We overcome this technical difficulty in the following way: We carefully analyze the situations in which the depth deviation can increase under the Glauber dynamics. Then we add a resampling step before every Glauber step to prevent those situations. When adding the resampling step, one has to be very careful not to introduce a bias. We will show that this is not the case. 

We begin with introducing some necessary definitions.
\begin{definition}[Local minimum and extremum of a graph homomorphism]\label{d_extrema_graph_homo}
	Let~$h: \mathbb{Z}^m \to \mathcal{T}$ be a graph homomorphism. We say that a point $x\in \mathbb{Z}^m$ is a local minimum for $h$ if
	\begin{align*}
	|h(x)| < |h(z)| \quad \mbox{for all } z \sim x. 
	\end{align*} 
	If additionally
	\begin{align*}
	h(z) = h(\tilde z) \qquad \mbox{for all } z,\tilde z \sim x
	\end{align*}
	we say that $x$ is a true local minimum, otherwise we say that $x$ is a fake minimum for $h$.
	Finally, if 
	\begin{align*}
	h(z) = h(\tilde z) \qquad \mbox{for all } z,\tilde z \sim x
	\end{align*}
	holds without the first condition on the depth, we say that $x$ is a local extremum of $h$.
\end{definition}
\begin{figure}
	
	
	
	
	
	
	\includegraphics[width=0.8\textwidth]{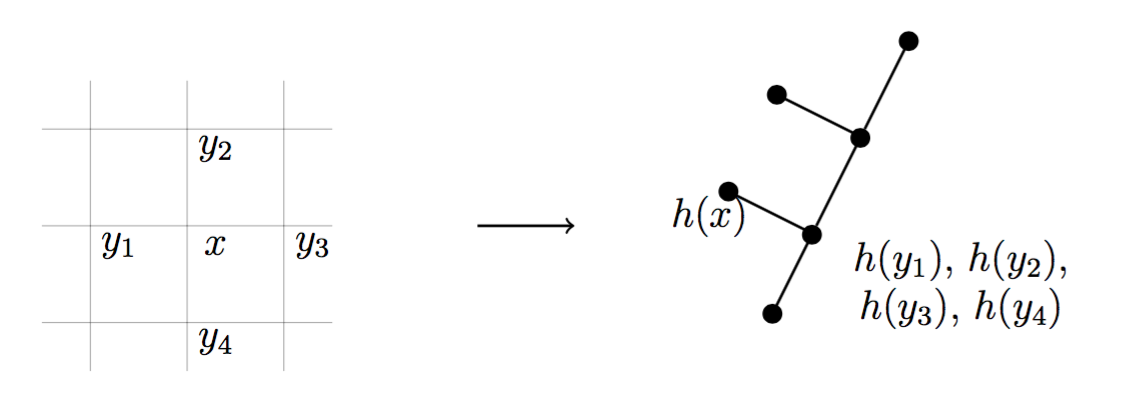}

	\caption{Extremum of a graph homomorphism. The $h(y_i)$'s are all equal.}\label{f_extremum_graph_homo}
\end{figure}

\begin{figure}
	
	
	
	
	
	
	\includegraphics[width=0.8\textwidth]{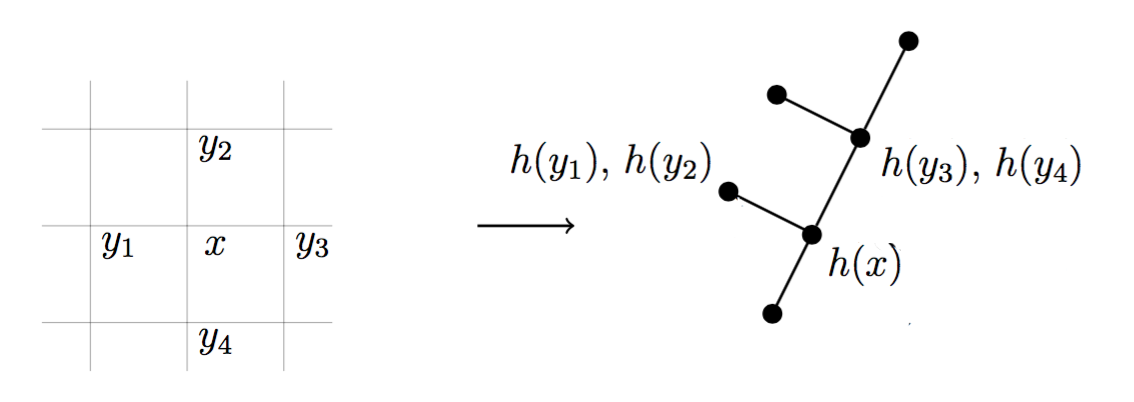}
	
	\caption{Fake local minimum a graph homomorphism}\label{f_fake_extremum_graph_homo}
\end{figure}

\begin{remark}\label{r_fake_local_maxima_do_not_exist}
	We want to observe that there is no fake local maximum for a graph homomorphism~$h$. The reason is that for a given point~$h(x)\in \tree$ there is only one neighbor~$h(x) \sim v \in \tree$ with lower depth.
\end{remark}

\begin{definition}(Pivoting)
	If $x$ is a local extremum of $h$, we call pivoting $x$ the following operation:
	\begin{itemize}
		\item For $z \sim x$ choose $w\in \tree$ randomly and with equal probability among the neighbors of $h(z)\in \tree$.
		\item Set $h(x)=w$.
	\end{itemize}
\end{definition}
In Figure~\ref{f_pivoting_graph_homo} we illustrate what pivoting means for a graph homomorphism to a tree. We will use the operation of pivoting to define a natural Glauber dynamics on~$\mathcal{H}^{\mathbf{g},free}_n(s)$. 

\begin{figure}


	
	
	
	
	
	
	
	
	\includegraphics[width=0.8\textwidth]{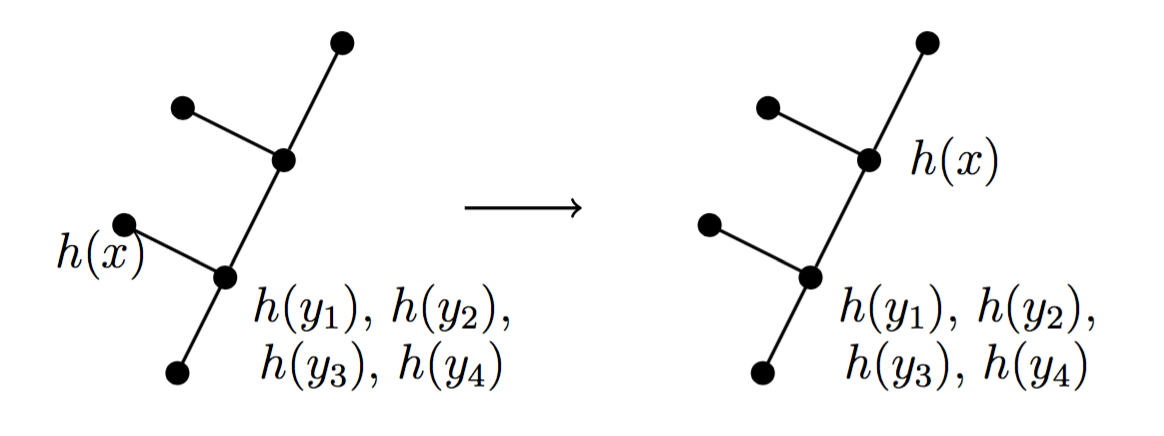}

	\caption{Example of pivoting the extremum x of a graph homomorphism}\label{f_pivoting_graph_homo}
\end{figure}

\begin{definition}(Glauber dynamic)\label{d_original_Glauber_dynamic}
	The Glauber dynamics on~$\mathcal{H}^{\mathbf{g},free}_n(s)$ is a discrete time Markov chain given by the following procedure: Let~$X_t \in \mathcal{H}^{\mathbf{g},free}_n(s)$ then~$X_{t+1}$ is attained via:
	\begin{enumerate}
		\item Choose a vertex~$x \in S_n \subset \mathbb{Z}^m$ uniformly at random.
		\item If~$X_t(x)$ is not a local extremum in the sense of Definition~\ref{d_extrema_graph_homo} do nothing and set~$X_{t+1}=X_t$.
		\item If~$X_t(x)$ is a local extremum in the sense of Definition~\ref{d_extrema_graph_homo} then pivot~$X_t$ around~$x$.
	\end{enumerate}
\end{definition}

\begin{figure}[h]
	\centering
	\includegraphics[width=0.6\textwidth]{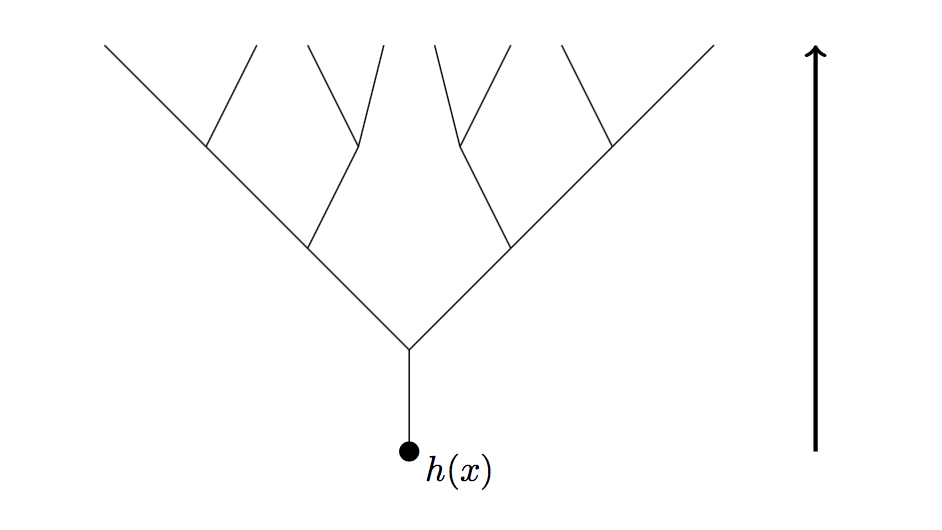}
	\caption{Complete branch of a 3-regular tree. The arrow indicates the direction in which the depth increases.}\label{f_branch}
\end{figure}

As we mentioned above, it is possible that the original Glauber dynamics increases the depth deviation. Let us explain how this is possible. Consider two graph homomorphisms~$h, g \in \mathcal{H}^{\mathbf{g},free}_n(s)$ such that:
\begin{itemize}
	\item $h$ has a true local minimum at~$x$,
	\item $g$ has a fake local minimum at~$x$ and
	\item $| h (x) | = |g (x)| + 2$.
\end{itemize}
Now, it might happen that the Glauber dynamics pivots~$h$ around~$x$ and increases~$|h(x)|$ by~$2$. Because~$g$ has a fake local minimum at~$x$, one cannot pivot the configuration~$g$ around~$x$. Therefore the depth~$|g (x)|$ stays the same. So overall, the depth difference at~$x$ increased to~$4$. From this example, one also understands that this cannot happen for graph homomorphisms to~$\mathbb{Z}$. The reason is that graph homomorphism to~$\mathbb{Z}$ cannot have fake local minima.
The purpose of the adapted Glauber dynamics is to eliminate this situation. In order to define the adapted Glauber dynamics, let us first introduce the concepts of excursion and resampling an excursion. In the reminder of this section the two-sided geodesic $\mathbf{g}$ and associated depth function are both fixed.

\begin{definition}[Complete Branch]
	Let $v_0$ be a vertex in $\tree$. We say that $\tree_0$ is a complete branch of $\tree$ with root $v_0$ (see Figure \ref{f_branch}) if and only if:
	\begin{itemize}
		\item $v_0 \in \tree_0$.
		\item There exist a unique neighbor $v$ of $v_0$ such that $|v|=|v_0|+1$ and $v \in \tree_0$.
		\item For all $v \in \tree_0 \setminus v_0: \left\{ v'\sim v \implies v' \in \tree_0 \right\}$.
	\end{itemize}
\end{definition}

\begin{definition}[Excursion]\label{d_excursion}
	Let $h: \Z^m \rightarrow \tree$ be a translation invariant homomorphism and $\tree_0$ be a complete branch of $\tree$. An excursion of $h$ is a connected component~$\mathcal{C}$ of $h^{-1}(\tree_0)$ on which the depth is bounded from above. If~$x \in \mathbb{Z}^m$ is mapped by~$h(x)$ onto the root of~$\tree_0$ then we say that the excursion starts at~$x$.
	
	We also will say that an edge $e_{xy}$ is in the excursion $\mathcal C$ if both vertices $x$ and $y$ are in $\mathcal C$. If additionally either $x$ or $y$ is mapped to the root of $\mathcal{C}$ we say that $e_{xy}$ is on the boundary of $\mathcal{C}$.
\end{definition}
We call~$\mathcal{C}$ excursion because for any element~$x \in \mathcal{C}$ and path~$\mathbf{p} \subset \mathbb{Z}^m$ from~$x$ to infinity, there must be an element~$z \in \mathbf{p}$ such that~$h(z)$ is mapped onto the root of the branch~$\tree_0$ (see Figure~\ref{f_excursion}). Later, we will need the following auxiliary statement about excursions.
\begin{lemma}\label{p_crucial_lemma_logic_statement_3}
 Assume that $x$ is a local minimum for a graph homomorphism $h \in \mathcal{H}^{\mathbf{g},free}_n(s)$ and two edges $e_{xy}$ and $e_{xy'}$ are not in an excursion of~$h$ starting at~$x$. Then 
	\begin{align} \label{e_crucial_lemma_logic_statement_3}
	\tilde h (e_{xy}) = \tilde h(e_{xy'}),
	\end{align}
	where~$\tilde h$ denotes the dual of the graph homomorphism~$h$ (cf.~ Definition~\ref{d_dual_graph_homomorphism}). 
\end{lemma}
\begin{proof}[Proof of Lemma~\ref{p_crucial_lemma_logic_statement_3}]
	We observe that if an edge is not in an excursion then it is contained in a path to a boundary point of~$\mathcal{T}$. Because~$h \in \mathcal{H}^{\mathbf{g},free}_n(s)$ is supported on one two-sided geodesic there can only be one such path. We also observe that $x$ is a local minimum for $h$. Hence, $\tilde h (e_{x,y})$ has to increase the depth of $|h(x)|$. This means that one has to move forward on the geodesic $g$ and therefore there is only one choice left for $\tilde h (e_{xy})$. This verifies~\eqref{e_crucial_lemma_logic_statement_3}.
\end{proof}

\begin{definition}[Excursion resampling]\label{d_excursion_resampling}
	Let $h:\mathbb{Z}^m \rightarrow \tree$ be a graph homomorphism and $\mathcal{C}$ be an excursion of $h$ starting at~$x$. We call the following operation \emph{resampling the excursion $\mathcal{C}$}:
	\begin{enumerate}
		\item Define $j$ to be the common index of all $h(e_{xy})$ such that $y$ is a neighbor of $x$ and $y \in \mathcal C$, and define $i_0$ to be index of the unique edge $\alpha_{i_0}$ such that $|h(x)\alpha_{i_0}| = |h(x)|-1$.
		\item Choose an index $i \in \{1,..,d\} \backslash \{i_0\}$ randomly with equal probability.
		\item We define a new dual graph homomorphism~$\tilde g$ by: 
		\begin{align}
		\tilde g(e_{xy}):= \begin{cases}
		\alpha_i & \mbox{if $e_{xy}$ is in $\mathcal C$ and $h(e_{xy})=\alpha_j$,}\\
		\alpha_j & \mbox{if $e_{xy}$ is in $\mathcal C$ and $h(e_{xy})=\alpha_i$,}\\
		\tilde h(e_{xy}) & \mbox{else}.
		\end{cases}
		\end{align}
		\item We set the new graph homomorphism~$h$ to be the graph homomorphism that is naturally associated to~$\tilde g$. 
	\end{enumerate}
\end{definition}
For an illustration of resampling an excursion we refer to Figure~\ref{f_before_resampling} and Figure~\ref{f_after_resampling}. One could ask why one does not choose the index~$i$ uniformly at random out of the set~$\{1,..,d\}$. The reason for choosing the index~$i$ out of the set~$\{1,..,d\} \backslash \{ i_0 \}$ is that by this procedure one guarantees that resampling an excursion does not change the depth profile of the configuration~$h$.\\

\begin{figure}[h]
	\caption{Illustration of an excursion (see Definition~\ref{d_excursion}). The edges in red are on the boundary of the excursion. See also Figure~\ref{f_before_resampling}.}\label{f_excursion}
	\centering
	\includegraphics[width=0.3\textwidth]{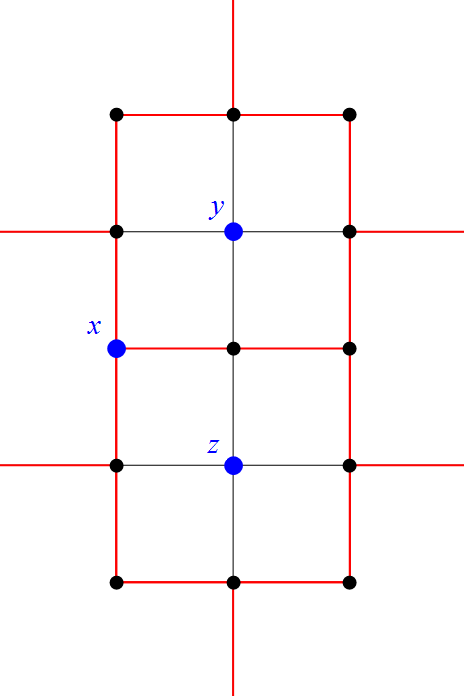}
\end{figure}

\begin{figure}[h]
	\caption{Excursion before resampling}\label{f_before_resampling}
	\centering
	\includegraphics[width=0.5\textwidth]{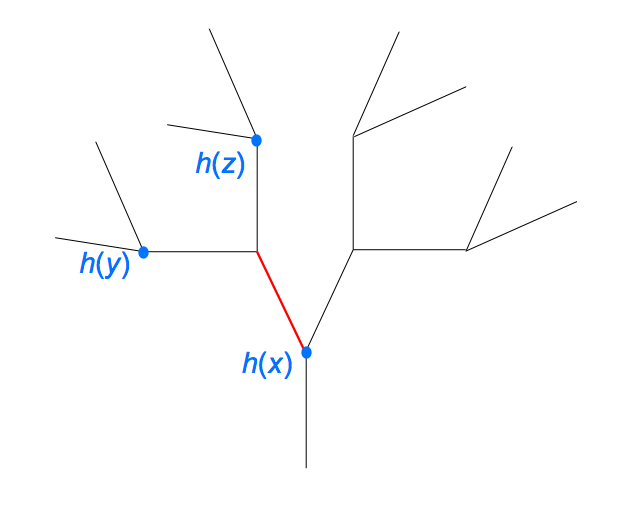}
\end{figure}

\begin{figure}[h]
	\caption{Excursion after resampling}\label{f_after_resampling}
	\centering
	\includegraphics[width=0.5\textwidth]{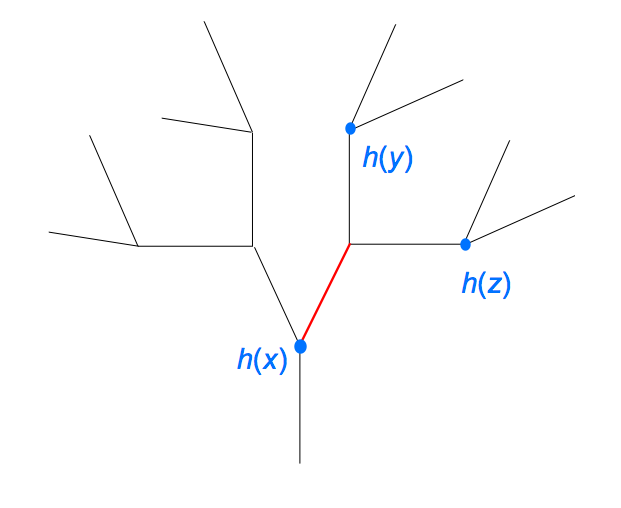}
\end{figure}

For our adapted Glauber dynamics it is important to decide if an edge~$e_{xy}$ of a fake local minimum at~$x$ is in an excursion~$\mathcal{C}$ or not. The next lemma helps a lot in that task.

\begin{lemma}\label{p_path_excursion}
	Let $x \in \Lambda$ be a fake local minimum for the homomorphism $h$ and let $e_{xy}$ be an edge starting from $x$. The edge $e_{xy}$ is not in an excursion~$\mathcal{C}$ that starts at~$x$ if and only if there exist an infinite path $\mathbf{p}=\{ x_0=x, x_1=y,...\}$ and such that:
	\begin{enumerate}
		\item For all $i \geq 1: |h(x_i)| > |h(x)|$
		\item $\sup_{i \geq 1}|h(x_i)|=\infty$
	\end{enumerate}
	Moreover, two edges $e_{xy}$ and $e_{xy'}$ are in a common excursion of~$h$ starting at~$x$ if both edges are in an excursion and there exist a path $\mathbf{p}=\{x_0=x,..,x_l=x\}$ whose first and last edges are $e_{xy}$ and $e_{y'x}$ and such that for all $1 < i < l: |h(x_i)| > |h(x)|$.
\end{lemma}

\begin{proof}[Proof of Lemma~\ref{p_path_excursion}]
	Suppose that $e_{xy}$ is not in an excursion, and consider the connected component $\mathcal{C}$ of $h^{-1}(\tree_0)$ which contains $y$ where $\tree_0$ is the complete tree with root $h(y)$. By definition $e_{xy}$ is not in an excursion if and only the depth is not bounded from above on $\mathcal{C}$. Consider a sequence $\{x_1,..,x_l,..\}$ in $\mathcal{C}^{\mathbb{N}}$ such that for all $i \geq \mathbb{N}: |h(x_i)|=|h(x)|+1$, and build a path $\mathbf{p}$ that goes through all the $x_i$'s while staying in $\mathcal{C}$ (this is always possible since $\mathcal{C}$ is connected), then $\mathbf{p}$ verifies the conditions of the lemma.
	Reciprocally, if $e_{xy}$ is in an excursion for $h$ then the depth is bounded from above on $\mathcal{C}$ and any path $\mathbf{p}$ on which the depth is not bounded from above must leave $\mathcal{C}$ at some point.
\end{proof}

The resampling-step will be used in our construction of the adapted Glauber dynamics. It is important to show that resampling maps an element from~$\mathcal{H}^{\mathbf{g},free}_n(s)$ onto an element in~$\mathcal{H}^{\mathbf{g},free}_n(s)$. This is a direct consequence of the following statement. 

\begin{lemma} \label{invariant_excursion}
	Assume that~$h \in \mathcal{H}^{\mathbf{g},free}_n(s)$. Then it holds that an edge $e_{xy}$ is in an excursion if an only if for all $1 \leq k \leq m$ the edge $e_{(x+n\vec{i}_k)(y+n\vec{i}_k})$ is also in an excursion. Additionally, it holds that $e_{xy}$ and $e_{x'y'}$ are on the boundary of a common excursion if and only if for all $1 \leq k \leq m$ the edges $e_{(x+n\vec{i}_k)(y+n\vec{i}_k)}$ and $e_{(x'+n\vec{i}_k)(y'+n\vec{i}_k)}$ are also one the boundary of a common excursion.
\end{lemma}
The proof of Lemma~\ref{invariant_excursion} is based on Lemma~\ref{p_path_excursion}.
\begin{proof}[Proof of Lemma~\ref{invariant_excursion}]
	Both claims follow from the fact that translating by $n\vec{i}_k$ leaves the depth difference invariant since the two-sided geodesic $\mathbf{g}$ which support the configuration is left unchanged by translation.
	
	Hence, if there exist an infinite path $\mathbf{p}=\{ x_0=x, x_1=y,...\}$ such that for all $i \geq 1: |h(x_i)| > |h(x)|$ and $\sup_{i \geq 1}|h(x_i)|=\infty$. Then the path $\mathbf{p}+n\vec{i}_k= \{ x_0=x+n\vec{i}_k, x_1=y+n\vec{i}_k,...\}$ also verifies for all $i \geq 1: |h(x_i)| > |h(x)|$ and $\sup_{i \geq 1}|h(x_i)|=\infty$.

	Similarly two edges $e_{xy}$ and $e_{x'y'}$ are on the boundary of a common excursion if there exist a path $\mathbf{p}=\{x_0=x,x_1=y,..,x_{l-1}=y',x_l=x \}$ such that for all $1 < i < l: |h(x_i)| > |h(x)|$ which is true if and only if the path $\mathbf{p}+n\vec{i}_k= \{ x_0 =x+n\vec{i}_k ,x_1=y+n\vec{i}_k ,..,x_{l-1}=y'+n\vec{i}_k,x_l=x+n\vec{i}_k\}$ also verifies for all $1<i <l : |h(x_i+n\vec{i}_k)| > |h(x+n\vec{i}_k)|$.
\end{proof}

We are now almost ready to define our adapted Glauber dynamics. Before we proceed we need to describe how to go from a Markov chain on $\mathcal{H}^{\mathbf{g},free}_n(s)$ to a Markov chain on $\mathcal{H}^\mathbf{g}_n(s)$: A Markov chain on the space~$\mathcal{H}^{\mathbf{g},free}_n(s)$ induces a Markov chain on $\mathcal{H}^\mathbf{g}_n(s)$ using the natural operation of geodesic shift which simply pin $\Pi_{\mathbf{g}}(h(0))$ at the right place on the geodesic. In the rest of Section~4 we assume wlog.~that the geodesic $\mathbf{g}$ is the only two-sided geodesic going through all the vertices $(\alpha_1\alpha_2)^k$ for $k \in \Z$. Since all geodesics are isomorphic, it is sufficient to prove the concentration inequality \eqref{e_main_concentration} for this specific choice of $\mathbf{g}$.

\begin{definition}[geodesic shifts]\label{d_gedosesic_shifts}
	Let $h \in \mathcal{H}^{\mathbf{g},free}_n(s)$ we define the shift $\theta_{\mathbf{g}}:\mathcal{H}^{\mathbf{g},free}_n(s) \to \mathcal{H}^\mathbf{g}_n(s)$ by
	\begin{align}
	\theta_{\mathbf{g}}(h)=(\alpha_1\alpha_2)^{-\lfloor \frac{|\Pi_{\mathbf{g}}(h(0))|}{2} \rfloor}h.
	\end{align}
\end{definition}

\begin{remark}
	An easy way to describe the shift $\theta_{\mathbf{g}}(h)$ of $h \in \mathcal{H}^{\mathbf{g},free}_n(s)$ is to say that it is the only homomorphism in $\mathcal{H}^\mathbf{g}_n(s)$ such that $\widetilde{\theta_{\mathbf{g}}(h)}=\tilde{h}$.
\end{remark}
The geodesic shift $\theta_{\mathbf{g}}$ leaves the gradient homomorphism $\tilde{h}$ associated to an homomorphism $h$ invariant. Hence for all $x \in \Z^m: |\theta_{\mathbf{g}}(h)(x)|_0=|h(x)|_0$ since the depth difference between two points is a gradient measurable quantity. 

Now, let us describe the adapted Glauber dynamics on $\mathcal{H}^{\mathbf{g},free}_n(s)$ that is used in the proof of Lemma~\ref{p_azuma_hypo_verification}.

\begin{definition}[Adapted Glauber dynamics]\label{d_adapted_glauber}
	Let $x \in S_n$ and let $\mathbf{p}_x=\{x_0=0,...,x_{|x|_{\ell_1}}=x\}$ be the canonical geodesic path between $0$ and $x$ defined in \eqref{e_canonical_path}. Let $(c_1,..,c_k) \in \Z^k$ and denote by $\mathcal{H}^{\mathbf{g},free}_n(s)[c_1,..,c_k]$ the set of homomorphisms $h \in \mathcal{H}^{\mathbf{g},free}_n(s)$ such that $|h(x_i)|-|h(0)|=c_i$ for $1 \leq i \leq k$. We consider the following dynamics on $\mathcal{H}^{\mathbf{g},free}_n(s)[c_1,..,c_k]$. Let~$X_t \in \mathcal{H}^{\mathbf{g},free}_n(s)[c_1,..,c_k]$. Then the new configuration~$X_{t+1} \in \mathcal{H}^{\mathbf{g},free}_n(s)[c_1,..,c_k]$ is obtained in the following way:
	\begin{enumerate}
		\item Choose a vertex~$y \in S_n$.
		\item If~$|X_t(y)|_0$ is not fixed then resample all the excursions~$\mathcal{C}$ of~$X_t$ that start in $y$. If~$y$ becomes an extremum, pivot $y$ after that and resample again all the excursions~$\mathcal{C}$ of~$X_t$ that start in $x$. 
		\item If $|X_t(y)|_0$ is fixed (that is $y=x_i$ for $i\leq k$ and $|X_t(y)|_0=c_i$) then only resample all the excursions~$\mathcal{C}$ of~$X_t$ that start in $y$ but do not pivot $y$.
	\end{enumerate}
\end{definition}

\begin{remark}
	Since the adapted Glauber dynamics only depends on gradient measurable events and the geodesic shift $\theta_{\mathbf{g}}$ leaves the gradient invariant, the Markov chain $\{ X_t \}_{t \in \mathbb{N}}$ on $\mathcal{H}^{\mathbf{g},free}_n(s)[c_1,..,c_k]$ given by Definition~\ref{d_adapted_glauber} induces a Markov chain $\{ \theta_{\mathbf{g}}(X_t) \}_{t \in \mathbb{N}}$ on the state space~$\mathcal{H}^\mathbf{g}_n(s)[c_1,..,c_k]$ of homomorphisms $h \in \mathcal{H}^\mathbf{g}_n(s)$ such that $$|h(x_i)|-|h(0)|=c_i$$ for $1 \leq i \leq k$. 
\end{remark}
In the next lemma, we show that the adapted Glauber dynamics converges to the correct law once it is projected on $\mathcal{H}^\mathbf{g}_n(s)$. 

\begin{lemma} \label{Markov_irreducibility}
	Let $|s|_{\ell_1} < 1$, let $\{ X_t \}_{t \in \mathbb{N}}$ be the Markov chain given by Definition~\ref{d_adapted_glauber} and let $\{ \theta_{\mathbf{g}}(X_t) \}_{t \in \mathbb{N}}$ be the Markov chain obtained by applying the geodesic shift described in Definition \ref{d_gedosesic_shifts} to $\{ X_t \}_{t \in \mathbb{N}}$. 
	The Markov chain $\{ \theta_{\mathbf{g}}(X_t) \}_{t \in \mathbb{N}}$ is reversible and irreducible on the state space~$\mathcal{H}^\mathbf{g}_n(s)[c_1,..,c_k]$. As a direct consequence, the law of $\{ \theta_{\mathbf{g}}(X_t) \}_{t \in \mathbb{N}}$ converges to the uniform measure on $\mathcal{H}^\mathbf{g}_n(s)[c_1,..,c_k]$.
\end{lemma}

\begin{proof}[Proof of Lemma~\ref{Markov_irreducibility}]
	We start with observing that the adapted Glauber dynamics leaves $\mathcal{H}^{\mathbf{g},free}_n(s)[c_1,..,c_k]$ invariant. Resampling excursions around a vertex~$x$ does not change the depth of~$x$ or $0$ and the dynamics only pivot vertices whose depth is not fixed.\smallskip
	
	We now show that the Markov chain $\{ \theta_{\mathbf{g}}(X_t) \}_{t \in \mathbb{N}}$ is reversible wrt.~the uniform probability measure on~$\mathcal{H}^\mathbf{g}_n(s)[c_1,..,c_k]$. Because we consider the uniform probability measure on~$\mathcal{H}^\mathbf{g}_n(s)[c_1,..,c_k]$ it suffices to show that
	\begin{align}
	\mathbb{P} \left[\theta_{\mathbf{g}}(X_t),\theta_{\mathbf{g}}(X_{t+1}) \right] =\mathbb{P} \left[\theta_{\mathbf{g}}(X_{t+1}),\theta(X_t) \right].
	\end{align}
	Let~$Z_t \in S_n$ denote the position that is chosen in the first step of the adapted Glauber dynamics. We notice that after conditioning on~$Z_t$ both the law given by resampling an excursion and pivoting are uniform over the spaces of reachable configurations in $ \mathcal{H}^{\mathbf{g},free}_n(s)[c_1,..,c_k]$. Thus we have that 
	\begin{align}
	\mathbb{P} \left[\theta_{\mathbf{g}}(X_t),\theta_{\mathbf{g}}(X_{t+1}) \right] & = \mathbb{P} \left[X_t,X_{t+1} \right] \\
	& = \sum_{x \in S_n} \mathbb{P} \left[X_t,X_{t+1} | Z_t=x \right] \frac{1}{n^2} \\
	& = \sum_{x \in S_n} \mathbb{P} \left[X_{t+1},X_t | Z_{t+1}=x \right] \frac{1}{n^2} \\
	& = \mathbb{P} \left[X_{t+1}, X_t \right] = \mathbb{P} \left[\theta_{\mathbf{g}}(X_{t+1}),\theta_{\mathbf{g}}(X_{t}) \right] .
	\end{align}
	This shows that the projected adapted Glauber dynamics is reversible on the state space~$\mathcal{H}^\mathbf{g}_n(s)[c_1,..,c_k]$.\\
	
	Now, we prove that the chain is irreducible. To show the irreducibility, one notices that in the first step there is a positive probability that the adapted Glauber dynamics transforms~$h \in \mathcal{H}^{\mathbf{g},free}_n(s)[c_1,..,c_k]$ after finitely many steps to a configuration~$g \in \mathcal{H}^{\mathbf{g},free}_n(s)[c_1,..,c_k]$ that is entirely supported on the two-sided geodesic~$\mathbf{g}$ and thus the same holds for $\theta_{\mathbf{g}}(g)$. Indeed, one can consider a sequence of moves which does not pivot any point but at each step fold one of the excursion of $h$ to the geodesic $\mathbf{g}$ using the excursion resampling. Since a given configuration only have a finite number of excursions, it takes finitely many move for this configuration to be completely folded on $\mathbf{g}$ using this process. 
	
	Now the main idea is the following. In every step, there is a positive probability that the resampling excursions does not change the configuration~$h \in \mathcal{H}^{\mathbf{g},free}_n(s)[c_1,..,c_k]$. Therefore, the Markov chain $\{\theta_{\mathbf{g}}(X_t)\}_{n \in \N}$ is irreducible if the original Glauber dynamics is irreducible on the space of configurations which are entirely supported on $\mathbf{g}$ in $\mathcal{H}^\mathbf{g}_n(s)[c_1,..,c_k]$. This means that the question of irreducibility of the adapted Glauber dynamics on~$\mathcal{H}^\mathbf{g}_n(s)[c_1,..,c_k]$ has been reduced to the question of the original Glauber dynamics being irreducible for graph homomorphisms taking value in~$\mathbb{Z}$, which we show now. \newline
	We define the distance between $h$ and $h'$ in $\mathcal{H}^\mathbf{g}_n(s)[c_1,..,c_k]$ that are entirely supported on $\mathbf{g}$ by $$d(h,h')=\Sigma_{x \in S_n}||h(x)|-|h'(x)||.$$
	It is clear that $d(h,h')=0$ implies $h$ equal $h'$ and that $d(h,h')$ must be finite since $S_n$ is finite. Hence, if we show that there exists always a pivot move which decreases the distance between two configurations, we can use inductive argument to prove that the chain is irreducible.
	Let $\mathcal{C}$ be the cluster in $S_n$ on which the function $||h(\cdot)|-|h'(\cdot)||$ is maximal and let $x_0$ be a point in $\mathcal{C}$. Suppose without restriction than $h$ is greater than $h'$ on $\mathcal{C}$ and define $\mathbf{p}= \{x_0,...,x_q \}$ to be a maximal increasing path for the depth of $h$ starting from $x_0$. That is a path of maximal length such that for all $0\leq i < q: |h(x_{i+1})|>|h(x_i)|$. This path has to be finite otherwise there exist two points $x_i$ and $x_j$ along the path $\mathbf{p}$ which are $n$-translates of each other and the slope $s$ verifies $|s|_{\ell_1}=1$. All points on the path have to be in $\mathcal{C}$ since for all $0\leq i <q$:
	$$h'(x_{i+1})-h'(x_i) \leq h(x_{i+1})-h(x_i)$$
	Moreover, the fact that $\mathbf{p}$ is of maximal length impose that $h(x_q)$ is a local maximum for the depth.
	Hence it is possible to do a pivoting move at $x_q$ which decrease the depth of $x_q$ by $2$ and decrease the distance between $h$ and $h'$.\smallskip
	
	Because the Markov chain has a finite state space and is reversible, it follows from standard theory of Markov processes that the law of the chain converges to the unique invariant probability measure (see for example~\cite{durrett}) on $\mathcal{H}^\mathbf{g}_n(s)[c_1,..,c_k]$.
\end{proof}

Before proceeding to the proof Lemma~\ref{p_azuma_hypo_verification} let us provide some auxiliary statements.

\begin{lemma}\label{p_excursion_for_minima_auxiliary}
	Let~$h_{up}, h_{dn} \in \mathcal{H}^{\mathbf{g},free}_n(s)$ such that
	\begin{align}\label{e_closeness_h_tilde_h}
	\max_{y \in \mathbb{Z}^m} \left| |h_{up}(y)| - | h_{dn}(y)| \right| \leq 2. 
	\end{align}
	Let~$x \in \mathbb{Z}^m$ be a local minimum for both $h_{up}$ and $h_{dn}$ such that
	\begin{align}
	|h_{up}(x)| = |h_{dn}(x)| + 2. 
	\end{align}
	Then it holds that: 
	\begin{align}\label{e_crucial_lemma_logic_statement_1}
	& \mbox{If an edge $e_{xy}$ is not in an excursion of $h_{up}$ starting at~$x$,} \\
	&\qquad \mbox{then $e_{xy}$ is also not in an excursion of $h_{dn}$ starting at~$x$.}
	\end{align}
	Additionally, assume that the edges~$E_x:=\left\{ e_{xy_1}, \ldots, e_{xy_k} \right\}$ are the same excursion for~$h_{up}$ around~$x$. Then it holds that: 
	\begin{align}
	& \mbox{Either all the edges in $E_x$ are in a } \label{e_crucial_lemma_logic_statement_2} \\
	& \quad \mbox{common excursion for $h_{dn}$ starting at~$x$}\\
	& \qquad \mbox{or no edge in $E_x$ is in an excursion for~$h_{dn}$ starting at~$x$.} 
	\end{align}
\end{lemma}

\begin{proof}[Proof of Lemma~\ref{p_excursion_for_minima_auxiliary}]
	
	Argument for~\eqref{e_crucial_lemma_logic_statement_1}: Let $e_{xy}$ be not in an excursion for $h_{up}$ that starts in~$x$. We know from Lemma~\ref{p_path_excursion} that there exists an infinite path $\mathbf{p}=\{x_0=x,x_1=y,...\}$ such that~$ |h_{up}(x_i)| > |h_{up}(x)|$ for all $i \in \mathbb{N}$. Since $|h_{up}(x)| = |h_{dn}(x)|+2$ and 
	$$\max_{x \in \Z^m}\left||h_{up}(x)| - |h_{dn}(x)|\right|=2,$$
	this also imposes that for all $|h_{dn}(x_i)| > |h_{dn}(x)|$ for all~$i \in \mathbb{N}$. This means that~$e_{xy}$ cannot be in an excursion for~$h_{dn}$.\\
	
	Argument for~\eqref{e_crucial_lemma_logic_statement_2}: We show the statement for only two edges~$e_{xy_1}$ and~$e_{xy_2}$. The generalization to arbitrary many edges~$\left\{ e_{xy_1}, \ldots e_{xy_k} \right\}$ is straightforward. It follows from Lemma ~\ref{p_path_excursion} that if two edges $e_{xy}$ and $e_{xy'}$ are in the same excursion for $h_{up}$ then there exists a path $\mathbf{p}=\{x_0=x,x_1=y,...,x_{l-1}=y',x_l=x\}$ such that $|h_{up}(x_i)| > |h_{up}(x)|$ for all~$1 \leq i \leq l-1$. For the same reason as in the previous paragraph, this imposes that $|h_{dn}(x_i)| > |h_{dn}(x)|$ for all~$1 \leq i \leq l-1 $. This yields that either $e_{xy}$ and $e_{xy'}$ are in the same excursion for $h_{dn}$ or both are not in an excursion for $h_{dn}$. 
\end{proof}

The following auxiliary statement is a simple consequence of Lemma~\ref{p_excursion_for_minima_auxiliary}.
\begin{lemma}\label{p_excursion_for_minima_auxiliary_2}
	Consider the same hypothesis as in Lemma~\ref{p_excursion_for_minima_auxiliary}. Let $\{\mathcal{C}_i\}_{i \in I}$ be the set of excursions of a graph homomorphism $h \in \mathcal{H}^{\mathbf{g},free}_n(s)$ starting at $x$. For each different excursion $\mathcal{C}_i$, choose an edge $e_{xz_i}$ in $\mathcal{C}_i$ starting at $x$ to be a representative edge of $\mathcal{C}_i$ and denote by $\mathcal{E}(h,x)$ the set containing those edges. It follows that $|\mathcal{E}(h,x)|$ is exactly equal to the number of different excursions around $x$ for the homomorphism $h$ (see Figure~\ref{f_excursion} for an example) and that
	\begin{align}\label{e_relation_number_of_resample_able_edges}
	&\left| \mathcal{E}(h_{dn},x)\right| \leq \left|\mathcal{E}(h_{up},x)\right|.
	\end{align}
\end{lemma}

\begin{proof}[Proof of Lemma~\ref{p_excursion_for_minima_auxiliary_2}]
	By the contraposition of the statement \eqref{e_crucial_lemma_logic_statement_1}, the edges $e_{xz} \in \mathcal{E}(h_{dn},x)$ are in an excursion of $h_{up}$ starting at $x$. By the contraposition of the statement \eqref{e_crucial_lemma_logic_statement_2}, it follows that any two edges $e_{xz}, e_{xz'} \in \mathcal{E}(h_{dn},x)$, $e_{xz} \neq e_{xz'}$, are in different excursions of $h_{up}$. This implies the desired estimate \eqref{e_relation_number_of_resample_able_edges}.
\end{proof}

\begin{lemma}\label{p_excursion_for_minima_auxiliary_3}
	Let~$h \in \mathcal{H}^{\mathbf{g},free}_n(s)$ such that~$x \in \mathbb{Z}^m$ is a local minimum of~$h$. Let~$L_x h \in \mathcal{H}^{\mathbf{g},free}_n(s) $ denote the state one obtains after the resampling the excursions of~$h_{up}$ and~$h_{dn}$ starting at~$x$ (see Definition~\ref{d_excursion_resampling}) and let $P(h,x)$ denote the probability that~$L_x h$ has a true local minimum at~$x$. Then:
	\begin{align}\label{e_formula_P_h_x}
	P(h,x) = 
	\begin{cases}
	\frac{1}{(d-1)^{|\mathcal{E}(h,x)|-1}}, & \parbox[t]{.6\textwidth}{if all edges around $x$ are in an excursion of~$h$ starting at~$x$,}\\
	\frac{1}{(d-1)^{|\mathcal{E}(h,x)|}}, & \parbox[t]{.6\textwidth}{if there is an edge around~$x$ that is not in an excursion of~$h$ starting at~$x$.}
	\end{cases}
	\end{align}
	Here~$|\mathcal{E}(h,x)|$ denotes the number of different excursions around $x$ for the homomorphism $h$.
\end{lemma}
\begin{proof}[Proof of Lemma~\ref{p_excursion_for_minima_auxiliary_3}]
	Let us first consider the case in which all edges around $x$ are in an excursion of~$h$ starting at~$x$. The resampling step means that each excursion of~$h$ starting at~$x$ will be attached to an uniformly chosen direction that increases the depth (see Definition~\ref{d_excursion_resampling} and 
	Figure \ref{f_before_resampling} and Figure \ref{f_after_resampling}). Note that in a $d$-regular tree there are $d-1$ many such directions. We recall that $P(h,x)$ denotes the probability that $h(x)$ becomes a true local minimum after resampling the excursions starting at~$x$ (see Definition \ref{d_extrema_graph_homo}). To get a local minimum all excursions must head into the same direction. Hence, for the first excursion one can choose any direction, but all other excursions must head into the same direction, which yields the desired formula
	\begin{align}
	P(h,x) = \frac{1}{(d-1)^{|\mathcal{E}(h,x)|-1}}.
	\end{align}
	Now, let us consider the second case in which there is an edge around $x$ that is not in an excursion of~$h$. It follows from the statement \eqref{e_crucial_lemma_logic_statement_3} that the graph homomorphism $h$ heads in the same direction for all the edges around $x$ that are not in an excursion. Hence, in order for $h$ to become a true local minimum after resampling the excursions around $x$, all the excursions have to head into the same direction. This yields the desired formula
	\begin{align}
	P(h,x) = \frac{1}{(d-1)^{|\mathcal{E}(h,x)|}}.
	\end{align}
\end{proof}

The following auxiliary statement is the main ingredient of the proof of Lemma~\ref{p_azuma_hypo_verification}.
\begin{lemma}\label{p_excursion_for_minima}
	Let~$h_{up}, h_{dn} \in \mathcal{H}^{\mathbf{g},free}_n(s)$ such that
	\begin{align}\label{e_closeness_h_tilde_h}
	\max_{y \in \mathbb{Z}^m} \left| |h_{up}(y)| - | h_{dn}(y)| \right| \leq 2. 
	\end{align}
	Let~$x \in \mathbb{Z}^m$ be a local minimum for both $h_{up}$ and $h_{dn}$ such that
	\begin{align}
	|h_{up}(x)| = |h_{dn}(x)| + 2. 
	\end{align}
	Let~$L_x h \in \mathcal{H}^{\mathbf{g},free}_n(s) $ denote the state one obtains after resampling the excursions of~$h_{up}$ and~$h_{dn}$ starting at~$x$ (see Definition~\ref{d_excursion_resampling}). Then it holds that
	\begin{align}\label{e_excursion_for_minima}
	P(h_{up},x) \leq P (h_{dn}, x),
	\end{align}
	where~$P(h_{up},x)$ and~$P (h_{dn}, x)$ denote the probability that~$L_x h_{up}$ and~$L_x h_{dn}$ respectively have a true local minimum at~$x$.
\end{lemma}

The proof of Lemma~\ref{p_excursion_for_minima} consists out of a combination of Lemma~\ref{p_excursion_for_minima_auxiliary}, Lemma~\ref{p_excursion_for_minima_auxiliary_2} and Lemma~\ref{p_excursion_for_minima_auxiliary_3}.

\begin{proof}[Proof of Lemma~\ref{p_excursion_for_minima}]
	We verify the desired estimate \eqref{e_excursion_for_minima} by considering several cases. In the first case, let us assume that there is an edge around $x$ that is not in an excursion for~$h_{up}$ starting at~$x$. It follows from the statement~\eqref{e_crucial_lemma_logic_statement_1} that there is also an edge that is not in an excursion for~$h_{dn}$ starting at $x$. Hence, it follows from a combination of~\eqref{e_relation_number_of_resample_able_edges} and~\eqref{e_formula_P_h_x} that
	\begin{align}
	P(h_{up}, x) = \frac{1}{(d-1)^{|\mathcal{E}(h_{up},x)|}} \leq \frac{1}{(d-1)^{|\mathcal{E}(h_{dn},x)|}} = P(h_{dn},x).
	\end{align}
	Let us consider the second case in which all edges starting at $x$ for~$h_{up}$ are in an excursion for~$h_{up}$. We make a further distinction and additionally assume that all the edges starting at $x$ are also in an excursion for~$h_{dn}(x)$. In this case, a combination of~\eqref{e_relation_number_of_resample_able_edges} and~\eqref{e_formula_P_h_x} yields that
	\begin{align}
	P(h_{up}, x) = \frac{1}{(d-1)^{|\mathcal{E}(h_{up},x)|-1}} \leq \frac{1}{(d-1)^{|\mathcal{E}(h_{dn},x)|-1}} = P(h_{dn},x).
	\end{align}
	Let us now consider the last case in which we assume that all edges starting at $x$ are in an excursion for~$h_{up}$ but there is an edge starting at $x$ that is not in in excursion for~$h_{dn}$. In this case it we will show that \begin{align}\label{e_comparison_number_of_excursions_last_case}
	|\mathcal{E}(h_{dn},x)|+1 \leq |\mathcal{E}(h_{up},x)|.
	\end{align}
	Postponing the verification of~\eqref{e_comparison_number_of_excursions_last_case} we get by using~\eqref{e_comparison_number_of_excursions_last_case}, ~\eqref{e_relation_number_of_resample_able_edges} and~\eqref{e_formula_P_h_x} that
	\begin{align}
	P(h_{up}, x) = \frac{1}{(d-1)^{|\mathcal{E}(h_{up},x)|-1}} \leq \frac{1}{(d-1)^{|\mathcal{E}(h_{dn},x)|}} \leq P(h_{dn},x),
	\end{align}
	which closes the argument.\\
	
	The only step remaining is to prove~\eqref{e_comparison_number_of_excursions_last_case}.
Denote by $\{\mathcal{C}_i^{up}\}_{1 \leq i \leq k}$ and $\{\mathcal{C}_j^{dn}\}_{1 \leq i \leq l}$ the set of excursions starting at~$x$ of~$h_{up}$ and $h_{dn}$ respectively and for each excursion $\mathcal{C}_j^{dn}$, denote by ~$e_{xz_j}$ its representative edge in $\mathcal{E}(h_{dn},x)$. As a consequence of the statement~\eqref{e_crucial_lemma_logic_statement_2}, we know that the edges~$e_{x z_1}, \ldots,e_{x z_l}$ must be in different excursions of~$h_{up}$. Hence, re-indexing allows us to assume that~$e_{x z_j} \in \mathcal{C}_j^{up} $ for~$j \in \left\{ 1, \ldots, l\right\}$. Recall that we assumed that there exist an edge~$e_{xz_0}$ that is not in an excursion of~$h_{dn}$. Since by assumption all edges near~$x$ are in an excursion of~$h_{up}$ it follows that there is an index~$i_{0} \in \left\{1, \ldots, k \right\}$ such that $e_{x z_0} \in \mathcal{C}_{i_{0}}^{up}$. We will show in a moment that for all $j \in \left\{1, \ldots,l \right\}$
	\begin{align} \label{e_different_excursion}
	\mathcal{C}_{i_{0}}^{up} \neq \mathcal{C}_j^{up}.
	\end{align} 
	This means that the graph homomorphism~$h_{up}$ has at least~$l+1$ many different excursions that start at~$x$, which verifies~\eqref{e_comparison_number_of_excursions_last_case}. \newline
	Let us turn to the verification of~\eqref{e_different_excursion}. We use an indirect argument and assume that wlog. 
	\begin{align} 
	\mathcal{C}_{i_0}^{up} = \mathcal{C}_1^{up}.
	\end{align} 
	Hence, the edge~$e_{xz_1}$ and~$e_{xz_0}$ are in the same excursion~$\mathcal{C}_{l}^{up}$. By the statement~\eqref{e_crucial_lemma_logic_statement_2} this implies that either~$\left\{ e_{xz_1}, e_{xz_0}\right\} \subset \mathcal{C}_{1}^{dn}$ or both edges~$e_{xz_1}$ and~$ e_{xz_0}$ are not in an excursion. This is a contradiction to the fact that by construction
	\begin{align}
	e_{xz_1} \in \mathcal{C}_{1}^{dn} \qquad \mbox{and} \qquad e_{xz_0} \notin \mathcal{C}_{1}^{dn}.
	\end{align}
	
	

	
	
\end{proof}

Now, we are ready to state the proof of Lemma~\ref{p_azuma_hypo_verification}.

\begin{proof}[Proof of Lemma~\ref{p_azuma_hypo_verification}]
	
	We consider a Markov chain~$\{X_t\}_{t \in \N}$ on the state space $$\mathcal{H}^{\mathbf{g},free}_n(s)[c_1,..,c_{k+1}]$$ given by the adapted Glauber dynamics of Definition~\ref{d_adapted_glauber}. Let~$$\{Y_t\}_{t \in \N}$$ denote the Markov chain on the state space~$\mathcal{H}^{\mathbf{g},free}_n(s)[c_1,..,c_{k}]$ that is also given by the adapted Glauber dynamics of Definition~\ref{d_adapted_glauber}. As outlined before, the strategy is to define a coupling~$(X_t, Y_t)$ of those Markov chains such that 
	\begin{align}\label{e_invariance_depth_deviation_under_coupling}
	\max_{y \in \mathbb{Z}^m} \left| |X_t(y)| - |Y_t(y)| \right| \leq 2 \ \Rightarrow \max_{y \in \mathbb{Z}^m}\left| |X_{t+1}(y)| - |Y_{t+1}(y)| \right| \leq 2.
	\end{align}
	We postpone the verification of~\eqref{e_invariance_depth_deviation_under_coupling} and show how it is used to derive the statement of Lemma~\ref{p_azuma_hypo_verification}. We pick~$h \in ~\mathcal{H}^{\mathbf{g},free}_n(s)[c_1,..,c_{k+1}]$ arbitrary and set~$X_0 = h$. Then by the Kirszbraun theorem (cf.~Theorem~\ref{Kirszbraun}) there exists an element~$\bar h \in ~\mathcal{H}^{\mathbf{g},free}_n(s)[c_1,..,c_{k}]$ such that
	\begin{align}
	\max_{y\in \mathbb{Z}^m} \left| |h(y)| - |\bar{h}(y)|\right| \leq 2.
	\end{align}
	Hence, if we set~$Y_0= \bar{h}$ and we get from~\eqref{e_invariance_depth_deviation_under_coupling} that for any realization of the Markov chain and all~$t \in \mathbb{N}$
	\begin{align}
	\max_{y \in \mathbb{Z}^m} \left||X_t(y)| - |Y_t(y)| \right| \leq 2.
	\end{align} 
	The last estimate implies 
	\begin{align}
	\left| \mathbb{E} \left[ |X_t (x)|_0\right] - \mathbb{E} \left[ |Y_t (x)|_0\right] \right| & \leq \mathbb{E} \left[ \left| |X_t (x)| - |Y_t (x)| \right|+\left| |X_t (0)| - |Y_t (0)| \right| \right] \leq 4.
	\end{align}
	Now Lemma \ref{Markov_irreducibility} yields that 
	\begin{align}
	\lim_{t \to \infty} \mathbb{E} \left[ |X_t(x)|_0 \right] = \lim_{t \to \infty} \mathbb{E} \left[ |\theta_{\mathbf{g}}(X_t(x))|_0 \right] = \mathbb{E}_{n}^{s}\left[ |h(x)|_0 \ | \mathcal{F}_{k+1} \right]
	\end{align}
	and
	\begin{align}
	\lim_{t \to \infty} \mathbb{E} \left[ |Y_t(x)|_0 \right] = \lim_{t \to \infty} \mathbb{E} \left[ |\theta_{\mathbf{g}}(X_t(x))|_0 \right] = \mathbb{E}_{n}^{s}\left[ |h(x)|_0 \ | \mathcal{F}_{k} \right].
	\end{align}
	Hence, we overall get the desired estimate~\eqref{e_azuma_hoeffding_crucial_ingredient}
	\begin{align}
	\left| \mathbb{E}_{n}^{(s_1, \ldots, s_m)}\left[ |h(x)|_0 \ | \mathcal{F}_{k+1} \right] - \mathbb{E}_{n}^{(s_1, \ldots, s_m)}\left[ |h(x)|_0 \ | \mathcal{F}_{k} \right] \right| \leq 4. 
	\end{align}\mbox{}\\

	The only thing left to show is that such a coupling~$(X_t,Y_t)$ indeed exists. Let us consider an element~$y \in \mathbb{Z}^m$ such that
	\begin{align}
	|X_t(y)| = |Y_t(y)|.
	\end{align}
	Then every coupling of the chain~$X_t$ and~$Y_t$ works because the adapted Glauber dynamics can only increase the depth deviation in one time step by~2.\\ 
	
	Therefore, let us now consider an element~$y \in \mathbb{Z}^m$ such that \begin{align}
	\left| |X_t(y)| - |Y_t(y)| \right| =2.
	\end{align}
	W.l.o.g~we assume that (else we just interchange the role of~$X_t$ and~$Y_t$ in the argument)
	\begin{align}
	|X_t(y)| \geq |Y_t(y)| +2.
	\end{align}
	In this situation, if a coupling pivots the same points for the configurations $X_t$ and $Y_t$, there is only one scenario in which the original Glauber dynamics would increases the depth deviation to 4, more precisely, such that
	\begin{align}
	\left| |X_{t+1}(y)| - |Y_{t+1}(y)| \right| =4.
	\end{align}
	The scenario is when~$y$ is a true local minimum for~$X_t$ and a fake local minimum for~$Y_t$ (recall that fake local maxima cannot exist for tree-valued height functions, cf.~Remark~\ref{r_fake_local_maxima_do_not_exist}). The depth deviation could now increase if the Glauber dynamics selects the site~$y$. Then the Glauber dynamics could pivot~$X_t(y)$ but not~$Y_t(y)$, which could possibly increase the depth deviation to 4. However, using the excursion resampling we can show that there is a coupling that prevents this scenario from happening. \\
	
	Let us explain this strategy in more details. First of all, we want to mention that resampling an excursion does not change the depth of a configuration. Hence, the additional resampling steps of our adapted Glauber dynamics cannot create a violation of the desired conclusion~\eqref{e_invariance_depth_deviation_under_coupling}. 
	Down below, we construct a coupling such that
	\begin{align}\label{e_coupling_crucial_property}
	& \left\{ \mbox{$y$ is a true local minima of~$L_y X_t$} \right\} \\ 
	& \quad \Rightarrow \left\{ \mbox{$y$ is a true local minimum of~$L_y Y_t$} \right\},
	\end{align}
	where~$L_y X_t \in \mathcal{H}^{\mathbf{g},free}_n(s)$ and~$L_y Y_t \in \mathcal{H}^{\mathbf{g},free}_n(s)$ denote the states obtained after resampling the excursions starting at the vertex~$y$ for the homomorphisms~$X_t$ and~$Y_t$ respectively. The next step of our Markov chain~$(X_t, Y_t)$ is pivoting~$L_y X_t$ and~$L_y Y_t$ around~$y$. Now, this step can easily be coupled such that if~$|L_y X_t (y)|$ increases or decreases, then so does~$|L_y Y_t (y)|$. This implies the desired conclusion~\eqref{e_invariance_depth_deviation_under_coupling}.\\

	We explain now how to construct a coupling that satisfies the statement~\eqref{e_coupling_crucial_property}. The auxiliary Lemma~\ref{p_excursion_for_minima} from below states that 
	\begin{align}\label{e_crucial_property_for_existence_of_coupling}
	P(X_t,y) \leq P(Y_t,y),
	\end{align} 
	where $P(X_t,y)$ is the probability that $y$ is a true minimum of $X_t$ after resampling of the excursions starting at $x$. The quantity~$P(Y_t,y)$ is defined analogously. It is a direct consequence of~\eqref{e_crucial_property_for_existence_of_coupling} is that there is a coupling of the resampling step~$L_y$ such that~\eqref{e_coupling_crucial_property} is satisfied: One throws a random variable~$U$ that is uniformly distributed on~$[0,1]$. If~$U \leq P(X_t, y)$ one decides that both~$L_y X_t$ and~$L_y Y_t$ will have a true local minimum around~$y$, and $L_y X_t$ and~$L_y Y_t$ are chosen uniformly among those states. If~$P(X_t, y) \leq U \leq P(Y_t, y)$ one decides that only~$L_y Y_t$ will have a true local minimum around $y$, but not~$L_y X_t$. And finally if~$P(Y_t, y) \leq U \leq 1$, one decides that both $L_y X_l$ and~$L_y Y_t$ will not have a true local minimum around~$y$. This completes the argument.
\end{proof}

\section{Existence of a continuum of shift-invariant ergodic gradient Gibbs measures}\label{s_existence_gradient_gibbs}

This section is independent of the variational principle (cf.~Theorem~\ref{p_main_result_variational_principle} and Theorem~\ref{p_profile_theorem}) and of its own interest. The Kirszbraun theorem and the concentration estimate obtained in Section~\ref{s_technical_observations} carry important information about the set $ex\mathcal{G}(\Z^m,\T_d)$ of gradient Gibbs measures that are ergodic wrt.~the translations of $\Z^m$. In this section, we show the existence of a continuum of translation-invariant, ergodic, gradient Gibbs measures. In order to make our statement precise, we begin with recalling some classical results about Gibbs measures for discrete systems. They can all be found in \cite{Georgii}.


\begin{definition}[Slope]\label{d_slope}
Let $\nu$ be a gradient Gibbs measure that is ergodic wrt.~translations of $\Z^m$. Then for all $1 \leq k \leq m$ the limit
$$s_i(\nu) = \lim_{n \rightarrow \infty} \frac{1}{n} d_{\mathcal{T}}(h(0),h(n\vec{i}_k))$$
exists $\nu$-almost surely and we call $(s_1(\nu),..,s_m(\nu))$ the slope of $\nu$.
\end{definition}
It is a classical result (e.g \cite{chandgotia2015entropyminimal}) that the subadditive ergodic theorem implies the existence of this limit. Moreover, every translation invariant gradient Gibbs measure can be decomposed into a mixture of ergodic gradient Gibbs measures. The latter allows to define the slope of a translation invariant Gibbs measure in the following way:

\begin{definition}\label{d_decompostion_shift_invariant_Gibbs_measures}
Let $\mu$ be a translation invariant gradient Gibbs measure. Then the slope~$s(\mu)=(s_1(\mu),..,s_m(\mu))$ of $\mu$ is
\[
s_i(\mu)= \int_{ex\mathcal{G}(\Z^m,\T_d)}s_i(\nu)w_{\mu}(d\nu),
\]
where~$w_{\mu}$ is the ergodic decomposition of~$\mu$. This means that for any test function~$f$ it holds that
\begin{align}
 \int f (x) \mu (dx) = \int_{ex\mathcal{G}(\Z^m,\T_d)} \int f(x) \nu (dx) w_{\mu}(d \nu).
\end{align}

\end{definition}

We will now show how the concentration inequality of Theorem~\ref{main_concentration_section_3} implies the existence of an ergodic gradient Gibbs measure for each slope~$s$ whose $\ell_1$-norm is strictly less than $1$. Before we are able to prove the main theorem of this section we provide some auxiliary statements.

\begin{lemma}
	Let $s \in (\mathbb{R}^+)^m$ be such that $|s|_1 < 1$. Then there exists a sequence of $n$-translation invariant homomorphisms $\{h_n \}_{n \in \mathbb{N}}$ with slope $\left( \frac{\lfloor s_1n \rfloor}{n},..,\frac{\lfloor s_mn \rfloor}{n} \right)$ (cf.~Definition~\ref{d_translational_invariant_graph_homomorphism}).
\end{lemma}

\begin{proof}
	The proof is a direct consequence of the Kirszbraun theorem. Let~$\mathbf{g}$ be a two-sided geodesic. For all $(k_1,..,k_m) \in \mathbb{Z}^m$, set $h(k_1n,..,k_mn)= \mathbf{g}(k_1\lfloor s_1n \rfloor +..+ k_m\lfloor s_mn \rfloor)$. Since $\geo$ is isomorphic to $\Z$ we can apply the Kirszbraun theorem between $\Z^m$ and $\geo$. Hence, there exist an extension of $h$ on the whole space $\Z^m$ which is entirely supported on the geodesic $\geo$. The slope of this extension must be $\left( \frac{\lfloor s_1n \rfloor}{n},..,\frac{\lfloor s_mn \rfloor}{n} \right)$ which concludes our proof.
\end{proof}

\begin{lemma}\label{p_auxiliary_ergodic}
Let $s \in (\mathbb{R}^+)^m$ be such that $0 <|s|_{\ell_1} < 1$ and let $\varepsilon > 0$ be small enough. Then there exist universal constants $C$ and $c$ independent of $n$ such that for all $n \in \N$
\begin{align}\label{e_last}
\mathbb{P}^s_n(|d_{\mathcal{T}}(h(0),h(x))-s \cdot x| \geq \varepsilon x) \leq C e^{-c\varepsilon^4 |x|_{\ell_1}}.
\end{align}
\end{lemma}

\begin{proof}
 Wlog.~we assume that $0 < \varepsilon < s_1$. We start with deducing the auxiliary estimate
\begin{align}\label{e_concentration_origin}
\mathbb{P}_n^{s} \left( d_{\mathcal{T}}(h(0),\mathbf{g})) \geq \varepsilon^2 |x|_{\ell_1} \right) \leq C e^{-c\varepsilon^4 |x|_{\ell_1}}
\end{align}
for some universal constants $c, C>0$.\newline
Indeed, let us recall the concentration inequality from Lemma \ref{p_azuma_distance}
\begin{align}
\mathbb{P}_n^{s} \left( \left| |h(x)|_0-\mathbb{E}[|h(x)|_0] \right| \geq \varepsilon |x|_{\ell_1} \right) \leq C e^{-c\varepsilon^2 |x|_{\ell_1}},
\end{align}
and also that (see the proof of Theorem \ref{main_concentration_section_3}) $$\mathbb{E}[|h(x)|_0]= s \cdot x.$$
By combining the two previous observations and summing over all possible $x \in S_n$ such that $|x|_{\ell_1} \geq k$ we obtain that
\begin{align}\label{e_sup_concentration}
\mathbb{P}_n^{s} \left( \sup_{x \in S_n: |x| \geq k} \left| |h(x)|_0- s \cdot x \right| \geq \varepsilon |x|_{\ell_1} \right) & \leq \sum_{|x| \geq k} C e^{-c\varepsilon^2 |x|_{\ell_1}}. \\
& \leq C e^{-c\varepsilon^2 k}. 
\end{align}
Let us now bound the probability that $h(0)$ is far from the geodesic $\mathbf{g}$. Suppose that for $x \in S_n$
 \[
 d_{\mathcal{T}}(h(0),\mathbf{g}) > \varepsilon |x|_{\ell_1}.
 \]
 Recall by convention of the set $\mathcal{H}^\mathbf{g}_n(s)$ that $\Pi_{\mathbf{g}}(h(0))=\mathbf{g}(0)$ or $\Pi_{\mathbf{g}}(h(0))=\mathbf{g}(1)$. Walking along the $x_1$-axis, we denote by $k$ the first integer such that $h(k\vec{i}_1)=\mathbf{g}(0)$ or $h(k\vec{i}_1)=\mathbf{g}(1)$. Notice that we must necessarily have $k \geq \varepsilon |x|_{\ell_1}$. Moreover, by convention we have~$|\mathbf{g}(0)|=0$. Therefore, we get
 $$||h(k\vec{i}_1)|_0 - s_1 k | \geq |h(0)|+s_1 k -1 \geq s_1 \varepsilon |x|_{\ell_1} \geq \varepsilon^2 |x|_{\ell_1}.$$
This implies by using \eqref{e_sup_concentration} the desired estimate~\eqref{e_concentration_origin}.\\
 
Let us now turn to the verification of ~\eqref{e_last}. Using the triangle inequality yields that
\[
|h(x)|_0 \leq d_{\mathcal{T}}(h(0),h(x)) \leq |h(x)|_0+ 2d_{\mathcal{T}}(h(0),\mathbf{g}).
\]
Hence
\begin{eqnarray}
\mathbb{P}_n^{s} \left( |d_{\mathcal{T}}(h(0),h(x))-s \cdot x| \geq \varepsilon |x|_{\ell_1} \right) & \leq & \mathbb{P}^s_n(||h(x)_0|-s \cdot x| \geq \frac{\varepsilon}{2} |x|_{\ell_1}) +\\
& & \mathbb{P}^s_n(|d_{\mathcal{T}}(h(0),\mathbf{g}))| \geq \frac{\varepsilon}{4} |x|_{\ell_1}).
\end{eqnarray}
since for the left-hand event to happen one of the two right-hand events must happen. Using \eqref{e_sup_concentration} and \eqref{e_concentration_origin} together this implies that 
\[
\mathbb{P}_n^{s} \left( |d_{\mathcal{T}}(h(0),h(x))-s \cdot x| \geq \varepsilon |x|_{\ell_1} \right) \leq C e^{-c\varepsilon^4 |x|_{\ell_1}}
\]
for some universal constants $C,c>0$ independent of $n$. This completes the argument.
\end{proof}

Let us now prove the existence of a continuum of ergodic gradient Gibbs measures.

\begin{theorem} \label{ergodic}
	For all $s \in (\mathbb{R}^+)^m$ such that $0 < |s|_{\ell_1} < 1$, there exists an ergodic gradient Gibbs measure $\mu$ with slope $s$.
\end{theorem}

\begin{proof}[Proof of Theorem~\ref{ergodic}]

	We already know from the previous lemma that for all $n \in \mathbb{N}$ the set $\mathcal{H}^{\mathbf{g}}_n \left( \frac{\lfloor s_1n \rfloor}{n},..,\frac{\lfloor s_mn \rfloor}{n} \right)$ is non empty. Hence, the uniform probability measure~$\mu_n(s)$ on~$\mathcal{H}^{\mathbf{g}}_n \left( \frac{\lfloor s_1n \rfloor}{n},..,\frac{\lfloor s_mn \rfloor}{n} \right)$ exists. By compactness of the space of gradient Gibbs measures in the topology of local convergence, we can extract a subsequence $\mu_{n_k}(s)$
	which converges to a gradient Gibbs measure~$\mu$ on the infinite volume space~$\Z^m$ (see Lemma~8.2.7 of~\cite{She05} for more details). The shift invariance of the measure~$\mu$ follows from the shift invariance of the spaces $\mathcal{H}^{\mathbf{g}}_n \left( \frac{\lfloor s_1n \rfloor}{n},..,\frac{\lfloor s_mn \rfloor}{n} \right)$. This means that we can write the measure~$\mu$ as a mixture of ergodic Gibbs measures (see Definition~\ref{d_decompostion_shift_invariant_Gibbs_measures}). Suppose, by contradiction that the slope is not almost surely equal to $s$ on the average $w_{\mu}$. Then there exist $i \in \{1..m \}$, $\delta$ and $\varepsilon$ such that $w_{\mu}(|s_i(\nu)-s_i|\geq \varepsilon)\geq \delta$, wlog. we will assume that $i=1$.
	Applying Lemma~\ref{p_auxiliary_ergodic} we obtain that there exist universal constants $C$ and $c$ independent of $n_k$ such that:
	\begin{align}
	\mu_{n_k} \left( |d_{\tree}\left(h(0),h(l\vec{i}_1)\right)-ls_1| \geq \varepsilon l \right) \leq C e^{-c\varepsilon^4 l}.
 	\end{align}
	Dividing by $l$ we can rewrite this inequality as
	$$\mu_{n_k}\left( \left| \frac{1}{l}d_{\tree}\left(h(0),h(l\vec{i}_1))\right)-s_1 \right| \geq \varepsilon \right) \leq C e^{-c\varepsilon^4 l}.$$ Thus, by going to the subsequential limit first in $k$ and then in $l$ it holds that
	$$\lim_{l \rightarrow \infty} \mu \left( \left| \frac{1}{l}d_{\tree}\left(h(0),h(l\vec{i}_1)\right)- s_1 \right| \geq \varepsilon \right) =0. $$
	On the other hand, we have that
	\begin{align}
	&\lim_{l \rightarrow \infty} \mu \left( \left| \frac{1}{l}d_{\tree}\left(h(0),h(l\vec{i}_1)\right)- s_1 \right| \geq \varepsilon \right) \\
	& = \lim_{l \rightarrow \infty} \int_{ex\mathcal{G}(\Z^m,\T_d)} \nu\left( \left| \frac{1}{l}d_{\tree}\left(h(0),h(l\vec{i}_1)\right)- s_1 \right| \geq \varepsilon \right)w_{\mu}(d\nu) \\
	& \geq \lim_{l \rightarrow \infty} \int_{ \{|s_i(\nu)-s_i|\geq \varepsilon \} } \nu\left( \left| \frac{1}{l}d_{\tree}\left(h(0),h(l\vec{i}_1)\right)- s_1 \right| \geq \varepsilon \right)w_{\mu}(d\nu) \\
	& \geq \int_{\{|s_i(\nu)-s_i|\geq \varepsilon \}} \lim_{l \rightarrow \infty} \nu\left( \left| \frac{1}{l}d_{\tree}\left(h(0),h(l\vec{i}_1)\right)- s_1 \right| \geq \varepsilon \right)w_{\mu}(d\nu) \\
	& \geq \delta,
	\end{align}
	which is a contradiction. Therefore, the slope is $w_{\mu}$-almost surely equal to $(s_1,..,s_m)$ and there exists at least one ergodic gradient Gibbs measure with slope $(s_1,..,s_m)$.
\end{proof}

%
%

\section{Proof of the variational principle}\label{s_proof_variational_principle}

In this section we prove the results of Section~\ref{s_main_result}. More precisely, we give the proof of 
Lemma~\ref{p_extension_of_bhf}, Theorem~\ref{p_main_result_variational_principle} and Theorem~\ref{p_profile_theorem}.
We start with the proof of Lemma~\ref{p_extension_of_bhf}. The argument uses a standard construction to extend Lipschitz functions from the boundary~$\partial R$ to the whole set~$R$.

\begin{proof}[Proof of Lemma~\ref{p_extension_of_bhf}]
Let $(f,(a_{i,j})_{k \times k})$ be an asymptotic boundary height profile on $\partial R$ and define $g:R \to \R^+ \times \{1,..,k\}$ by:
\begin{align} \label{e_def_extension}
g^1(y) = & \max \{0, \max_{x\in \partial R}\{f^1(x)-|x-y|_{\ell_1}\} \\
g^2(y) = & f^2( \arg \max_{x\in \partial R}\{f^1(x)- |x-y|_{\ell_1}\}).
\end{align}
We will show that~$g$ extends~$f$ to an asymptotic height profile on the whole region $R$. In order to prove this we show the following three properties:
\begin{itemize}
\item $g=f$ on $\partial R$;
\item $g$ satisfies the condition~\eqref{e_asy_hf_lipschitz};
\item $g$ satisfies the condition~\eqref{e_asyhf_admissible}.
\end{itemize}
The first property is a simple consequence of the inequality~\eqref{e_asy_hf_lipschitz} that holds for the function~$f^1$. Indeed using~\eqref{e_asy_hf_lipschitz} yields that for all $x,y \in \partial R$ we have $\max_{x\in \partial \R}\{f^1(x)-d(x,y)\} \leq f^1(y)$ and thus $g^1(y)=f^1(y)$. \newline
For the second property, we first observe that by a combination of the triangle inequality and the fact that~$f^1$ satisfies~\eqref{e_asy_hf_lipschitz} on~$\partial R$, the function~$\tilde g^1 : R \to \R$ given by
\begin{align*}
\tilde g^1 (y) := \max_{x\in \partial R}\{f^1(x)-|x-y|_{\ell_1}\} 
\end{align*}
satisfies on~$R$ the condition~\eqref{e_asy_hf_lipschitz}. It is now a simple consequence that ~$g^1$ also satisfies~\eqref{e_asy_hf_lipschitz} on~$R$. \newline
Let us turn to the third property. We show the third property by contradiction. Let us assume that the map~$g$ does not satisfy the condition~\eqref{e_asyhf_admissible}. Then there is a point~$x \in \overline{g^{-1}(\mathbb{R}^+,i)} \cap \overline{g^{-1}(\mathbb{R}^+,j)} $ such that
\begin{align}
 g^1 (x) > a_{i,j}.
\end{align}
Then by continuity there are two points~$x_1,x_2 \in R$ such that
\begin{align}
 g_2(x_1) =i , \quad g_2(x_2)=j , \quad g^1(x_1) > a_{i,j}\geq 0, \quad \mbox{and} \quad g^1(x_2) > a_{i,j} \geq 0.
\end{align}
By definition~\eqref{e_def_extension} of the map~$g$ it follows that there is a point~$z_1 \in \partial R$ and~$z_2 \in \partial R$ such that
\begin{align} \label{e_estiamte_extension_contradiction_1}
g^1(x_1) = f^1(z_1) - |x_1 - z_1|_{\ell_1} > a_{i,j} \geq 0
\end{align}
and
\begin{align}
g^1(x_2) = f^1(z_2) - |x_2 - z_2|_{\ell_1} > a_{i,j} \geq 0.
\end{align}
Because we can choose~$|x_1 -x_2|_{\ell_1}$ to be arbitrarily small the last estimate yields that
\begin{align}\label{e_estiamte_extension_contradiction_2}
f^1(z_2) - |x_1 - z_2|_{\ell_1} \geq f^1(z_2) - |x_2 - z_2|_{\ell_1} - |x_1 -x_2|_{\ell_1} > a_{i,j} \geq 0.
\end{align} 
A combination of~\eqref{e_estiamte_extension_contradiction_1} and\eqref{e_estiamte_extension_contradiction_2} yields that 
 \begin{align}
 \left| f^1(z_1) - a_{i,j} \right| + \left| f^1(z_2) - a_{i,j} \right| & > |x_1 - z_1|_{\ell_1} + |x_1 - z_2|_{\ell_1} \\
 & \geq|z_1 - z_2|_{\ell_1}.
 \end{align}
The last estimate contradicts inequality~\eqref{e_extenability_asym_bhf} and completes the argument.
\end{proof}

Let us turn to the verification of Theorem~\ref{p_main_result_variational_principle} and Theorem~\ref{p_profile_theorem}. For the general strategy of the argument we refer to the discussion at the end of Section~\ref{s_main_result}. First, we turn to the proof of Theorem~\ref{p_profile_theorem}. As explained at the end of Section~\ref{s_main_result} the main ingredient is to show the equivalence of the entropy of fixed and fluctuating boundary conditions on a simplex. This ingredient is provided in Lemma~\ref{p_entropy_depth_condition_local_surface_tension} below. Before turning to Lemma~\ref{p_entropy_depth_condition_local_surface_tension}, we show in the next auxiliary lemma that if the boundary height function on a simplex is close to an affine function then it must stay close to a single geodesic. Similar to Remark~\ref{r_geodesic_effect_negligible} of Section~\ref{s_main_result}, this also justifies that the entropic effect, that results from the additional freedom of choosing the geodesic~$\mathbf{g}$, is of lower order.

\begin{definition}
A geodesic segment~$\mathbf{s}$ is given by an injective graph homomorphism~$\mathbf{s}: \left\{0, 1, \ldots k \right\} \to \mathcal {T}$ (cf.~Definition~\ref{d_geodesic}). 
\end{definition}

\begin{lemma}\label{p_entropic_effect}
We consider the simplex 
	\[
	\Delta_n= \{x \in \Z^d : 0 \leq x_1 \leq .. \leq x_m \leq n\}.
	\]
		Let $\delta > 0$, let $c > 0$ and let $s \in \R^m$ such that $|s|_{\ell_1} \leq 1 $. Let $h: \partial \Delta_n \to \mathcal{T}$ be a graph homomorphism such that for all $x \in \partial \Delta_n $:
	\begin{align}\label{e_condition_overestimating_local_entropy}
	| d_{\mathcal{T}} (h(x),\mathbf{r} ) - (c+s \cdot x) | \leq \delta n.
	\end{align}
	Then there exist a geodesic segment $\mathbf{s}$ such that for all $x \in \partial \Delta_n $:
	\begin{align}\label{e_reduction_of_depth_bc_to_geodesic_bc}
	d_{\tree}(h(x), \mathbf{s} \left(\lfloor c + s \cdot x \rfloor) \right)\leq 4\delta n,
	\end{align}
 where~$ \mathbf{s} (\lfloor c+s \cdot x \rfloor)$ denotes the unique element~$v \in \mathbf{s}$ such that~$$ d_{\mathcal{T}} (v, \mathbf{r}) = \lfloor c+ s \cdot x \rfloor.$$
\end{lemma}

\begin{proof}[Proof of Lemma~\ref{p_entropic_effect}]

We start by proving the following auxiliary statement: Let $h$ be a graph homomorphism satisfying the conditions of Lemma~\ref{p_entropic_effect}. If $\mathbf{p}=\{x_0,..,x_l\}$ is a path in $\partial \Delta_n$ on which the function $x \mapsto s \cdot x$ is non-decreasing then there exist a geodesic segment $\mathbf{s}\subset \mathcal{T}$ such that for all~$x \in \mathbf{p}$:
\begin{align}\label{e_reduction_of_depth_bc_to_geodesic_bc_vp}
d_{\tree}(h(x), \mathbf{s} (\lfloor c+s \cdot x \rfloor) \leq 2\delta n.
\end{align}

Argument for~\eqref{e_reduction_of_depth_bc_to_geodesic_bc_vp}: Let $\tilde{\mathbf{s}}$ denote the geodesic segment of~$h$ between the vertices $h(x_0)$ and $h(x_l)$. We then define $\mathbf{s}$ to be a geodesic segment which goes through the root $\mathbf{r}$, the vertex with lowest depth in $\mathbf{s}$ and $h(x_l)$. This construction insure that the depth is strictly increasing along $\mathbf{s}$. We suppose that there exist $x_k \in \mathbf{p}$ such that
$$d_{\tree}(h(x_k),\mathbf{s}(\lfloor c+ s \cdot x_k \rfloor) > 2\delta n.$$
	By definition of a geodesic segment in a tree, the path~$\left\{h(x_k) , \ldots ,h( x_l) \right\}$ has to return to~$\mathbf{s}$. Let ~$l \geq \tilde k > k$ denote the smallest number such that $h \left( x_{\tilde k} \right) \in \mathbf{s}$. Then it holds that $|h(x_{\tilde k})| < |h(x_k)| -2\delta n$. A direct calculation yields (cf.~the proof of Theorem \ref{main_concentration_section_3}) that
\begin{align}
2 \delta n & < |h(x_k)| - |h(x_{\tilde k})| \\
 & = |h(x_k)| - (c+s \cdot x_k) + (c+s \cdot x_k) \\
 & \qquad - (c+s \cdot x_{\tilde k}) + (c+s \cdot x_{\tilde k}) - |h(x_{\tilde k})| \\ 
& \leq \left||h(x_k)| - (c+s \cdot x_k) \right| + \left| (c+s \cdot x_{\tilde k}) - |h(x_{\tilde k})| \right|.
\end{align}
Hence, it follows that either
$$||h(x_{\tilde k})|- \left( c+s \cdot x_{\tilde k} \right)| > \delta n$$
or 
$$||h(x_k)|- \left( c+s \cdot x_k \right) |> \delta n, $$
which is a contradiction with our starting hypothesis and completes the proof of \eqref{e_reduction_of_depth_bc_to_geodesic_bc_vp}.\\

Let us now turn to the main argument. We distinguish the cases $m \neq 2$ and $m = 2$.

If $m \neq 2$ then for all $k \in \Z$ the sets $S_k= \{ x \in \partial \Delta_n : -k+1 \leq s \cdot x \leq k+1 \}$ are connected as subsets of the graph $\partial \Delta_n$. As a consequence we obtain that $h(S_k)$ is also a connected subset of the tree $\mathcal{T}$. Thus for all $x,y \in S_k$:
\begin{align}
 |d_{\mathcal{T}} (h(x),h(y) ) |\leq 2\delta n.
 \end{align} 
Moreover, both the maximum and the minimum of the function $x \mapsto s \cdot x$ on $\partial \Delta_n$ must be reached on some corner (non-necessarily unique) of the simplex $\Delta_n$. Consider two such corners $x_{\min}$ and $x_{\max}$ which reach respectively the minimum and the maximum of $h$ and define $\mathbf{p}$ to be an increasing path on $\partial \Delta_n$ between $x_{\min}$ and $x_{\max}$. We can use~\eqref{e_reduction_of_depth_bc_to_geodesic_bc_vp} to define a geodesic segment $\mathbf{s}$ such that for all $x \in \mathbf{p}$:
	\begin{align}
d_{\tree}(h(x), \mathbf{s} (\lfloor c + s \cdot x \rfloor) \leq 2\delta n.
\end{align}

 Since each vertex in $\mathbf{p}$ is in $h(S_k)$ for some $k \in \Z$, we must have that for all $x \in \partial \Delta_n: 	d_{\tree}(h(x), \mathbf{s} (\lfloor c + s \cdot x \rfloor)) \leq 4\delta n$. This concludes our proof for $m \neq 2$.\\

Let us now consider the case~$m=2$. The reasoning is similar, there must be two extremal vertices $x_{\min}$ and $x_{\max}$ of the simplex which minimize and maximize the function $x \mapsto s \cdot x$ on $\partial \Delta_n$. Moreover both geodesic paths from $x_{\min}$ to $x_{\max}$ are necessarily increasing for the function $$x \mapsto s \cdot x.$$ Hence we obtain that for all~$x \in \partial \Delta_n$:
\begin{align}\label{e_reduction_of_depth_bc_to_geodesic_bc}
d_{\tree}(h(x), \mathbf{s} ( \lfloor c+ s \cdot x \rfloor ) \leq 2\delta n,
\end{align}
where $\mathbf{s}$ is a geodesic segment defined as in the first part of this proof.
\end{proof}

We will now show that the entropy per site on a simplex~$\Delta_n$ for fluctuating linear boundary conditions is equivalent to the local surface tension as~$n \to \infty$.

\begin{lemma}\label{p_entropy_depth_condition_local_surface_tension}
	Under the same assumptions as in Lemma~\ref{p_entropy_depth_condition_local_surface_tension}, let us denote with~$M_{fluc}(\partial \Delta_n,s,\delta n)$ the set of graph homomorphisms $h: \Delta_n \rightarrow \tree$
	such that for all $x \in \partial \Delta_n $
		\begin{align}\label{e_condition_overestimating_local_entropy}
		| d_{\mathcal{T}} (h_{\partial \Delta_n}(x),\mathbf{r} ) -( c+s \cdot x) | \leq \delta n.
		\end{align}
	Then 
	\begin{align}
	- \lim_{n \to \infty} \frac{m!}{n^m} \ln |M_{fluc}(\partial \Delta_n,s,\delta n)|= \ent(s)+ \theta \left( \frac{1}{n}\right) +\theta (\delta),
	\end{align}
	where~$\ent(s)$ is the local surface tension given by Definition~\ref{e_def_micro_surface_tension}.
\end{lemma}

\begin{remark}
	Since the volume of the simplex $\Delta_n$ is $\frac{n^m}{m!}$, the quantity $ -\frac{m!}{n^m} \ln |M_{fluc}(\partial \Delta_n,s,\delta n)|$ can be understood as the entropy per site associated to fluctuating boundary conditions on $\partial \Delta_n$ with slope $s$. 
\end{remark}

We already deduced an analogous result for hypercubes in Theorem~\ref{geodesic_square}. Hence, the main idea of the proof of Lemma~\ref{p_entropy_depth_condition_local_surface_tension} consists of leveraging the result of Lemma~\ref{p_entropy_depth_condition_local_surface_tension} to simplices by a two-scale decomposition of a simplex into smaller cubes, and vice versa.

\begin{proof}[Proof of Lemma~\ref{p_entropy_depth_condition_local_surface_tension}]
 Let $h \in M_{fluc}(\partial \Delta_n,s,\delta n)$. By Lemma \ref{p_entropic_effect}, there exist a geodesic segment $\mathbf{s}\subset \mathcal{T}$ such that for all~$x \in \partial \Delta_n$:
	\begin{align}\label{e_reduction_of_depth_bc_to_geodesic_bc}
	d_{\tree}(h(x), \mathbf{s} (\lfloor c+ s \cdot x \rfloor) \leq 2\delta n.
	\end{align}
 We start with showing that the entropic effect of having different geodesic segments~$\mathbf{s}$ is negligible. Therefore, let us estimate how many different geodesic segment~$\mathbf{s}$ can exist. On~$\partial \Delta_n$ there are less than~$d^{n^{m-1}}$ possible distinct boundary graph homomorphisms. Hence, the number of possible distinct geodesic segments is also bounded by~$d^{n^{m-1}}=\theta(n^m)$, which is of lower order.\\

As a consequence let us fix a geodesic segment~$\mathbf{s}$. It is sufficient to show that for all $h_{\partial \Delta_n}$ such that
	\begin{align}\label{e_reduction_of_depth_bc_to_geodesic_bc}
	d_{\tree}(h(x), \mathbf{s} (\lfloor c+s \cdot x \rfloor) \leq 2\delta n,
	\end{align}
the number $|M(\Delta_n,h_{\partial \Delta_n})|$ of homomorphisms which extends the boundary height function $h_{\partial \Delta_n}$ to the whole $\Delta_n$ verifies
		\begin{align}\label{e_estimate_vp}
	- \lim_{n \to \infty} \frac{m!}{n^m} \ln |M(\Delta_n,h_{\partial \Delta_n})| = \ent(s)+ \theta\left( \frac{1}{n}\right) + \theta (\delta).
		\end{align}
 Indeed, since the number of different homomorphisms restricted to $\partial \Delta_n$ is bounded by $Cn^{m-1}= \theta(n^m)$, this directly implies the desired statement of Lemma~\ref{p_entropy_depth_condition_local_surface_tension}.\\

 In order to deduce the estimate~\eqref{e_estimate_vp} we consider the cases $|s|_{\ell_1}<1$ and $|s|_{\ell_1}=1$ separately.\\
	
\begin{figure} \centering
	\begin{subfigure}{0.4\textwidth}
		\begin{tikzpicture}
		\newcommand*\width{3.1}
		\newcommand*\num{20.5}
		\pgfmathparse{\width/\num}
		\let\step\pgfmathresult
		\begin{scope}
		\draw [black,fill=lightgray] (0.2,0.2)--(0.2,\width-0.4)--(\width-0.4,0.2)--cycle;	
		\draw [black,fill=lightgray] (\width-0.2,\width-0.2)--(0.4,\width-0.2)--(\width-0.2,0.4)--cycle;	
		\end{scope}
		\draw[thick] (0,0) -- (\width,0) -- (\width,\width) -- (0,\width) -- cycle;
		\draw[very thin,decorate,decoration=brace]
	 (\width-0.4,-0.05) -- 	node[below=2pt] {$n$} (0.4,-0.05) ;
		\draw[very thin,decorate,decoration=brace]
		(-0.05,0) -- node[left=2pt] {$n+2\frac{\delta n}{1-|s|_{\ell_1}}$} +(0,\width);
		\end{tikzpicture}
	\end{subfigure}\qquad%
	\begin{subfigure}{0.4\textwidth}
		\begin{tikzpicture}
		\newcommand*\width{3.1}
		\newcommand*\num{10.5}
		\pgfmathparse{\width/\num}
		\let\step\pgfmathresult
		\begin{scope}
		\clip (0,0) -- (\width,0) -- (0,\width) -- cycle;
		\fill[lightgray] (0.28,0.28) rectangle (\width,\width);
		\foreach \j in {0,...,\num}
		\fill[white] (\step*\j,{\step*(floor(\num)-\j-1)})
		rectangle +(\step,2*\step);
		\draw[step=\step,very thin] (0,0) grid (\width,\width);
		\end{scope}
		\draw[thick] (0,0) -- (\width,0) -- (0,\width) -- cycle;
		\draw[very thin,decorate,decoration=brace]
		({\step*(floor(\num/2)},-0.05) --
		node[below=2pt] {\makebox[10pt][l]{%
				$ \varepsilon n$}}
		+(-\step,0);
		\draw[very thin,decorate,decoration=brace]
		(-.05,0) -- node[left=2pt] {$n$} +(0,\width);
		\end{tikzpicture}
	\end{subfigure}
		\caption{ Visualization of the gluing procedures employed in the proof of Lemma \ref{p_entropy_depth_condition_local_surface_tension}. On the left, simplices of size $n$ are attached to a square of size $n+2\frac{\delta n}{1-|s|_{\ell_1}}$. On the right, a simplex of size $n$ is tiled with squares of size $\varepsilon n$ and the squares which touch the border strip of size $\delta n$ are removed. The resulting set $\mathcal{S}$ is in light gray.}
			\label{f_tiling}
	\end{figure}
	
In the case~$|s|_{\ell_1}=1$ we observe that by definition~$\ent(s)=0$. Hence, the verification of~\eqref{e_estimate_vp} is reduced to a straightforward combinatorial argument using Stirling formula which is left out in this article (see e.g \cite{CKP01} Lemma 3.5 or \cite{morales2018asymptotics} Lemma 5.1 for an analogous argument). \\

 Argument for the case $|s|_{\ell_1}<1$: In order to verify~\eqref{e_estimate_vp} it suffices to show that
	\begin{align}~\label{e_fixed_vs_free_bc_lower_estiamte}
	\ent_{ n }(s) \leq \Ent (\Delta_n, h_{\partial \Delta_n})+ \theta (\delta)+\theta \left( \frac{1}{n} \right) 
	\end{align}
	and 
	\begin{align}~\label{e_fixed_vs_free_bc_upper_estiamte}
	\ent_{ n }(s) \geq \Ent (\Delta_n, h_{\partial \Delta_n})+ \theta (\delta) +\theta \left( \frac{1}{n} \right) .
	\end{align}
 
	Argument for~\eqref{e_fixed_vs_free_bc_lower_estiamte}: The argument consists of dividing a large simplex in smaller cubes. We underestimate the entropy (i.e.~overestimate the number of configurations) by choosing independent periodic boundary conditions on each cube. This will yield the correct lower bound by Theorem~\ref{geodesic_square}. Let us now give the details of the argument. Consider an $m-$dimensional hypercube $H_{\varepsilon n}$ with edge size $\lfloor \varepsilon n \rfloor$ and lower left corner at the origin. We use now the convention that
	\[
	h^\star_{\partial H_{\varepsilon n}}=\argmax_{h \in \mathcal{H}^{\mathbf{g}}_{\varepsilon n} (s) }|M(H_{\varepsilon n}, h_{\partial H_{\varepsilon n}})|,
	\]
where~$h_{\partial H_{\varepsilon n}}$ is the restriction of~$h: H_{\varepsilon n} \to \mathcal {T}$ to the boundary~$\partial H_{\varepsilon n}$. This means that $h^\star_{\partial H_{\varepsilon n}}$ is the periodic boundary data of slope $s$ on $\partial H_{\varepsilon n}$ which has the most possible extensions to a graph homomorphism $H_{\varepsilon n}$ on among all periodic boundary data of slope $s$ on $\partial H_{\varepsilon n}$.

 By definition of $\mathcal{H}^{\mathbf{g}}_{\varepsilon n} (s)$, it follows that $|\mathcal{H}^{\mathbf{g}}_{\varepsilon n} (s)|$ is bounded by $|M(H_{\varepsilon n}, h^{\star})|$ times the number of possible boundary homomorphisms on $H_{\varepsilon n}$. More precisely, 
	\[
	|\mathcal{H}^{\mathbf{g}}_{\varepsilon n} (s)| \leq d^{2^m(\varepsilon n)^{m-1}}|M(H_{\varepsilon n}, h^\star_{\partial H_{\varepsilon n}})|.
	\]
	Now we fix~$\varepsilon>0$ small enough and define for~$k_1,..,k_m \in \Z$ the translated hypercube
	$$H_{\varepsilon n}^{k_1,..,k_m}=H_{\varepsilon n}+(k_1 \lfloor \varepsilon n \rfloor,..,k_2 \lfloor \varepsilon n \rfloor).$$
	Let us consider the subset $\mathcal{S}$ of $\Z^m$ 
	$$\mathcal{S}=\{(k_1,...,k_m):H_{\varepsilon n}^{k_1,..,k_m} \subset \Delta_n \text{ and } d_{\Z^m}(H_{\varepsilon n}^{k_1,..,k_m},\partial \Delta_n) \geq 2\delta n\}.$$
 The set~$\mathcal{S}$ indexes all hypercubes~$H_{\varepsilon n}^{k_1,..,k_m}$ that are contained in the simplex~$\Delta_n$ and not too close to the boundary (cf.~the right side of Figure~\ref{f_tiling}). The cardinality of all sites contained in those boxes (represented in light grey in the right side of Figure~\ref{f_tiling}) is bounded from below by $$\frac{1}{\varepsilon^m m!}-\frac{2^{m+1}(\delta+2\varepsilon)}{\varepsilon^{m}}. $$
 Here, the second term of the bound comes from the maximum number of cubes that are removed from the strip of size $2\delta n$ along the border of the simplex. Moreover, since the boundary homomorphism $h^\star_{\partial H_{\varepsilon n}}$ is $\lfloor \varepsilon n\rfloor $-periodic, there exist a graph homomorphism $$g: \cup_{k_1,...,k_m\in \mathcal{S}} H_{\varepsilon n}^{k_1,..,k_m} \to \mathcal{T}$$ such that for all $k_1,...,k_m\in \mathcal{S}$:
	\[
	\tilde{g}_{\partial H_{\varepsilon n}^{k_1,..,k_m}}=	\tilde{h}^\star_{\partial H_{\varepsilon n}} \circ \tau_{(k_1 \lfloor \varepsilon n \rfloor,..,k_2 \lfloor \varepsilon n \rfloor)} .
	\]
Here, $\tilde{g}$ and $\tilde{h}^\star$ are the gradient graph homomorphisms associated to $g$ and $h^\star$ respectively, and $\tau_{x}$ is the shift by $x$ in $\Z^m$. Using Lemma~\ref{p_entropic_effect}, we observe that there exist a one sided geodesic $\mathbf{g}$ such that
\[
	\sup_{\partial H_{\varepsilon n}} 	| d_{\mathcal{T}}(g(x), \mathbf{g} ( \lfloor s \cdot x \rfloor ) | \leq \varepsilon n.
\]
We choose wlog.~$\mathbf{g}$ to be the same geodesic for both $g$ and $h_{\partial \Delta_n}$.

	As a consequence, it is possible for $\frac{1-|s|}{\varepsilon} < \delta$ to extend the homomorphisms $g$ to $h_{\partial \Delta_n}$ on $\partial \Delta_n$ (use a similar calculation as in equation \eqref{e_Kirszbraun_estimate} of Lemma \ref{p_comparison_ent_n_ent_m} and the Kirszbraun theorem for graphs). If we combine the previous observations we obtain that:
	\begin{eqnarray}
	\frac{m!}{n^m}\ln |M(\Delta_n, h_{\partial \Delta_n})| & \geq & \frac{m!}{n^m} \left(\frac{1}{\varepsilon^mm!}-\frac{2^{m+1}(\delta+2\varepsilon)}{\varepsilon^{m}}\right) \ln |M (H_{\varepsilon n},h^{\star}_{H_{\varepsilon n}})| \\
	& \geq & \left( 1- 2^{m+1}(\delta+2\varepsilon) \right) \frac{1}{\varepsilon^m n^m}\ln |M (H_{\varepsilon n},h^{\star}_{H_{\varepsilon n}})| \\
	 & \geq & \left(1-\theta(\delta )+\theta (\varepsilon) \right)\ent(s)+\theta \left( \frac{1}{\varepsilon n} \right). 
	\end{eqnarray}
	Sending now first $n \to \infty$ and then $\varepsilon \to 0$ yields the desired estimate~\eqref{e_fixed_vs_free_bc_lower_estiamte}.\\

	Argument for \eqref{e_fixed_vs_free_bc_upper_estiamte}: Let us describe the main strategy. We will overestimate the entropy on a cube by underestimating the possible number of configurations. We observe that on the continuum there is a natural decomposition of a cube into $m!$~simplices. We now leverage this decomposition onto the lattice, fitting $m!$~simplices of size $n$ in a cube of size slightly larger than $n$. We now underestimate the number of configurations by fixing the boundary values on each simplex. Because we leave enough space between the simplices, we can glue together the boundary graph homomorphisms by the Kirszbraun theorem. Now, the entropy per site on each simplex is an upper bound for the entropy per site on the large cube yielding the correct estimate by Theorem~\ref{geodesic_square}.\\
 
 Let us now outline the details of the construction. We can tile the hypercube $H_{n+2\frac{\delta n}{1-|s|_{\ell_1}}}$ of size $n+2\frac{\delta n}{1-|s|_{\ell_1}}$ using $m!$ simplices of size $n$ each within distance $2\frac{\delta n}{1-|s|_{\ell_1}}$ of each other (see the left side of the Figure~\ref{f_tiling}). Let us now explain how to choose the boundary condition on each simplex. The boundary condition will be fixed by a graph-homomorphism~$g \in \mathcal{H}^{\mathbf{g}}_{n+2\frac{\delta n}{1-|s|_{\ell_1}}} (s)$. Let $h: \partial \Delta_n \to \mathcal{T}$ be a graph-homomorphism such that 
	\begin{align}\label{e_condition_overestimating_local_entropy}
	\sup_{x \in \partial {\Delta_n}} | d_{\mathcal{T}} (h(x),\mathbf{r} ) - (c+s \cdot x) | \leq \delta n.
	\end{align}
 We observe that the all simplices in the cube are characterized by $m$-permutations~$\sigma$. For a given permutation~$\sigma$ we define~$\Delta^{\sigma}_n$ as the simplex obtained by permuting the coordinates of~$\Delta_n$ according to~$\sigma$. Let us denote by $h^{\sigma}$ the graph-homomorphisms on ${\partial \Delta^{\sigma}_n}$ obtained by permuting the coordinates components of $h$ according to $\sigma$. We set $g=h^{\sigma}$ on the boundary of each one of those simplices. Moreover, by the Kirszbraun theorem (see also equation~\eqref{e_Kirszbraun_estimate} of Lemma \ref{p_comparison_ent_n_ent_m}) we can extend the homomorphism $g$ to $H_{n+2\frac{\delta n}{1-|s|_{\ell_1}}}$. Now, we underestimate the number of graph homomorphisms in~$\mathcal{H}^{\mathbf{g}}_{n+2\frac{\delta n}{1-|s|_{\ell_1}}} (s)$ by
 \begin{align}
 & |\mathcal{H}^{\mathbf{g}}_{n+2\frac{\delta n}{1-|s|_{\ell_1}}} (s)| \geq \\
 & \left\{h \in \mathcal{H}^{\mathbf{g}}_{n+2\frac{\delta n}{1-|s|_{\ell_1}}} (s) \ | \ h(x)= g(x) \ \mbox{for all $x$ } \right. \\
 & \qquad \left. \mbox{that are not in the interior of a simplex} \right\}.
 \end{align}
 By self-similarity of the simplices this implies that
	\begin{eqnarray}
	 |M(\Delta_n,h_{\partial \Delta_n})|^{m!} \leq |\mathcal{H}^{\mathbf{g}}_{n+2\frac{\delta n}{1-|s|_{\ell_1}}} (s) |
	\end{eqnarray}
	and therefore
	\begin{eqnarray}
	\frac{1}{m!} \ln |M(\Delta_n,h_{\partial \Delta_n})| \leq -\ent(s)+\theta(\delta)+\theta \left( \frac{1}{n} \right),
	\end{eqnarray}
	which verifies~\eqref{e_fixed_vs_free_bc_upper_estiamte}.
\end{proof}

\begin{figure}


 \includegraphics[width=0.6\textwidth]{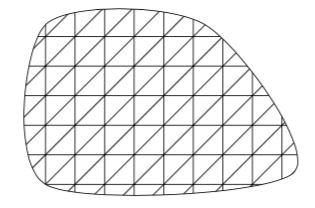}
\caption{Illustration of the set~$R$. The grid is the set~$R_{\tt{grid, \varepsilon}}.$}\label{f_set_R_with_grid}
\end{figure}

\begin{figure}





 \includegraphics[width=0.55\textwidth]{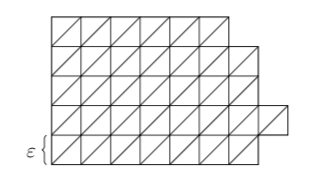}

\caption{Approximation of the set~$R$ in the simplicial complex $\mathcal{K}_{\varepsilon}$. }\label{f_block_approximation_of_R}
\end{figure}

In the proof of Theorem~\ref{p_profile_theorem} we use the observation that $1$-Lipschitz functions can be approximated very well by affine functions on a simplicial complex.

\begin{definition}\label{d_linear_approximation}
	We denote by $\mathcal{K}_{\varepsilon}$ the simplicial lattice of $\R^m$ formed by the points on the boundary of the simplices 
	\[
	\{x \in \R^m: 0 \leq x_{\sigma(1)}-\varepsilon k_1\leq ...x_{\sigma(m)}-\varepsilon k_m \leq \varepsilon\}
	\]
	for any $m$-permutation $\sigma$ and $(k_1,..,k_m) \in \Z^m$.
 Let~$f_{\varepsilon}^1$ denote a function on~$R$ such that~$f_{\varepsilon}^1$ is piecewise affine on the simplicial lattice~$\mathcal{K}_{\varepsilon} \cap R$ and~$ f_{\varepsilon}^1(x) = f^1(x)$ for all vertices~$x \in R$ of the simplicial lattice~$\mathcal{K}_{\varepsilon}\cap R$. The function~$f_{\varepsilon}^1$ is called affine approximation of~$f^1$.
\end{definition}

\begin{remark}
	Using our previous notation in Definition \ref{d_ball_around_h} we have that $\mathcal{K}_{\frac{\lfloor \varepsilon n \rfloor}{n} }=\frac{1}{n}\mathcal{K}^n_{\varepsilon}$. 
\end{remark}

\begin{lemma}[Approximation via triangulation] \label{lem_approx_tri}
Let~$f^1_{\varepsilon}$ be an affine approximation of $f^1$ as defined above. For $\delta > 0$ let $K_{\varepsilon}(R, \delta)$ be the union of all simplices $\Delta \in \mathcal{K}_{\varepsilon}$ strictly included in $R$ and such that:
	\[
	\sup_{x \in \partial \Delta}|f^1_{\varepsilon}(x)-f^1(x)| \leq \frac{\delta}{2} \varepsilon.
	\]
.
	 
Then it holds that for $\varepsilon > 0$ sufficiently small
	\begin{enumerate}
 \item $|R \setminus K_{\varepsilon}(R, \delta)| = \theta(\varepsilon)$, \\[-1ex]
 
		\item $\E(f^1_{\varepsilon})= \E(f^1)+\theta(\epsilon)$.
	\end{enumerate}
\end{lemma}
The statement of Lemma~\ref{lem_approx_tri} is proven in the two dimensional case in Lemma 2.2 and Lemma 2.3 of~\cite{CKP01}. The proof does not involve any argument specific to the $2$-dimensional case and can be naturally generalized to $m$-dimensional simplices. Hence we are not reproving Lemma~\ref{lem_approx_tri} in this article. Let us now turn to the proof of Theorem~\ref{p_profile_theorem}.

\begin{proof}[Proof of Theorem \ref{p_profile_theorem}]
Recall the definition of the sets

\begin{align}\label{e_balls_around_AHP}
& HP_n (f, \varepsilon, \delta) \\
&= \left\{ h_{n} \in M (R_n, h_{\partial R_n}) \ | \ \sup_{x \in \frac{1}{n}\mathcal{K}_n^{ \varepsilon} \cap \frac{1}{n}R_n } \left| \frac{1}{n} d_{\mathcal{T}} (h_n (x), \mathbf{r}) - f^1 \left( \frac{x}{n} \right) \right| \leq \varepsilon \delta \right\}.
\end{align} 

For $\varepsilon >0$ and $n \in \N $, denote $\varepsilon_n= \frac{\lfloor \varepsilon n \rfloor}{n}$. For fixed $\varepsilon, \delta >0$ and given the asymptotic height profile~$(f, a_{i,j})$ let us consider the affine approximation $f^1_{\varepsilon_n}$ of $f^1$ on the simplicial lattice $\mathcal{K}_{\varepsilon_n}$ (see Definition~\ref{d_linear_approximation}). It follows from Lemma~\ref{lem_approx_tri} that $$\E(f^1_{\varepsilon_n})=\ \E(f^1)+\theta(\epsilon)+\theta\left (\frac{1}{n} \right).$$
Hence, Theorem~\ref{p_profile_theorem} is a consequence of the following statement:
\begin{align}\label{e_limit_HP}
\lim_{n \to \infty} \frac{1}{n^m} \ln |HP_n (f, \delta, \varepsilon)| = - \E(f_{\varepsilon_n}
) + \theta_{\varepsilon}(\delta) + \theta(\varepsilon) .
\end{align}

We also know from property (1) of Lemma~\ref{lem_approx_tri} that $$|K_{\varepsilon}(R,\delta)|=|R|+\theta(\varepsilon).$$ 

Let us now start by underestimating the number of homomorphisms in $HP_n (f, \delta, \varepsilon)$.
We begin by proving that $HP_n (f, \delta, \varepsilon)$ is non-empty for $n$ large enough. Define $g_n: R_n \to \mathcal{T}$ to be the graph homomorphism such that for all $x \in R_n$: 
$$g_n(x)= \begin{cases}
\mathbf{g}_{f^2(\frac{x}{n})}(\lfloor f^1 (\frac{x}{n}) \rfloor ) , & \mbox{if~$x$ and $ \mathbf{g}_{f^2(\frac{x}{n})}(\lfloor f^1 (\frac{x}{n}) \rfloor )$ have same parity},\\
\mathbf{g}_{f^2(\frac{x}{n})}(\lceil f^1 (\frac{x}{n}) \rceil ) , & \mbox{if~$x$ and $\mathbf{g}_{f^2(\frac{x}{n})}(\lceil f^1 (\frac{x}{n}) \rceil )$ have same parity}.
\end{cases}
$$
According to the definition of the convergence of a sequence of boundary height functions, for $n$ large enough the following holds:
\[
\sup_{x \in \partial R_n} d_{\mathcal{T}}(g_n(x),h_{\partial R_n}(x)) \leq \frac{\varepsilon \delta}{2} n.
\]
Hence we can apply Corollary \ref{p_corollary_kirszbraun} to the homomorphism $g_n$ and any graph homomorphism which extends $h_{\partial R_n}$ on $R$ to deduce that there exist at least one homomorphism $\bar{h}$ (depending on $n$) defined on $R_n$ in $HP_n (f, \frac{\delta}{2}, \varepsilon)$.

For a given simplex $\Delta \in \mathcal{K}_{\varepsilon_n}$ with $\varepsilon_n= \frac{\lfloor \varepsilon n \rfloor}{n}$, denote by $n\Delta$ the same simplex scaled by $n$ and belonging to $\mathcal{K}^n_\varepsilon$. 
It follows directly from the definition of $K_{\varepsilon_n}(R,\delta)$ that for all simplex $\Delta \in K_{\varepsilon_n}(R,\delta)$:
$$ \sup_{x \in \partial \Delta} |\frac{1}{n}d_{\mathcal{T}}(\bar{h}(x),\mathbf{r})-f^1_{\varepsilon_n} \left( \frac{x}{n} \right)| \leq \varepsilon \delta .$$
As a consequence for all $\Delta \in K_{\varepsilon_n}(R,\delta)$,
the homomorphism $ \bar{h}$ is in $ M_{fluc}(\partial (n\Delta), \nabla f^1_{\varepsilon_n}, \epsilon\delta n )$.

We can now underestimate the number of homomorphisms in $HP_n (f, \delta, \varepsilon)$ by counting all configurations exactly equal to $\bar{h}$ on the boundary of the simplices $n\Delta$ such that $\Delta \in K_{\varepsilon_n}(R,\delta)$. This gives us:

\begin{eqnarray}
\frac{1}{n^m} \ln |HP_n (h, \delta, \varepsilon)| & \geq & \frac{1}{n^m} \ln \prod_{\Delta \in K_{\varepsilon_n}(R,\delta) } \left| M(n\Delta,\bar{h}_{\partial (n\Delta)}) \right| \\
 & \geq & \frac{1}{n^m} \sum_{\Delta \in K_{\varepsilon_n}(R,\delta) } \ln \left| M(\Delta,\bar{h}_{\partial (n\Delta)}) \right| \\
 & \geq & \frac{\varepsilon_n^m}{m!} \sum_{\Delta \in K_{\varepsilon_n}(R,\delta) } \left( \ent(\nabla f^1_{\varepsilon_n})+\theta_{\varepsilon}(\delta)+\theta \left(\frac{1}{\varepsilon n} \right) \right) \\
 & \geq & -\E(f^1_{\varepsilon_n})+\theta_{\varepsilon}(\delta)+\theta \left( \frac{1}{\varepsilon n} \right),
\end{eqnarray}
where in the third inequality we used Lemma~\ref{p_entropy_depth_condition_local_surface_tension} and in the last inequality we used that by definition $$ \frac{\varepsilon_n^m}{m!} \sum_{\Delta \in K_{\varepsilon_n}(R,\delta) } \ent(\nabla f^1_{\varepsilon_n})=-\E(f^1_{\varepsilon_n}).$$ 

Let us now turn to the overestimation of the number of graph homomorphisms in $|HP_n (f, \delta, \varepsilon)|$. Any function which is within $\varepsilon \delta$ of $f^1$ on $K_{\varepsilon_n}(R,\delta )$ must be within $2 \varepsilon \delta$ of $f^1_{\varepsilon_n}$ on $K_{\varepsilon_n}(R,\delta )$. Hence, the strategy is to look at the free product of all configurations $h$ such that $d_{\T}(h,\mathbf{r})$ stays within $2 \varepsilon \delta n$ of $f^1_{\varepsilon_n}$ on each rescaled simplex $n\Delta$ for which $\Delta \in K_{\varepsilon_n}(R,\delta )$. We subsequently multiply this number by $d^{\left( |R|-|K_{\varepsilon_n}(R,\delta)| \right)n^m}$ which is a natural upper bound on the number of different configurations outside of those simplices. This gives an upper bound on the total of possible homomorphisms in $|HP_n (f, \delta, \varepsilon)|$. Using the same notations as in Lemma~\ref{p_entropy_depth_condition_local_surface_tension}, we obtain:
\begin{align}
 \frac{1}{n^m} \ln |HP_n (h_{\varepsilon_n}, \delta, \varepsilon)| \\
 \leq & \frac{\ln d^{\left( | R| - |K_{\varepsilon_n}(R,\delta) | \right)n^m}}{n^m} \prod_{\Delta \in K_{\varepsilon_n}(R,\delta) } |M_{fluc}(\Delta_{\varepsilon n},\nabla f^1_{\varepsilon_n},\delta \varepsilon n)|\\
 \leq & (\ln d) \theta(\varepsilon)+ \frac{1}{n^m} \sum_{\Delta \in K_{\varepsilon_n}(R,\delta) } \ln |M_{fluc}(\Delta_{\varepsilon n},\nabla f^1_{\varepsilon_n},\delta \varepsilon n)| \\
 \leq & \frac{\varepsilon_n^m}{m!}\sum_{\Delta \in K_{\varepsilon_n}(R,\delta) } \left(\ent(\nabla f^1_{\varepsilon_n})+\theta_{\varepsilon}(\delta)+\theta \left( \frac{1}{\varepsilon n} \right) \right)+ \theta(\varepsilon). \\
 \leq & -\E (f^1_{\varepsilon_n})+\theta(\varepsilon)+\theta_{\varepsilon}(\delta)+ \theta \left( \frac{1}{\varepsilon n} \right).
\end{align}

Here we used the estimate of Lemma~\ref{p_entropy_depth_condition_local_surface_tension} to deduce the third inequality. Combining the overestimation and the underestimation yields the desired formula~\eqref{e_limit_HP}.
\end{proof}

We deduce Theorem~\ref{p_main_result_variational_principle} from Theorem~\ref{p_profile_theorem} using a compactness argument similar to the one used in~\cite{CKP01}.

	\begin{proof}[Proof of Theorem~\ref{p_main_result_variational_principle}]
		For a fixed asymptotic boundary height profile~$ \left( f_{\partial R}, \left( a_{i,j}\right)_{k \times k} \right)$ let us consider the set~$AHP(f_{\partial R}, \left( a_{i,j}\right)_{k \times k})$ of all possible extensions to an asymptotic height profile~$ \left( f, \left( a_{i,j}\right)_{k \times k} \right)$ (see~Definition~\ref{d_asymp_height_profile}). We observe that $$AHP(f_{\partial R}, \left( a_{i,j}\right)_{k \times k}) \neq \emptyset.$$ We observe that the space~$AHP(f_{\partial R}, \left( a_{i,j}\right)_{k \times k})$ is closed and that the local surface tension~$\ent$ is convex (cf.~Theorem~\ref{p_convexity_local_surface_tension}) and uniformly bounded from below by~$-d$. Hence, it follows that
		\begin{align}\label{e_inf_is_min}
		\inf_{f\in AHP(f_{\partial R}, \left( a_{i,j}\right)_{k \times k})} \E \left(f^1 \right) & = \min_{f\in AHP(f_{\partial R}, \left( a_{i,j}\right)_{k \times k})} \E \left(f^1 \right) \\
		& = \E \left(f^1_{\min} \right) ,
		\end{align}
		where~$(f^1_{\min}, a_{i,j})$ is a minimizer of the macroscopic entropy~$\E ( f^1)$.\\
		
		Let us fix an~$\eta>0$. We have to show that there exists an integer~$n_0 \in \mathbb{N}$ such that for all~$n \geq n_0$ it holds
		\begin{align}\label{e_variational_principle_upper_bound}
		\Ent \left( R_n, h_{\partial R_n}\right) \leq \E \left(f^1_{\min} \right) + \eta
		\end{align}
		and
		\begin{align}\label{e_variational_principle_lower_bound}
		\Ent \left( R_n, h_{\partial R_n}\right) \geq \E \left(f^1_{\min} \right) - \eta. 
		\end{align}
		\mbox{}\\

		We start with deducing the estimate~\eqref{e_variational_principle_upper_bound}. Underestimating the number of graph homomorphisms and using the identity~\eqref{e_profile_theorem} of Theorem \ref{p_profile_theorem} yields that 
		\begin{align}
		\Ent \left( \Lambda_n, h_{\partial R_n}\right) & = - \frac{1}{|R_n|} \ln | M (R_n, h_{\partial R_n}) | \\
		& \leq - \frac{1}{|R_n|} \ln | HP_n \left( f_{\min}, \delta, \varepsilon \right) | \\
		& = \E \left(f^1_{\min} \right)+ \theta(\varepsilon)+ \theta_{\varepsilon} (\delta) + \theta_{\delta,\varepsilon} \left( \frac{1}{ n} \right). \label{e_proof_variational_principle_upper_bound_calculation_1}
		\end{align} 
		Choosing now first~$\varepsilon>0$ small then ~$\delta>0$ small (depending on $\varepsilon$), and finally~$n$ large (depending on~$\varepsilon$ and $\delta$) yields the desired upper bound~\eqref{e_variational_principle_upper_bound}. \\

		Let us now deduce the lower bound~\eqref{e_variational_principle_lower_bound}. This is equivalent to showing that for all $\eta >0$:
		\begin{align}
		- \lim_{n \to \infty} \frac{1}{|R_n|} \Ent \left( \Lambda_n, h_{\partial R_n}\right) & \geq \E (f^1) - \eta \\
		\end{align}
		Choose $\eta >0$, we know from~\eqref{e_profile_theorem} that all asymptotic profile $f$ there exist $\varepsilon_f$ (depending on $f$ and $\eta$) and $\delta_{f}$ (depending on $h$, $\eta$ and $\varepsilon_f$) such that
		\begin{align}
		- \lim_{n \to \infty} \frac{1}{|R_n|} \ln \left| HP_n \left( f , \delta_{f,\varepsilon_f} , \varepsilon_f \right) \right| & \geq \E (f^1)-\eta.
		\end{align}
		Since the set of $1$-Lipschitz functions defined on a closed bounded region is compact, there exist an integer $l$, depending on~$\eta$ but not on~$n$, and a collection of asymptotic height profiles~$f_1,\ldots, f_l$ such that
		\begin{align}\label{e_variational_principle_lower_bound_1}
		M (R_n, h_{\partial R_n}) \subset \bigcup_{i=1}^l HP_n \left( f_i , \delta_{f_i} ,\varepsilon_{f_i} \right). 
		\end{align}
		By taking the logarithm and using the identity~\eqref{e_profile_theorem} of Theorem \ref{p_profile_theorem} we obtain that
		\begin{align}
		- \lim_{n \to \infty} \frac{1}{|R_n|} \Ent \left( \Lambda_n, h_{\partial R_n}\right) & \geq - \lim_{n \to \infty} \frac{1}{|R_n|} \ln \left|\bigcup_{i=1}^l HP_n \left( f_i , \delta_{f_i} ,\varepsilon_{f_i} \right)\right| \\
                                                                                                     & \geq - \lim_{n \to \infty} \frac{\ln l}{|R_n|} \\
                  & \qquad - \lim_{n \to \infty} \frac{1}{|R_n|} \max_{1 \leq i \leq l} \ln \left| HP_n \left( f_i , \delta_{f_i} ,\varepsilon_{f_i} \right)\right| \\
		& \geq \E \left(f^1_{min}\right) - \eta,
		\end{align}
		which finishes the proof.
	\end{proof}

\begin{center}
 \textsc{Acknowledgement}
\end{center}

The authors want to thank Marek Biskup, Nishant Chandgotia, Filippo Colomo, Nicolas Destainville, Richard Kenyon, Michel Ledoux, Igor Pak, Laurent Saloff-Coste, Scott Sheffield, Peter Winkler and Tianyi Zheng for helpful discussions and comments. The authors also would like to express special thanks of gratitude to the anonymous referees and to Andrew Krieger. Their precise comments and questions helped a lot to further improve the manuscript.

\bibliographystyle{alpha}
\bibliography{bib}

\end{document}